\documentclass{amsart}
\usepackage{fullpage, verbatim, xy, amssymb,url}
\usepackage{amsmath}
\usepackage{amssymb}
\usepackage{amsfonts}
\usepackage{color}

\xyoption{all}
\newtheorem{theorem}{Theorem}[section]

\newtheorem{corollary}[theorem]{Corollary}

\newtheorem{lemma}[theorem]{Lemma}

\newtheorem{proposition}[theorem]{Proposition}

\theoremstyle{remark}
\newtheorem{definition}[theorem]{Definition}
\newtheorem{remark}[theorem]{Remark}

\newtheorem{example}[theorem]{Example}
\input{tcilatex}

\renewcommand{\epsilon}{\varepsilon}

\DeclareMathOperator{\proj}{proj}
\DeclareMathOperator{\Hom}{Hom}
\DeclareMathOperator{\GL}{GL}

\renewcommand{\phi}{\varphi}

\begin{document}
\title{Triple product $p$-adic $L$-functions for balanced weights}
\author{Matthew Greenberg}
\author{Marco Adamo Seveso}

\begin{abstract}
We construct $p$-adic triple product $L$-functions that interpolate (square
roots of) central critical $L$-values in the balanced region. Thus, our
construction complements that of M. Harris and J. Tilouine.

There are four central critical regions for the triple product $L$-functions
and two opposite settings, according to the sign of the functional equation.
In the first case, three of these regions are of interpolation, having
positive sign; they are called the unbalanced regions and one gets three $p$%
-adic $L$-functions, one for each region of interpolation (this is the
Harris-Tilouine setting). In the other setting there is only one region of
interpolation, called the balanced region. An especially interesting feature
of our construction is that we get three different $p$-adic triple product $%
L $-functions with the same (balanced) region of interpolation. To the best
of the authors' knowledge, this is the first case where an interpolation
problem is solved on a single critical region by different $p$-adic $L$%
-functions at the same time. This is possible due to the structure of the
Euler-like factors at $p$ arising in the interpolation formulas, the
vanishing of which are related to the dimensions of certain Nekovar period
spaces. Our triple product $p$-adic $L$-functions arise as specializations
of $p$-adic period integrals interpolating normalizations of the local
archimedean period integrals. The latter encode information about classical
representation theoretic branching laws. The main step in our construction
of $p$-adic period integrals is showing that these branching laws vary in a $%
p$-adic analytic fashion. This relies crucially on the Ash-Stevens theory of
highest weight representations over affinoid algebras.
\end{abstract}

\maketitle
\tableofcontents

\section{Introduction}

Consider three $p$-adic Coleman families $\mathbf{f}=\left( \mathbf{f}_{1},%
\mathbf{f}_{2},\mathbf{f}_{3}\right) $ of tame levels $\left(
N_{1},N_{2},N_{3}\right) $, parametrized by some admissible open $U_{\mathbf{%
f}}\subset \mathcal{X}^{3}$, where $\mathcal{X}$ denotes the weight space.
For an arithmetic point $\underline{k}=\left( k_{1}^{\varepsilon
_{1}},k_{2}^{\varepsilon _{2}},k_{3}^{\varepsilon _{3}}\right) $ of weight $%
\left( k_{1},k_{2},k_{3}\right) $ and nebentype $\left( \varepsilon
_{1},\varepsilon _{2},\varepsilon _{3}\right) $, let us write $\mathbf{f}_{%
\underline{k}}=\left( \mathbf{f}_{1,k_{1}},\mathbf{f}_{2,k_{2}},\mathbf{f}%
_{3,k_{3}}\right) $ for the specialization of $\mathbf{f}$, which is the $p$%
-stabilization of some $\mathbf{f}_{\underline{k}}^{\#}=\left( \mathbf{f}%
_{1,k_{1}}^{\#},\mathbf{f}_{2,k_{2}}^{\#},\mathbf{f}_{3,k_{3}}^{\#}\right) $
for almost every arithmetic point. Hence we have%
\begin{equation*}
\mathbf{f}_{i,k_{i}}^{\#}=f_{i,k_{i}}\in S_{k_{i}+2}\left( \Gamma _{0}\left(
N_{i}\right) ,\varepsilon _{i}\right) ^{N_{i}\mathrm{-new}}\text{.}
\end{equation*}%
The problem we are interested in is about interpolating the function%
\begin{equation*}
\underline{k}:=\left( k_{1},k_{2},k_{3}\right) \mapsto L\left(
f_{k_{1}}\times f_{k_{2}}\times f_{k_{3}},c_{\underline{k}}\right)
\end{equation*}%
for the central critical value $c_{\underline{k}}:=\frac{k_{1}+k_{2}+k_{3}+4%
}{2}$. Here $L\left( f_{k_{1}}\times f_{k_{2}}\times f_{k_{3}},s\right) $ is
the triple product complex $L$-function. Let us write $\pi _{i}:=\pi
_{f_{i,k_{i}}}$ for the automorphic representation attached to $%
f_{i}:=f_{i,k_{i}}$. If we want $\Pi _{\mathbf{f}_{\underline{k}}^{\#}}:=\pi
_{1}\otimes \pi _{2}\otimes \pi _{3}$ to be selfdual, the condition $%
\varepsilon _{1}\varepsilon _{2}\varepsilon _{3}=1$ needs to be imposed.
There are four central critical regions, namely%
\begin{eqnarray*}
&&\Sigma _{1}:=\left\{ \left( k_{1}^{\varepsilon _{1}},k_{2}^{\varepsilon
_{2}},k_{3}^{\varepsilon _{3}}\right) :k_{1}>k_{2}+k_{3},\varepsilon
_{1}\varepsilon _{2}\varepsilon _{3}=1\right\} \text{,} \\
&&\Sigma _{2}:=\left\{ \left( k_{1}^{\varepsilon _{1}},k_{2}^{\varepsilon
_{2}},k_{3}^{\varepsilon _{3}}\right) :k_{2}>k_{1}+k_{3},\varepsilon
_{1}\varepsilon _{2}\varepsilon _{3}=1\right\} \text{,} \\
&&\Sigma _{3}:=\left\{ \left( k_{1}^{\varepsilon _{1}},k_{2}^{\varepsilon
_{2}},k_{3}^{\varepsilon _{3}}\right) :k_{3}>k_{1}+k_{2},\varepsilon
_{1}\varepsilon _{2}\varepsilon _{3}=1\right\} \text{,} \\
&&\Sigma _{123}:=\left\{ \left( k_{1}^{\varepsilon _{1}},k_{2}^{\varepsilon
_{2}},k_{3}^{\varepsilon _{3}}\right) :k_{1}\leq k_{2}+k_{3},k_{2}\leq
k_{1}+k_{3},k_{3}\leq k_{1}+k_{2},\varepsilon _{1}\varepsilon
_{2}\varepsilon _{3}=1\right\} \text{.}
\end{eqnarray*}%
The transcendental nature of the Deligne's period $\Omega $ depends on the
critical region. We have, up to powers of $\pi $,%
\begin{equation*}
\Omega =\left\langle f_{i},f_{i}\right\rangle ^{2}\text{ on }\Sigma _{i}%
\text{ and }\Omega =\left\langle f_{1},f_{1}\right\rangle ^{2}\left\langle
f_{2},f_{2}\right\rangle ^{2}\left\langle f_{3},f_{3}\right\rangle ^{2}\text{
on }\Sigma _{123}\text{.}
\end{equation*}

Let $S_{i}$ be the set of places such that $\pi _{i,v}$ admits a
Jacquet-Langlands lift $\pi _{i,v}^{D}$\ to the group of units of the
division $\mathbb{Q}_{v}$-algebra. Set $S:=S_{1}\cap S_{2}\cap S_{3}$ and,
for every $v\in S$, let $d_{v}$ (resp. $d_{v}^{D}$) be the dimension of the
space of trilinear forms on $\pi _{1,v}\otimes \pi _{2,v}\otimes \pi _{3,v}$
(resp. $\pi _{1,v}^{D}\otimes \pi _{2,v}^{D}\otimes \pi _{3,v}^{D}$).
Define, for every $v\in S$,%
\begin{equation*}
\varepsilon _{v}\left( f_{1}\times f_{2}\times f_{3}\right) =\left\{ 
\begin{array}{cc}
1 & \text{if }d_{v}=1\text{ and }d_{v}^{D}=0 \\ 
-1 & \text{if }d_{v}=0\text{ and }d_{v}^{D}=1\text{.}%
\end{array}%
\right.
\end{equation*}%
It is a theorem of Prasad (see \cite{Pr})\ that the above function is indeed
well defined, i.e. only one of the above two possibilities occurs. Write $%
S=S^{+}\sqcup S^{-}$, where $S^{\pm }:=\left\{ v:\varepsilon _{v}\left(
f_{1}\times f_{2}\times f_{3}\right) =\pm 1\right\} $ and set $%
D:=\tprod\nolimits_{l\in S^{-}-\left\{ \infty \right\} }l$. Recalling the
dependence of these consideration from the weight, so that $S^{-}=S_{\mathbf{%
f}_{\underline{k}}^{\#}}^{-}$ (resp. $D=D_{\mathbf{f}_{\underline{k}}^{\#}}$%
), the sign of the function equation at $\underline{k}$ is given by the
formula%
\begin{equation*}
\varepsilon \left( \mathbf{f}_{\underline{k}}^{\#}\right)
=\tprod\nolimits_{v\in S}\varepsilon _{v}\left( \mathbf{f}_{\underline{k}%
}^{\#}\right) :=\tprod\nolimits_{v\in S}\varepsilon _{v}\left( f_{1}\times
f_{2}\times f_{3}\right) =\left( -1\right) ^{\#S_{\mathbf{f}_{\underline{k}%
}^{\#}}^{-}}\text{.}
\end{equation*}%
Let $\varepsilon _{\mathrm{fin}}\left( \mathbf{f}_{\underline{k}%
}^{\#}\right) $ be the product of the finite local signs, so that $%
\varepsilon \left( \mathbf{f}_{\underline{k}}^{\#}\right) =\varepsilon _{%
\mathrm{fin}}\left( \mathbf{f}_{\underline{k}}^{\#}\right) \varepsilon
_{\infty }\left( \mathbf{f}_{\underline{k}}^{\#}\right) $. We remark that
the nature of the local sign at infinity depends on the critical region: we
have $\varepsilon _{\infty }\left( \mathbf{f}_{\underline{k}}^{\#}\right) =1$
if $\underline{k}\in \Sigma _{i}$ for $i=1$, $2$ or $3$, while $\varepsilon
_{\infty }\left( \mathbf{f}_{\underline{k}}^{\#}\right) =-1$ if $\underline{k%
}\in \Sigma _{123}$. On the other hand, assuming that each $N_{i}$ is
squarefree, it is easy to see that $\varepsilon _{\mathrm{fin}}\left( 
\mathbf{f}_{\underline{k}}^{\#}\right) \equiv 1$ or $\varepsilon _{\mathrm{%
fin}}\left( \mathbf{f}_{\underline{k}}^{\#}\right) \equiv -1$ for every
arithmetic $\underline{k}$. Indeed, under this assumption we have, for every
arithmetic weight with trivial nebentype and every finite $v=l\in S_{\mathbf{%
f}_{\underline{k}}^{\#}}^{-}$,%
\begin{equation*}
\varepsilon _{l}\left( \mathbf{f}_{\underline{k}}^{\#}\right) =-a_{l}\left(
f_{1,k_{1}}\right) a_{l}\left( f_{2,k_{2}}\right) a_{l}\left(
f_{3,k_{3}}\right) l^{-\frac{k_{1}+k_{2}+k_{3}}{2}}\text{,}
\end{equation*}%
a function which can be $p$-adically interpolated and then needs to be
constant for all weights. Hence, having fixed $\mathbf{f}$ there is a well
defined finite "generic sign" $\varepsilon _{\mathrm{fin}}\left( \mathbf{f}%
\right) $ of the family and a well-posed interpolation problem. As explained
below, we can prove cases where this generic sign is well defined also when
the $N_{i}$s are not assumed to be squarefree, essentially because our $p$%
-adic interpolation technique avoids choosing test vectors and, hence,
having an a priori well-posed problem. Of course, we expect $\varepsilon _{%
\mathrm{fin}}\left( \mathbf{f}\right) $ to be defined in general. At this
point the consideration splits in two cases.

If $\varepsilon _{\mathrm{fin}}\left( \mathbf{f}\right) =1$ (hence $D$ is
the product of an \emph{even} number of primes), then $\varepsilon \left( 
\mathbf{f}_{\underline{k}}^{\#}\right) =1$ if and only if $\underline{k}\in
\Sigma _{1}$, $\Sigma _{2}$ or $\Sigma _{3}$. As expected, one gets three
(square root)$\ p$-adic $L$-function $\mathcal{L}^{\Sigma _{i}}\left( 
\mathbf{f}\right) $, one for every region $\Sigma _{i}$, with the property
that%
\begin{equation*}
\mathcal{L}^{\Sigma _{i}}\left( \mathbf{f}\right) \left( \underline{k}%
\right) ^{2}\overset{\cdot }{=}L\left( f_{k_{1}}\times f_{k_{2}}\times
f_{k_{3}},c_{\underline{k}}\right) \text{ for }\underline{k}\in \Sigma _{i}%
\text{.}
\end{equation*}%
Here we write $\overset{\cdot }{=}$ to mean equality up to explicit factors.
On the other hand, $L\left( f_{k_{1}}\times f_{k_{2}}\times f_{k_{3}},c_{%
\underline{k}}\right) =0$ when $\underline{k}\in \Sigma _{123}$ because of
the sign of the functional equation and the interpolation problem on $\Sigma
_{123}$\ is trivial. This is the Harris-Tilouine setting studied in \cite{HT}%
, under some ordinariness assumption and supposing $D=1$. These $p$-adic $L$%
-function have recently found interesting applications in \cite{DR1} and 
\cite{DR2}.

When $\varepsilon _{\mathrm{fin}}\left( \mathbf{f}\right) =-1$ (hence $D$ is
the product of an \emph{odd} number of primes), then $\varepsilon \left( 
\mathbf{f}_{\underline{k}}^{\#}\right) =1$ if and only if $\underline{k}\in
\Sigma _{123}$. The interpolation problem is therefore non-trivial only in
the balanced region. An unexpected new phenomenon is that, though we have
only one region of interpolation, we get in a natural way three (square root)%
$\ p$-adic $L$-functions $\mathcal{L}_{i}^{\Sigma _{123}}\left( \mathbf{f}%
\right) $ for $i=1$, $2$ or $3$, all of them interpolating in the same
region $\Sigma _{123}$:%
\begin{equation*}
\mathcal{L}_{i}^{\Sigma _{123}}\left( \mathbf{f}\right) \left( \underline{k}%
\right) ^{2}\overset{\cdot }{=}L\left( f_{k_{1}}\times f_{k_{2}}\times
f_{k_{3}},c_{\underline{k}}\right) \text{ for }\underline{k}\in \Sigma _{123}%
\text{.}
\end{equation*}%
Of course, since they interpolate in the same region and they are distinct,
they have different Euler factors $\mathcal{E}_{i}\left( \mathbf{f}\right) $%
. Let $V_{\mathbf{f}_{i}}$ be the $p$-adic representation attached to the
Coleman family $\mathbf{f}_{i}$ and set $V_{\mathbf{f}}:=V_{\mathbf{f}%
_{1}}\otimes V_{\mathbf{f}_{2}}\otimes V_{\mathbf{f}_{3}}$. Let $\widetilde{H%
}_{f,\Sigma _{123}}^{1}\left( \mathbb{Q},V_{\mathbf{f}}\right) $ be the
Nekovar extended Selmer group attached to $V_{\mathbf{f}}$ and the balanced
region $\Sigma _{123}$ (see \cite{Nek} and \cite{Pot} for the extension to
non-ordinary families). This means that $\widetilde{H}_{f,\Sigma
_{123}}^{1}\left( \mathbb{Q},V_{\mathbf{f}}\right) $ interpolates the Selmer
groups $\widetilde{H}_{f}^{1}\left( \mathbb{Q},V_{\mathbf{f}_{\underline{k}%
}^{\#}}\right) $ defined by the Bloch-Kato conditions at every $\underline{k}%
\in \Sigma _{123}$. It is a cohomological avatar of the extended Selmer
group of an elliptic curve and contains a period space $H^{0}\left( \mathbb{Q%
}_{p},V_{\mathbf{f}_{\underline{k}}^{\#}}\right) $. More precisely, the
Bloch-Kato conditions at $\underline{k}\in \Sigma _{123}$ are defined by
means of an exact sequence%
\begin{equation*}
0\rightarrow V_{\mathbf{f}_{\underline{k}}^{\#}}^{+}\rightarrow V_{\mathbf{f}%
_{\underline{k}}^{\#}}\rightarrow V_{\mathbf{f}_{\underline{k}%
}^{\#}}^{-}\rightarrow 0
\end{equation*}%
of $G_{\mathbb{Q}_{p}}$-modules (or $\left( \varphi ,\Gamma \right) $%
-modules, if we allow the families to be non-ordinary). We can interpolate
this exact sequence in a family%
\begin{equation*}
0\rightarrow V_{\mathbf{f}}^{+}\rightarrow V_{\mathbf{f}}\rightarrow V_{%
\mathbf{f}}^{-}\rightarrow 0
\end{equation*}%
which defines $\widetilde{H}_{f,\Sigma _{123}}^{1}\left( \mathbb{Q},V_{%
\mathbf{f}}\right) $. Under mild conditions, there is an exact sequence%
\begin{equation*}
0\rightarrow H^{0}\left( \mathbb{Q}_{p},V_{\mathbf{f}_{\underline{k}%
}^{\#}}\right) \rightarrow \widetilde{H}_{f}^{1}\left( \mathbb{Q},V_{\mathbf{%
f}_{\underline{k}}^{\#}}\right) \rightarrow H_{f}^{1}\left( \mathbb{Q},V_{%
\mathbf{f}_{\underline{k}}^{\#}}\right) \rightarrow 0
\end{equation*}%
where $H_{f}^{1}\left( \mathbb{Q},V_{\mathbf{f}_{\underline{k}}^{\#}}\right) 
$ is the usual Bloch-Kato Selmer group at $\underline{k}\in \Sigma _{123}$.
It can be proved that the vanishing of these Euler factors $\mathcal{E}%
_{i}\left( \mathbf{f}\right) $\ at some $\underline{k}\in \Sigma _{123}$ is
related to the dimension of this Nekovar period space at $\underline{k}\in
\Sigma _{123}$, which is almost three. More precisely, it can be proved that 
$H^{0}\left( \mathbb{Q}_{p},V_{\mathbf{f}_{\underline{k}}^{\#}}\right) \neq
0 $ if and only if $\mathcal{E}_{i}\left( \mathbf{f}\right) \left( 
\underline{k}\right) =0$ for some $i\in \left\{ 1,2,3\right\} $. In this
case, there are two possibilities: either $H^{0}\left( \mathbb{Q}_{p},V_{%
\mathbf{f}_{\underline{k}}^{\#}}\right) $ is two dimensional and only one
form $\mathbf{f}_{i,k_{i}}^{\#}=\mathbf{f}_{i,k_{i}}$ is new at $p$, say $%
i=3 $, and then $\mathcal{E}_{1}\left( \mathbf{f}\right) \left( \underline{k}%
\right) =\mathcal{E}_{2}\left( \mathbf{f}\right) \left( \underline{k}\right)
=0$; or $H^{0}\left( \mathbb{Q}_{p},V_{\mathbf{f}_{\underline{k}%
}^{\#}}\right) $ is three dimensional, all the modular forms $\mathbf{f}%
_{i,k_{i}}^{\#}=\mathbf{f}_{i,k_{i}}$ are new at $p$ and then $\mathcal{E}%
_{1}\left( \mathbf{f}\right) \left( \underline{k}\right) =\mathcal{E}%
_{2}\left( \mathbf{f}\right) \left( \underline{k}\right) =\mathcal{E}%
_{3}\left( \mathbf{f}\right) \left( \underline{k}\right) =0$. We refer the
reader to \cite{BSV} for an investigation of this kind of exceptional zero
phenomenona, a kind of exotic analogue of those discovered in \cite{MTT} and
studied in \cite{GrSt}.

In particular, by deforming $p$-new families, it may happen that two of
these $p$-adic $L$-function vanish at some point and the other does not,
showing that they are independent from the Iwasawa theoretic point of view,
because they generate different ideals in contrast with the fact that, by
the usual Iwasawa main conjecture flavour, one should expect all of them to
be related to the same characteristic ideal $\chi _{\Sigma _{123}}\left( 
\mathbf{f}\right) $\ attached to $\widetilde{H}_{f,\Sigma _{123}}^{1}\left( 
\mathbb{Q},V_{\mathbf{f}}\right) $. More precisely, we do not have $\left(
\chi _{\Sigma _{123}}\left( \mathbf{f}\right) \right) =\left( \mathcal{L}%
_{i}^{\Sigma _{123}}\left( \mathbf{f}\right) \right) $ in the Iwasawa
algebra $\Lambda _{\mathbf{f}}\subset \mathcal{O}\left( U_{\mathbf{f}%
}\right) $ (suppose all the families to be ordinary, for simplicity). In
fact, we expect our $p$-adic $L$-functions to encode a finer information,
related to the matrix coefficients of Nekovar weight pairing (see \cite{BSV}
and also \cite{V1} and \cite{V2} for the definition of the pairing).

The approach we take in the $p$-adic interpolation process is inspired from
Ichino's formula \cite{Ich}, as reformulated in \cite{GS2}. More precisely,
based on \cite{HK},\ it is proved in \cite{Ich} that we have the equality%
\begin{equation}
I_{\underline{k}}\left( \phi \right) ^{2}=\frac{1}{2^{3}}\frac{\zeta _{%
\mathbb{Q}}^{2}\left( 2\right) L\left( 1/2,\Pi _{\underline{k}}\right) }{%
L\left( 1,\Pi ,\mathrm{Ad}\right) }\prod\nolimits_{v}\alpha _{v}\left( \phi
_{v}\right)  \label{Intro F special value}
\end{equation}%
for $\phi =\otimes _{v}\phi _{v}\in \Pi ^{D_{\mathbf{f}_{\underline{k}%
}^{\#}}}$, the Jacquet-Langlands lift of $\Pi _{\mathbf{f}_{\underline{k}%
}^{\#}}$ to the quaternion algebra $B=B_{\mathbf{f}_{\underline{k}}^{\#}}$
of discriminant $D=D_{\mathbf{f}_{\underline{k}}^{\#}}$. Here $\alpha _{v}$
are integral of matrix coefficients and $I_{\underline{k}}$ is a global
period integral (see \cite{Ich}). In Theorem \ref{T Main} we\ prove an
equality between functionals, showing that there are three linear forms on
quaternionic families of $p$-adic modular forms interpolating the global
period integral $I_{\underline{k}}$\ that appears in Ichino's formula,
independently of any consideration about the explicit form of the test
vectors putting restrictions on the levels. As a consequence, we get
applications to the existence of $p$-adic families of test vectors over
Zariski open subsets of three copies $\mathcal{X}^{3}$\ of the weight space.
As explained, both the quaternion algebra $B=B_{\mathbf{f}}$ or discriminant 
$D=D_{\mathbf{f}}$ and the resulting interpolation region\ we have to deal
with are not a priori well defined in general (but they are in the
squarefree case). Another application of our result is the following
conditional result, which however removes any level assumption. Let $D_{%
\mathbf{f}_{\underline{k}}^{\#}}$ be the discriminant predicted by \cite{Pr}
and write $\Pi ^{D_{\mathbf{f}_{\underline{k}}^{\#}}}$ for the
Jacquet-Langlands lifts of $\Pi _{\mathbf{f}_{\underline{k}}^{\#}}$. The
Jacquet conjecture proved in \cite{HK} tells us that $L\left( \Pi _{\mathbf{f%
}_{\underline{k}}^{\#}},1/2\right) \neq 0$ if and only if $I_{\underline{k}%
}\circ \Pi ^{D_{\mathbf{f}_{\underline{k}}^{\#}}}\neq 0$. The following
results are immediate applications of our result.

\begin{itemize}
\item[$\left( A\right) $] If $L\left( \Pi _{\mathbf{f}_{\underline{k}%
}^{\#}},1/2\right) \neq 0$ for some arithmetic $\underline{k}$ with $D=D_{%
\mathbf{f}_{\underline{k}}^{\#}}$ divisible by an odd number of primes and
such that the Euler factors $\left( \text{\ref{F Euler}}\right) $ do not
vanish, then $D_{\mathbf{f}_{\underline{k}^{\prime }}^{\#}}=D$ for every
arithmetic $\underline{k}^{\prime }$ as above in a Zariski open subset $\phi
\neq U_{\mathbf{f}}^{\prime }$\ of $U_{\mathbf{f}}$. In particular, there is
a notion of "generic sign" $\varepsilon _{\mathrm{fin}}\left( \mathbf{f}%
\right) $ and a natural region of interpolation, namely the balanced region.

\item[$\left( B\right) $] If $L\left( \Pi _{\mathbf{f}},1/2\right) \neq 0$
for some arithmetic $\underline{k}$ with $D=D_{\mathbf{f}_{\underline{k}%
}^{\#}}$ divisible by an odd number of primes and such that the Euler
factors $\left( \text{\ref{F Euler}}\right) $ do not vanish, then there is a
quaternionic $p$-adic family $\mathbf{\varphi }$ on the definite quaternion
algebra of discriminant $D$ such that $\mathbf{\varphi }_{\underline{k}%
^{\prime }}$ is a global test vector for every $\underline{k}^{\prime }$ as
above in a Zariski open subset $\phi \neq U_{\mathbf{f}}^{\prime }$\ of $U_{%
\mathbf{f}}$. In particular, we have $L\left( \Pi _{\mathbf{f}_{\underline{k}%
^{\prime }}^{\#}},1/2\right) \neq 0$.
\end{itemize}

We expect $U_{\mathbf{f}}^{\prime }=U_{\mathbf{f}}$ and the result above
being unconditional: in order to prove this fact, an investigation of the $p$%
-adic variation of the right hand side of $\left( \text{\ref{Intro F special
value}}\right) $ is required for all possible $D$s.

\section{Modular forms and $p$-adic modular forms}

Let $B$ be a definite quaternion division $\mathbb{Q}$-algebra which is
split at the prime $p$ and let $\mathbf{B}$ (resp. $\mathbf{B}^{\times }$)
be the associated ring scheme (resp. algebraic group). We write $\mathbb{A}=%
\mathbb{A}_{f}\times \mathbb{R}$ for the adele ring of $\mathbb{Q}$ and
define $\mathbb{A}_{f}^{p}$ by the rule $\mathbb{A}_{f}=\mathbb{A}%
_{f}^{p}\times \mathbb{Q}_{p}$. We set $B_{f}:=\mathbf{B}\left( \mathbb{A}%
_{f}\right) $ (resp. $B_{f}^{\times }:=\mathbf{B}^{\times }\left( \mathbb{A}%
_{f}\right) $), $B_{f}^{\times ,p}:=\mathbf{B}\left( \mathbb{A}%
_{f}^{p}\right) $\ and $B_{v}=\mathbf{B}\left( \mathbb{Q}_{v}\right) $
(resp. $B_{v}^{\times }:=\mathbf{B}^{\times }\left( \mathbb{Q}_{v}\right) $)
if $v$ is either a finite place or $v=\infty $, so that $B_{f}^{\times
}=B_{f}^{\times ,p}\times B_{p}^{\times }$. We write $b\mapsto b^{\iota }$
for the main involution and $\mathrm{nrd}:\mathbf{B}^{\times }\rightarrow 
\mathbf{G}_{m}$ for the reduced norm.

If $\mathbf{Z}\subset \mathbf{Z}_{\mathbf{B}^{\times }}=\mathbf{G}_{m}$ is a
closed subgroup (such as the trivial subgroup or the whole center), we
define $Z_{f}:=\mathbf{Z}\left( \mathbb{A}_{f}\right) $, $Z_{v}:=\mathbf{Z}%
\left( \mathbb{Q}_{v}\right) $ and $Z_{f}^{p}:=\mathbf{Z}\left( \mathbb{A}%
_{f}^{p}\right) $, so that $Z_{f}=Z_{f}^{p}\times Z_{p}$. We will need to
consider double cosets of the form%
\begin{equation*}
\left[ B_{f}^{\times }\right] _{\mathbf{Z}}:=Z_{f}\backslash B_{f}^{\times
}/B^{\times }\text{.}
\end{equation*}%
In order to later apply the results from \cite{GS2}, we fix measures as
follows. We take the Tamagawa measure $\mu _{\mathbf{Z}\left( \mathbb{A}%
\right) \backslash \mathbf{B}^{\times }\left( \mathbb{A}\right) }$\ on $%
\mathbf{Z}\left( \mathbb{A}\right) \backslash \mathbf{B}^{\times }\left( 
\mathbb{A}\right) $ and write $\mu _{\left[ \mathbf{B}^{\times }\left( 
\mathbb{A}\right) \right] _{\mathbf{Z}}}$ for the quotient measure
(normalized in the usual way). Next we choose $\mu _{\mathbf{Z\backslash B}%
,\infty }$ on $\mathbf{Z}\left( F_{\infty }\right) \backslash \mathbf{B}%
^{\times }\left( F_{\infty }\right) \ $and $\mu :=\mu _{B_{f}^{\times }}$ on 
$B_{f}^{\times }$ such that $\mu \left( K\right) \in \mathbb{Q}$ for some
(and hence every) open and compact subgroup $K\subset B_{f}^{\times }$ and
such that, writing $\mu _{B_{f}^{\times }/B^{\times }}$ for the induced
quotient measure on $B_{f}^{\times }/B^{\times }$ (normalized in the usual
way), which restricts to an invariant measure $\mu _{\left[ B_{f}^{\times }%
\right] _{\mathbf{Z}}}$ on $C(B_{f}\backslash B_{f}^{\times }/B^{\times
})\subset C(B_{f}^{\times }/B^{\times })$,%
\begin{equation*}
\tint\nolimits_{\left[ \mathbf{B}^{\times }\left( \mathbb{A}\right) \right]
_{\mathbf{Z}}}f\left( x\right) d\mu _{\left[ \mathbf{B}^{\times }\left( 
\mathbb{A}\right) \right] _{\mathbf{Z}}}\left( x\right) =\tint\nolimits_{%
\left[ B_{f}^{\times }\right] _{\mathbf{Z}}}\left( \tint\nolimits_{\mathbf{Z}%
\left( F_{\infty }\right) \backslash \mathbf{B}^{\times }\left( F_{\infty
}\right) }f\left( x_{f}x_{\infty }\right) d\mu _{\mathbf{Z\backslash B}%
,\infty }\left( x_{\infty }\right) \right) d\mu _{\left[ B_{f}^{\times }%
\right] _{\mathbf{Z}}}\left( x_{f}\right)
\end{equation*}%
is satisfied. We let $m_{\mathbf{Z\backslash B},\infty }$ be the total
measure of $\mathbf{Z}\left( F_{\infty }\right) \backslash \mathbf{B}%
^{\times }\left( F_{\infty }\right) $.

Let $\Sigma _{0}\left( p\mathbb{Z}_{p}\right) \subset \mathbf{M}_{2}\left( 
\mathbb{Z}_{p}\right) $ be the subsemigroup of matrices having non-zero
determinant, upper left entry $a\in \mathbb{Z}_{p}^{\times }$ and lower left
entry $c\in p\mathbb{Z}_{p}$ and set $\Gamma _{0}\left( p\mathbb{Z}%
_{p}\right) :=\Sigma _{0}\left( p\mathbb{Z}_{p}\right) \cap \mathbf{GL}%
_{2}\left( \mathbb{Z}_{p}\right) $. Consider an open and compact subgroup $%
K_{p}^{\diamond }\subset B_{p}$ (it will be $\Gamma _{0}\left( p\mathbb{Z}%
_{p}\right) $). We will also need to consider a subsemigroup $%
K_{p}^{\diamond }\subset \Sigma _{p}\subset B_{p}^{\times }$ and to define $%
\Sigma _{p}\left( B_{f}^{\times }\right) :=B_{f}^{\times ,p}\times \Sigma
_{p}$ (we will take $\Sigma _{p}=\Sigma _{0}\left( p\mathbb{Z}_{p}\right) $).

\bigskip

Let $\mathcal{K}:=\mathcal{K}\left( B_{f}^{\times }\right) $ (resp. $%
\mathcal{K}^{\diamond }:=\mathcal{K}\left( B_{f}^{\times },K_{p}^{\diamond
}\right) $) be the set of open and compact subgroups $K\subset B_{f}^{\times
}$ (resp. $K=K^{p}\times K_{p}$ with $K^{p}\subset B_{f}^{\times ,p}$ and $%
K_{p}\subset K_{p}^{\diamond }$\ open and compact). If $S$ is a $%
B_{f}^{\times }$-module (resp. a $\Sigma _{p}\left( B_{f}^{\times }\right) $%
-module), then we define%
\begin{equation*}
S^{\mathcal{K}}:=\tbigcup\nolimits_{K\in \mathcal{K}}S^{K}\text{ (resp. }S^{%
\mathcal{K}^{\diamond }}:=\tbigcup\nolimits_{K\in \mathcal{K}^{\diamond
}}S^{K}\text{).}
\end{equation*}%
We note that the Hecke operators $\mathcal{H}\left( B_{f}^{\times }\right) $
(resp. $\mathcal{H}\left( \Sigma _{p}\left( B_{f}^{\times }\right) \right) $%
)\ act on $S^{\mathcal{K}}$ (resp. $S^{\mathcal{K}^{\diamond }}$) by double
cosets on $B_{f}^{\times }$ (resp. $\Sigma _{p}\left( B_{f}^{\times }\right) 
$). We describe the action on $S^{\mathcal{K}^{\diamond }}$ for a $\Sigma
_{p}\left( B_{f}^{\times }\right) $-module $S$. If $K_{1},K_{2}\in \mathcal{K%
}^{\diamond }$ and $\pi \in \Sigma _{p}\left( B_{f}^{\times }\right) $, the
space $K_{1}\backslash K_{1}\pi K_{2}$ is finite\footnote{%
Indeed note that $K_{1}\pi K_{2}$ is compact, being the image of $%
K_{1}\times K_{2}$ by means of the continuous map given by $\left(
x,y\right) \mapsto x\pi y$. Since $K_{1}$ is open, $K_{1}\pi
K_{2}=\bigsqcup\nolimits_{i}K_{1}\pi _{i}$ is an open covering which, by
compactness, admits a finite refinement.} and we may write%
\begin{equation*}
K_{1}\pi K_{2}=\tbigsqcup\nolimits_{x\in K_{1}\backslash K_{1}\pi
K_{2}}K_{1}x\text{.}
\end{equation*}%
As usual we may define%
\begin{equation*}
\left[ K_{1}\pi K_{2}\right] :S^{K_{1}}\rightarrow S^{K_{2}}
\end{equation*}%
by the rule%
\begin{equation}
v\mid \left[ K_{1}\pi K_{2}\right] =\tsum\nolimits_{x\in K_{1}\backslash
K_{1}\pi K_{2}}vx\text{.}  \label{F Hecke Def}
\end{equation}%
The mapping $u\mapsto \pi u$ induces a bijection $\left( K_{2}\cap \pi
^{-1}K_{1}\pi \right) \backslash K_{2}\rightarrow K_{1}\backslash K_{1}\pi
K_{2}$, so that we may take $x=\pi u$ in the above expression:%
\begin{equation}
v\mid \left[ K_{1}\pi K_{2}\right] =\tsum\nolimits_{u\in \left( K_{2}\cap
\pi ^{-1}K_{1}\pi \right) \backslash K_{2}}v\pi u\text{.}
\label{F Hecke Operator}
\end{equation}%
We can define in this way an action of the Hecke algebra $\mathcal{H}\left(
\Sigma _{p}\left( B_{f}^{\times }\right) \right) $ of double cosets on $%
\Sigma _{p}\left( B_{f}^{\times }\right) $. When $K_{p}^{\diamond }=G_{p}$
we have $V^{\mathcal{K}^{\diamond }}=V^{\mathcal{K}}$ and we have an action
of $\mathcal{H}\left( \Sigma _{p}\left( B_{f}^{\times }\right) \right) =%
\mathcal{H}\left( B_{f}^{\times }\right) $. Let $\mathcal{K}^{\diamond
\diamond }\subset \mathcal{K}^{\diamond }$ be the subset of those groups
such that $K_{p}=K_{p}^{\diamond }$ and write $\mathcal{H}\left( \Sigma
_{p}\right) $ for the Hecke algebra of double cosets $K\pi K$ with $\pi $
contentrated in $\pi _{p}\in \Sigma _{p}$ and $K\in \mathcal{K}^{\diamond
\diamond }$. Then $\left( \text{\ref{F Hecke Def}}\right) $ defines an
operator on $V^{\mathcal{K}^{\diamond \diamond }}=\left( V^{\mathcal{K}%
^{\diamond }}\right) ^{K_{p}^{\diamond }}$ by means of the formula $vT_{\pi
}:=v\mid \left[ K\pi K\right] $ if $v\in V^{K}$ where $K\in \mathcal{K}%
^{\diamond \diamond }$, i.e. it does not depend on $K\in \mathcal{K}%
^{\diamond \diamond }$. It follows that $V^{\mathcal{K}^{\diamond \diamond
}} $ is endowed with an action of $B_{f}^{\times ,p}\times \mathcal{H}\left(
\Sigma _{p}\right) $.

\bigskip

Let $\left( V,\rho \right) $ be a right representation of $G_{\infty }$
(resp. $\Sigma _{p}$)\ with coefficients in some commutative unitary ring $R$%
. If $g\in \mathbf{B}^{\times }\left( \mathbb{A}\right) $, we will write $%
g_{v}\in B_{v}^{\times }$ for its $v$-component. When $\rho $ is understood,
we simply write $vg_{\infty }$ (resp. $vg_{p}$)\ for $v\rho \left( g_{\infty
}\right) $ (resp.\ $v\rho \left( g_{\infty }\right) $). Fix a character $%
\omega _{0}:Z_{f}\longrightarrow R^{\times }$ (resp. $\omega
_{0,p}:Z_{f}\longrightarrow R^{\times }$). Define $S\left( B_{f}^{\times
},\rho \right) $ (resp. $S_{p}\left( B_{f}^{\times },\rho \right) $)\ to be
the space of maps $\varphi :B_{f}^{\times }\rightarrow V$ endowed with the $%
\left( \mathbf{B}^{\times }\left( \mathbb{A}\right) ,B_{f}^{\times }\right) $%
-action (resp. $\left( \mathbf{B}^{\times }\left( \mathbb{A}\right) ,\Sigma
_{p}\left( B_{f}^{\times }\right) \right) $-action)\ given by%
\begin{eqnarray*}
&&\text{ }\left( g\varphi u\right) \left( x\right) :=\varphi \left(
uxg_{f}\right) \rho \left( g_{\infty }^{-1}\right) ,\quad \text{where }g\in 
\mathbf{B}^{\times }\left( \mathbb{A}\right) \text{ and }u\in B_{f}^{\times }
\\
&&\text{ (resp. }\left( g\varphi u\right) \left( x\right) :=\varphi \left(
uxg_{f}\right) \rho \left( u_{p}\right) \text{, where }g\in \mathbf{B}%
^{\times }\left( \mathbb{A}\right) \text{ and }u\in \Sigma _{p}\left(
B_{f}^{\times }\right) \text{).}
\end{eqnarray*}%
Then%
\begin{eqnarray*}
&&\text{ }S(B_{f}^{\times },\rho ,\omega _{0}):=\{\varphi \in
S(B_{f}^{\times },\rho ):\varphi z=\omega _{0}(z)\varphi \text{ for all $%
z\in Z_{f}$}\} \\
&&\text{ (resp. }S_{p}\left( G_{f},\rho ,\omega _{0,p}\right) :=\left\{
\varphi \in S_{p}(B_{f}^{\times },\rho ):\varphi z=\omega _{0,p}(z)\varphi 
\text{ for all $z\in Z_{f}$}\right\} \text{)}
\end{eqnarray*}%
is a sub $\left( \mathbf{B}^{\times }\left( \mathbb{A}\right) ,B_{f}^{\times
}\right) $-module (resp. sub $\left( \mathbf{B}^{\times }\left( \mathbb{A}%
\right) ,\Sigma _{p}\left( B_{f}^{\times }\right) \right) $-module). We also
write%
\begin{equation*}
S\left( B_{f}^{\times }/B^{\times },\rho _{/B^{\times }},\omega _{0}\right)
:=S\left( B_{f}^{\times },\rho ,\omega _{0}\right) ^{\left( B^{\times
},1\right) }\text{ (resp. }S_{p}\left( B_{f}^{\times }/B^{\times },\rho
_{/B^{\times }},\omega _{0,p}\right) :=S_{p}\left( B_{f}^{\times },\rho
,\omega _{0,p}\right) ^{\left( B^{\times },1\right) }\text{)}
\end{equation*}%
and%
\begin{equation*}
M\left( B_{f}^{\times },\rho ,\omega _{0}\right) :=S\left( B_{f}^{\times
}/B^{\times },\rho _{/B^{\times }},\omega _{0}\right) ^{(1,\mathcal{K})}%
\text{ (resp. }M_{p}\left( B_{f}^{\times },\rho ,\omega _{0,p}\right)
:=S_{p}\left( B_{f}^{\times }/B^{\times },\rho _{/B^{\times }},\omega
_{0,p}\right) ^{\left( 1,\mathcal{K}^{\diamond }\right) }\text{).}
\end{equation*}%
The former is called the \emph{space of }$\rho $\emph{-valued modular forms}
and the latter the space of \emph{space of }$\rho $\emph{-valued }$p$\emph{%
-adic modular forms}; they are Hecke modules as explained above. We omit $%
\omega _{0}$ from the notation when $Z_{f}=1$ and write $M\left(
Z_{f}\backslash B_{f}^{\times },\rho \right) :=M\left( B_{f}^{\times },\rho
,\omega _{0,p}\right) $ when $\omega _{0}$ is the trivial character of $%
Z_{f} $. Sometimes we will abusively replace $\rho $ with the underlying
subspace $V$ in the notation. The same shorthands apply in the $p$-adic case.

The connection between modular forms and $p$-adic modular forms is the
content of the following proposition. We suppose that we have given $\omega
_{0}:Z_{f}\rightarrow R^{\times }$ and coefficient rings $i_{\infty
}:R\subset R_{\infty }$ and $i_{p}:R\subset R_{p}$. For a character $\chi $
of some group with values in $R^{\times }$, we let $i_{p\ast }\left( \chi
\right) :=i_{p}\circ \chi $ and $i_{\infty \ast }\left( \chi \right)
:=i_{\infty }\circ \chi $. We also assume that we have given a
representation $\rho _{p}$ (resp. $\rho _{\infty }$)\ of $B_{p}^{\times }$
(resp. $B_{\infty }^{\times }$) with coefficients in $R_{p}$ (resp. $%
R_{\infty }$) with the property that%
\begin{equation*}
\rho :=\rho _{p\mid B^{\times }}=\rho _{\infty \mid B^{\times }}\subset \rho
_{p},\rho _{\infty }
\end{equation*}%
with coefficients in $R$.

\begin{lemma}
\label{L1}The rules%
\begin{equation*}
\begin{array}{ccc}
M\left( B_{f}^{\times },\rho _{p}\right) \rightarrow M_{p}\left(
B_{f}^{\times },\rho _{p}\right) &  & M_{p}\left( B_{f}^{\times },\rho
_{p}\right) \rightarrow M\left( B_{f}^{\times },\rho _{p}\right) \\ 
\varphi \mapsto \psi _{\varphi }:\psi _{\varphi }\left( x\right) :=\varphi
\left( x\right) x_{p}^{-1} &  & \psi \mapsto \varphi _{\psi }:\varphi _{\psi
}\left( x\right) :=\psi \left( x\right) x_{p}%
\end{array}%
\end{equation*}%
set up a right $\Sigma _{p}\left( B_{f}^{\times }\right) $-equivariant
bijection and $M\left( B_{f}^{\times },\rho \right) \subset M\left(
B_{f}^{\times },\rho _{p}\right) $ is identified with the submodule of those 
$\psi \in M_{p}\left( B_{f}^{\times },\rho _{p}\right) $ such that $\psi
\left( x\right) \in \rho \subset \rho _{p}$ for every $x\in B_{f}^{\times }$%
. Furthermore, if $\rho _{p}$ has central character $\omega _{\rho _{p}}$
and $\left( -\right) _{p}:Z_{f}\rightarrow Z_{p}$ is the projection induced
by $B_{f}^{\times }\rightarrow B_{p}^{\times }$, then the bijection induces%
\begin{equation*}
M\left( B_{f}^{\times },\rho ,\omega _{0}\right) \subset M\left(
B_{f}^{\times },\rho _{p},i_{p\ast }\left( \omega _{0}\right) \right) \simeq
M_{p}\left( B_{f}^{\times },\rho _{p},\omega _{0,p}\right)
\end{equation*}%
with $\omega _{0,p}:=i_{p\ast }\left( \omega _{0}\right) \omega _{\rho
_{p}}^{-1}\left( \left( -\right) _{p}\right) $. These identifications and
inclusions are $\mathcal{H}\left( \Sigma _{p}\left( G_{f}\right) \right) $%
-equivariant.
\end{lemma}

\begin{proof}
Indeed the above rules induce a $\left( B^{\times },\Sigma _{p}\left(
B_{f}^{\times }\right) \right) $-equivariant identification $S\left(
B_{f}^{\times },\rho _{p}\right) \simeq S_{p}\left( B_{f}^{\times },\rho
_{p}\right) $. Since $B_{f}=B_{f}^{\times ,p}\times B_{p}^{\times }$
topologically, $\mathcal{K}^{\diamond }\subset \mathcal{K}$ is a cofinal
family and we have $S^{\mathcal{K}}=S^{\mathcal{K}^{\diamond }}$ for every $%
B_{f}^{\times }$-module. Hence, taking $\left( B^{\times },\mathcal{K}%
\right) $-invariant on the left and $\left( B^{\times },\mathcal{K}%
^{\diamond }\right) $-invariants on the right yields the $\Sigma _{p}\left(
B_{f}^{\times }\right) $-equivariant identification. Then one checks that
the correspondence has the required properties.
\end{proof}

\bigskip

\begin{example}
The above lemma notably applies in the following setting: let $\mathbf{V}$
be an algebraic representation of $B^{\times }$ over $R=\mathbb{Q}\subset 
\mathbb{C},\mathbb{Q}_{p}$ (or a quadratic field $R=K$ which splits $B$)\
and set $\left( V,\rho \right) :=\mathbf{V}\left( \mathbb{Q}\right) $ and $%
\left( V_{p},\rho _{p}\right) :=\mathbf{V}\left( \mathbb{Q}_{p}\right) $. We
can also take $R$ large enough for the values of the characters $\omega
_{\rho _{p}}$ and $\omega _{0}$ to take values in it.
\end{example}

\subsection{\label{SS Norm forms}The norm forms}

Here is a key example of modular form. Consider the (normalized)\ absolute
value functions $\left\vert -\right\vert _{v}:\mathbb{Q}_{v}^{\times
}\rightarrow \mathbb{R}_{+}^{\times }$, $\left\vert -\right\vert _{\mathbb{A}%
_{f}}:\mathbb{A}_{f}^{\times }\rightarrow \mathbb{Q}_{+}^{\times }$ and $%
\left\vert -\right\vert _{\mathbb{A}}:\mathbb{A}^{\times }\rightarrow 
\mathbb{R}_{+}^{\times }$. Setting%
\begin{equation*}
\mathrm{N}:=\left\vert -\right\vert _{\mathbb{A}_{f}}^{-1}\left\vert
-\right\vert _{\infty }:\mathbb{A}^{\times }=\mathbf{G}_{m}\left( \mathbb{A}%
\right) \rightarrow \mathbb{C}^{\times }
\end{equation*}%
gives a function such that $\mathrm{N}_{f}\mathrm{N}_{\infty
}^{-1}=\left\vert -\right\vert _{\mathbb{A}}^{-1}$ is trivial on $\mathbb{Q}%
^{\times }=\mathbf{G}_{m}\left( \mathbb{Q}\right) $ by the product formula.
Suppose that $\chi :\mathbf{B}^{\times }\rightarrow \mathbf{G}_{m}$ is an
algebraic character and that $\tau :R^{\times }\rightarrow G$ is a
character. Then we define $\tau _{\chi }:\mathbf{B}^{\times }\left( R\right) 
\overset{\chi _{R}}{\rightarrow }R^{\times }\overset{\tau }{\rightarrow }G$.
In particular, we have the continuous character%
\begin{equation*}
\mathrm{N}_{\chi }:\mathbf{B}^{\times }\left( \mathbb{A}\right) \overset{%
\chi _{\mathbb{A}}}{\rightarrow }\mathbb{A}^{\times }\overset{\mathrm{N}}{%
\rightarrow }\mathbb{R}_{+}^{\times }
\end{equation*}%
and, recalling that $\mathrm{N}_{f}=\left\vert -\right\vert _{\mathbb{A}%
_{f}}^{-1}$ and $\mathrm{N}_{\infty }=\left\vert -\right\vert _{\infty }$,%
\begin{equation*}
\mathrm{N}_{\chi ,f}:\mathbf{B}^{\times }\left( \mathbb{A}_{f}\right) 
\overset{\chi _{\mathbb{A}_{f}}}{\rightarrow }\mathbb{A}_{f}^{\times }%
\overset{\left\vert -\right\vert _{\mathbb{A}_{f}}^{-1}}{\rightarrow }%
\mathbb{Q}_{+}^{\times }\text{ and }\mathrm{N}_{\chi ,\infty }:\mathbf{B}%
^{\times }\left( \mathbb{R}\right) \overset{\chi _{\infty }}{\rightarrow }%
\mathbb{R}^{\times }\overset{\left\vert -\right\vert _{\infty }}{\rightarrow 
}\mathbb{R}_{+}^{\times }\text{.}
\end{equation*}%
Of course $\mathrm{N}_{\chi ,f}$ (resp. $\mathrm{N}_{\chi ,\infty }$) is the
finite adele (resp. $\infty $) component of $\mathrm{N}_{\chi }$, as
suggested by the notation. If $\kappa :\mathbb{Q}_{+}^{\times }\rightarrow
R^{\times }$ is a character (that we usually write exponentially $r\mapsto
r^{\kappa }$), we can also define%
\begin{equation*}
\mathrm{N}_{\chi ,f}^{\kappa }:\mathbf{G}\left( \mathbb{A}_{f}\right) 
\overset{\mathrm{N}_{\chi ,f}}{\rightarrow }\mathbb{Q}_{+}^{\times }\overset{%
\kappa }{\rightarrow }R^{\times }
\end{equation*}%
Note that $\chi _{\infty }\left( \mathbf{B}^{\times }\left( \mathbb{R}%
\right) \right) =\chi _{\infty }\left( \mathbf{B}^{\times }\left( \mathbb{R}%
\right) ^{\circ }\right) \subset \mathbb{R}_{+}^{\times }$ (because $B$ is
definite), implying that $\chi _{\mathbb{Q}}\left( \mathbf{B}^{\times
}\left( \mathbb{Q}\right) \right) \subset \mathbb{Q}_{+}^{\times }$ and we
may consider $\kappa _{\chi }:=\kappa \circ \chi _{\mathbb{Q}}$. If $%
V=\left( V,\rho \right) $ is a representation of $\mathbf{G}\left( F_{\infty
}\right) $ with coefficients in $R$, we write $V\left( \kappa _{\chi
}\right) =\left( V,\rho \left( \kappa _{\chi }\right) \right) $ for the
representation $\rho \left( \kappa _{\chi }\right) \left( g\right) \left(
v\right) :=\kappa _{\chi }\left( g\right) \rho \left( g\right) v$.

\begin{remark}
\label{Borel-Weil R1}The continuous character $\mathrm{N}_{\chi }$ is such
that $\mathrm{N}_{\chi ,f}\mathrm{N}_{\chi ,\infty }^{-1}$ is trivial on $%
\mathbf{G}\left( F\right) $ and we have%
\begin{equation*}
\mathrm{N}_{\chi ,f}^{\kappa }\in M\left( \mathbf{G}\left( \mathbb{A}%
_{f}\right) ,R\left( \kappa _{\chi }\right) ,\mathrm{N}_{\chi ,f\mid
Z_{f}}^{\kappa }\right) ^{K}
\end{equation*}%
for every open and compact $K\in \mathcal{K}$.
\end{remark}

\begin{proof}
This is an application of the product formula and the fact that $\chi _{%
\mathbb{Q}}\left( \mathbf{B}^{\times }\left( \mathbb{Q}\right) \right)
\subset \mathbb{Q}_{+}^{\times }$.
\end{proof}

\bigskip

Taking%
\begin{equation*}
\chi =\mathrm{nrd}:\mathbf{B}^{\times }\rightarrow \mathbf{G}_{m}
\end{equation*}%
yields, for every $\kappa =k\in \mathbb{Z}$ (viewed as the character $k:%
\mathbb{Q}^{\times }\rightarrow R$ via $r\mapsto r^{k}$), the norm form%
\begin{equation*}
\mathrm{Nrd}_{f}^{k}:=\mathrm{N}_{\chi ,f}^{k}\in M\left( \mathbf{G}\left( 
\mathbb{A}_{f}\right) ,\mathbb{Q}\left( k\right) ,\mathrm{N}_{f\mid
Z_{f}}^{2k}\right) ^{K}\text{, for every }K\in \mathcal{K}\text{.}
\end{equation*}%
Applying Lemma \ref{L1} with $\rho =\mathbb{Q}\left( k\right) $, $\rho _{p}=%
\mathbb{Q}_{p}\left( k\right) $ and $\varphi =\mathrm{Nrd}_{f}^{k}\left(
-\right) _{p}\in M\left( \mathbf{G}\left( \mathbb{A}_{f}\right) ,\mathbb{Q}%
_{p}\left( k\right) ,2k\right) ^{K}$ yields the $p$-adic modular form%
\begin{equation*}
\mathrm{Nrd}_{p}^{k}:=\psi _{\varphi }\in M_{p}\left( \mathbf{G}\left( 
\mathbb{A}_{f}\right) ,\mathbb{Q}_{p}\left( k\right) ,\mathrm{N}_{p\mid
Z_{f}}^{2k}\right) ^{K}\text{, for every }K\in \mathcal{K}\text{.}
\end{equation*}%
We have, explicitely,%
\begin{equation*}
\mathrm{Nrd}_{p}^{k}\left( x\right) =\mathrm{Nrd}_{f}^{k}\left( x\right)
_{p}x_{p}^{-1}=\left( \frac{\mathrm{Nrd}_{f}\left( x\right) _{p}}{\mathrm{nrd%
}_{p}\left( x_{p}\right) }\right) ^{k}\text{.}
\end{equation*}

We now remark that $\frac{\mathrm{Nrd}_{f}\left( x\right) _{p}}{\mathrm{nrd}%
_{p}\left( x_{p}\right) }\in \mathbb{Z}_{p}^{\times }$ for any $x\in
B_{f}^{\times }$. Suppose now that we have given $\mathbf{k}:\mathbb{Z}%
_{p}^{\times }\rightarrow \mathcal{O}^{\times }$ which is a continuous group
homomorphism, with $\mathcal{O}$ a locally convex $\mathbb{Q}_{p}$-algebra.
Since $K_{p}^{\diamond }\subset B_{p}^{\times }$ is a compact subgroup, $%
\mathrm{nrd}_{p}$ maps it into the maximal open compact subgroup $\mathbb{Z}%
_{p}^{\times }\subset \mathbb{Q}_{p}^{\times }$:%
\begin{equation*}
\mathrm{nrd}_{p}:K_{p}^{\diamond }\rightarrow \mathbb{Z}_{p}^{\times }\text{.%
}
\end{equation*}%
If $D$ is a $K_{p}^{\diamond }$-module with coefficients in $\mathcal{O}$,
it makes sense to consider $D\left( \mathbf{k}\right) :=D\left( \mathrm{nrd}%
_{p}^{\mathbf{k}}\right) $, the same representation with action $v\cdot _{%
\mathbf{k}}g:=\mathrm{nrd}_{p}^{\mathbf{k}}\left( g\right) vg$. With this
notation, we have%
\begin{equation*}
\mathrm{Nrd}_{p}^{\mathbf{k}}\in M_{p}\left( \mathbf{G}\left( \mathbb{A}%
_{f}\right) ,\mathcal{O}\left( \mathbf{k}\right) ,\mathrm{N}_{p\mid Z_{f}}^{2%
\mathbf{k}}\right) ^{K}\text{, for every }K\in \mathcal{K}^{\diamond }
\end{equation*}%
which \emph{interpolates} the norm forms $\mathrm{Nrd}_{f}^{k}\simeq \mathrm{%
Nrd}_{p}^{k}$ with $k\in \mathbb{Z}$.

\subsection{multilinear forms}

For $x\in B_{f}^{\times }$ and $K\in \mathcal{K}$, define $\Gamma
_{K}(x)=B^{\times }\cap x^{-1}Kx$. Being discrete (as $B^{\times }$ is) and
compact (as $K$ is), the set $\Gamma _{K}(x)$ is finite. For each $K\in 
\mathcal{K}$ and each set $R_{K}\subset G_{f}$ of representatives of $%
K\backslash B_{f}/B^{\times }$, define%
\begin{equation*}
T_{R_{K}}:M\left( B_{f}^{\times },R\right) ^{K}\longrightarrow R\quad \text{%
by}\quad T_{R_{K}}\left( f\right) :=\mu \left( K\right) \tsum\limits_{x\in
R_{K}}\frac{f\left( x\right) }{\left\vert \Gamma _{K}\left( x\right)
\right\vert }.
\end{equation*}%
It is easy to see that this is a well defined quantity which is independent
from the choice of $K$ (see \cite[\S 3.3.1]{GS2} for details), implying that
this family defines%
\begin{equation*}
T_{B_{f}/B^{\times }}:M\left( B_{f}^{\times },R\right) \rightarrow R\text{
and }T_{Z_{f}\backslash B_{f}/B^{\times }}:M\left( Z_{f}\backslash
B_{f}^{\times },R\right) \rightarrow R
\end{equation*}%
where $T_{B_{f}/B^{\times }}=T_{K}$ on $M\left( G_{f},R\right) ^{K}$ and $%
T_{Z_{f}\backslash B_{f}/B^{\times }}:=T_{B_{f}/B^{\times }\mid M\left(
Z_{f}\backslash B_{f}^{\times },R\right) }$.

\bigskip

Suppose that we have given a right representation $\left( V,\rho \right) $\
of $G_{\infty }$ (resp. $\Sigma _{p}$)\ with coefficients in some
commutative unitary ring $R$ and group homomorphisms $k:\mathbb{Q}^{\times
}\rightarrow R^{\times }$ (resp. $\mathbf{k}:\mathbb{Z}_{p}^{\times
}\rightarrow R^{\times }$). If%
\begin{equation*}
\Lambda \in Hom_{R\left[ B_{\infty }^{\times }\right] }\left( \rho ,R\left(
k\right) \right) \text{ (resp. }\in Hom_{R\left[ K_{p}^{\diamond }\right]
}\left( \rho ,R\left( \mathbf{k}\right) \right) \text{),}
\end{equation*}%
Then we may define the $R$-linear morphisms%
\begin{equation*}
M\left( \Lambda \right) :M\left( B_{f}^{\times },\rho ,\mathrm{N}_{f\mid
Z_{f}}^{2k}\right) \rightarrow R\text{ (resp. }M_{p}\left( \Lambda \right)
:M_{p}\left( B_{f}^{\times },\rho ,\mathrm{N}_{p\mid Z_{f}}^{2k}\right)
\rightarrow R\text{)}
\end{equation*}%
by the rule%
\begin{eqnarray*}
&&\text{ }M\left( \Lambda \right) \left( \varphi \right) :=\mu \left(
K\right) \tsum\nolimits_{x\in K\backslash B^{\times }/B^{\times }}\frac{%
\Lambda \left( \varphi \left( x\right) \right) }{\left\vert \Gamma
_{K}\left( x\right) \right\vert \mathrm{Nrd}_{f}^{k}\left( x\right) }\text{
if }\varphi \in M\left( B_{f}^{\times },\rho ,\mathrm{N}_{f\mid
Z_{f}}^{2k}\right) ^{K} \\
&&\text{ (resp. }M\left( \Lambda \right) \left( \varphi \right) :=\mu \left(
K\right) \tsum\nolimits_{x\in K\backslash B^{\times }/B^{\times }}\frac{%
\Lambda \left( \varphi \left( x\right) \right) }{\left\vert \Gamma
_{K}\left( x\right) \right\vert \mathrm{Nrd}_{p}^{\mathbf{k}}\left( x\right) 
}\text{ if }\varphi \in M_{p}\left( B_{f}^{\times },\rho ,\mathrm{N}_{p\mid
Z_{f}}^{2\mathbf{k}}\right) ^{K}\text{)}
\end{eqnarray*}%
Alternatively, we have%
\begin{eqnarray*}
&&\text{ }M\left( \Lambda \right) :M\left( B_{f}^{\times },\rho ,\mathrm{N}%
_{f\mid Z_{f}}^{2k}\right) \overset{\Lambda _{\ast }}{\rightarrow }M\left(
B_{f}^{\times },R\left( k\right) ,\mathrm{N}_{f\mid Z_{f}}^{2k}\right) 
\overset{\left\langle \cdot ,\mathrm{Nrd}_{f}^{-k}\right\rangle }{%
\rightarrow }R \\
&&\text{ (resp. }M_{p}\left( \Lambda \right) :M_{p}\left( B_{f}^{\times
},\rho ,\mathrm{N}_{p\mid Z_{f}}^{2\mathbf{k}}\right) \overset{\Lambda
_{\ast }}{\rightarrow }M_{p}\left( B_{f}^{\times },R\left( k\right) ,\mathrm{%
N}_{p\mid Z_{f}}^{2\mathbf{k}}\right) \overset{\left\langle \cdot ,\mathrm{%
Nrd}_{p}^{-\mathbf{k}}\right\rangle }{\rightarrow }R\text{),}
\end{eqnarray*}%
where:

\begin{itemize}
\item $\Lambda _{\ast }$ is the morphism induced by functoriality and $%
\Lambda $, i.e. $\Lambda _{\ast }\left( \varphi \right) \left( x\right)
:=\Lambda \left( \varphi \left( x\right) \right) $;

\item $\left\langle \cdot ,\cdot \right\rangle $ is the natural pairing%
\begin{eqnarray*}
&&\text{ }\left\langle \cdot ,\cdot \right\rangle :M\left( B_{f}^{\times
},R\left( k\right) ,\mathrm{N}_{f\mid Z_{f}}^{2k}\right) \otimes _{R}M\left(
B_{f}^{\times },R\left( -k\right) ,\mathrm{N}_{f\mid Z_{f}}^{-2k}\right) 
\overset{\otimes }{\rightarrow }M\left( Z_{f}\backslash B_{f}^{\times
},R\right) \overset{T_{Z_{f}\backslash B_{f}/B^{\times }}}{\rightarrow }R \\
&&\text{ (resp. }\left\langle \cdot ,\cdot \right\rangle :M\left(
B_{f}^{\times },R\left( k\right) ,\mathrm{N}_{f\mid Z_{f}}^{2k}\right)
\otimes _{R}M\left( B_{f}^{\times },R\left( -k\right) ,\mathrm{N}_{f\mid
Z_{f}}^{-2k}\right) \overset{\otimes }{\rightarrow }M\left( Z_{f}\backslash
B_{f}^{\times },R\right) \overset{T_{Z_{f}\backslash B_{f}/B^{\times }}}{%
\rightarrow }R\text{),}
\end{eqnarray*}%
with $\left( \varphi _{1}\otimes \varphi _{2}\right) \left( x\right)
:=\varphi _{1}\left( x\right) \varphi _{2}\left( x\right) $.
\end{itemize}

It follows from this description that the quantity is well defined. Finally,
when $\rho =\rho _{1}\otimes _{R}...\otimes _{R}\rho _{n}$ and $\omega
_{0,i} $ (resp. $\omega _{0,p,i}$) are such that%
\begin{equation}
\omega _{0,1}...\omega _{0,n}=\mathrm{N}_{f\mid Z_{f}}^{2k}\text{ (resp. }%
\omega _{0,p,1}...\omega _{0,p,n}=\mathrm{N}_{p\mid Z_{f}}^{2k}\text{),}
\label{F multilinear}
\end{equation}%
we can define the $R$-linear morphism%
\begin{eqnarray}
&&\text{ }J\left( \Lambda \right) :M\left( B_{f}^{\times },\rho _{1},\omega
_{0,1}\right) \otimes _{R}...\otimes _{R}M\left( B_{f}^{\times },\rho
_{n},\omega _{0,n}\right) \overset{\otimes }{\rightarrow }M\left(
B_{f}^{\times },\rho ,\mathrm{N}_{f\mid Z_{f}}^{2k}\right) \overset{M\left(
\Lambda \right) }{\rightarrow }R  \notag \\
&&\text{ (resp. }J_{p}\left( \Lambda \right) :M\left( B_{f}^{\times },\rho
_{1},\omega _{0,1}\right) \otimes _{R}...\otimes _{R}M\left( B_{f}^{\times
},\rho _{n},\omega _{0,n}\right) \overset{\otimes }{\rightarrow }M\left(
B_{f}^{\times },\rho ,\mathrm{N}_{f\mid Z_{f}}^{2k}\right) \overset{M\left(
\Lambda \right) }{\rightarrow }R\text{)}  \label{p-adic F Prof n-Form2}
\end{eqnarray}%
where $\otimes $ is obatined by iteration of $\left( \varphi _{1}\otimes
\varphi _{2}\right) \left( x\right) :=\varphi _{1}\left( x\right) \otimes
_{R}\varphi _{2}\left( x\right) $ in case $n=2$.

\bigskip

Let us now assume that we are in the setting of Lemma \ref{L1}, with
representations $\rho _{i,p}$ (resp. $\rho _{i,\infty }$)\ of $B_{p}^{\times
}$ (resp. $B_{\infty }^{\times }$) with coefficients in $R_{p}$ (resp. $%
R_{\infty }$) with the property that $\rho _{i}:=\rho _{i,p\mid B^{\times
}}=\rho _{i,\infty \mid B^{\times }}\subset \rho _{i,p},\rho _{i,\infty }$
with coefficients in $R$ and $\left( \text{\ref{F multilinear}}\right) $
satisfied. Furthermore, suppose that $\left( \Lambda _{p},\Lambda _{\infty
}\right) $ is a couple of elements $\Lambda _{p}\in Hom_{R_{p}\left[
B_{p}^{\times }\right] }\left( \rho _{p},R_{p}\left( k\right) \right) $ and $%
\Lambda _{\infty }\in Hom_{R_{\infty }\left[ B_{\infty }^{\times }\right]
}\left( \rho _{\infty },R_{\infty }\left( k\right) \right) $ with the
property that 
\begin{equation*}
\Lambda :=\Lambda _{p\mid \rho }=\Lambda _{\infty \mid \rho }\in Hom_{R\left[
B^{\times }\right] }\left( \rho ,R\left( k\right) \right)
\end{equation*}%
(Here we assume that $k:\mathbb{Q}^{\times }\rightarrow R$ and identify it
with $i_{p\ast }\left( \chi \right) :=i_{p}\circ \chi $ and $i_{\infty \ast
}\left( \chi \right) :=i_{\infty }\circ \chi $).

\begin{proposition}
\label{P1}Via the inclusions/identifications provided by Lemma \ref{L1}, we
have%
\begin{equation*}
J_{p}\left( \Lambda _{p}\right) _{\mid \otimes _{i=1}^{n}M\left(
B_{f}^{\times },\rho _{i},\omega _{0,i}\right) }=J\left( \Lambda \right)
_{\mid \otimes _{i=1}^{n}M\left( B_{f}^{\times },\rho _{i},\omega
_{0,i}\right) }
\end{equation*}%
on%
\begin{equation*}
\otimes _{i=1}^{n}M\left( B_{f}^{\times },\rho _{i},\omega _{0,i}\right)
\subset \otimes _{i=1}^{n}M_{p}\left( B_{f}^{\times },\rho _{i,p},\omega
_{0,p,i}\right) ,\otimes _{i=1}^{n}M\left( B_{f}^{\times },\rho _{i,\infty
},i_{\infty \ast }\left( \omega _{0,i}\right) \right) \text{.}
\end{equation*}
\end{proposition}

\begin{proof}
It is easily checked that all the canonical morphisms involved in the
definition of $J_{p}\left( \Lambda _{p}\right) $ and $J\left( \Lambda
_{\infty }\right) $ match: the non canonical ones, namely $\left\langle
\cdot ,\mathrm{Nrd}_{f}^{-k}\right\rangle $ and $\left\langle \cdot ,\mathrm{%
Nrd}_{p}^{-k}\right\rangle $ match because $\mathrm{Nrd}_{f}^{-k}$
corresponds to $\mathrm{Nrd}_{p}^{-k}$ via Lemma \ref{L1}.
\end{proof}

\subsection{\label{SSS Adjointness}Pairings and adjointness}

Suppose that $D$ (resp. $E$) is a $\Sigma _{D}$ (resp. $\Sigma _{E}$)
module, where $\Sigma _{D}$ (resp. $\Sigma _{E}$) satisfies the assumption
that was done on $\Sigma _{p}$, and we let $\omega _{0,p,D},\omega
_{0,p,E}:Z_{f}\rightarrow R^{\times }$ be characters such that $\omega
_{0,p,D}\omega _{0,p,E}=\omega _{0,p}$. We assume that we have given a group
homomorphism $\mathbf{k}:\mathbb{Z}_{p}^{\times }\rightarrow R^{\times }$\
and a pairing%
\begin{equation*}
\left\langle -,-\right\rangle \in Hom_{R\left[ K_{p}^{\diamond }\right]
}\left( D\otimes _{R}E,R\left( \mathbf{k}\right) \right) \text{.}
\end{equation*}%
Then $\left( \text{\ref{p-adic F Prof n-Form2}}\right) $ gives%
\begin{equation*}
\left\langle -,-\right\rangle _{M_{p}}:M_{p}\left( G_{f},D,\omega
_{0,p,D}\right) \otimes _{R}M_{p}\left( G_{f},E,\omega _{0,p,E}\right)
\rightarrow R\text{.}
\end{equation*}

We suppose $\Sigma _{D}=\Sigma _{p}$, $\Sigma _{D}=\Sigma _{p}^{\iota }$ and 
$\left( K_{p}^{\diamond }\right) ^{\iota }=K_{p}^{\diamond }\subset \Sigma
_{p}\cap \Sigma _{p}^{\iota }$(as in case $K_{p}^{\diamond }=\Gamma
_{0}\left( p\mathbb{Z}_{p}\right) $ and $\Sigma _{p}=\Sigma _{0}\left( p%
\mathbb{Z}_{p}\right) $) and that $\mathbf{Z}=\mathbf{Z}_{\mathbf{B}^{\times
}}=\mathbf{G}_{m}$. Assuming that $E$ has central character $\kappa _{E}:%
\mathbb{Z}_{p}^{\times }\rightarrow R^{\times }$, we can consider the second
of the following compositions:%
\begin{equation*}
\mathrm{nrd}_{f}^{\omega _{0,p,E}}:B_{f}^{\times }\overset{\mathrm{nrd}_{f}}{%
\rightarrow }\mathbb{A}_{f}^{\times }=Z_{f}\overset{\omega _{0,p,E}}{%
\rightarrow }R^{\times }\text{ and }\mathrm{nrd}_{p}^{\kappa
_{E}}:K_{p}^{\diamond }\overset{\mathrm{nrd}_{p}}{\rightarrow }\mathbb{Z}%
_{p}^{\times }\overset{\kappa _{E}}{\rightarrow }R^{\times }
\end{equation*}

Suppose that $\mathbf{k},\kappa _{E}:\mathbb{Z}_{p}^{\times }\rightarrow
R^{\times }$ extends to a character $\widetilde{\mathbf{k}},\widetilde{%
\kappa _{E}}:\mathbb{Q}_{p}^{\times }\rightarrow R^{\times }$. Then%
\begin{equation*}
\mathrm{nrd}_{p}^{\widetilde{\kappa _{E}}}:K_{p}^{\diamond }\overset{\mathrm{%
nrd}_{p}}{\rightarrow }\mathbb{Q}_{p}^{\times }\overset{\widetilde{\kappa
_{E}}}{\rightarrow }R^{\times }
\end{equation*}%
is an extension of $\mathrm{nrd}_{p}^{\kappa _{E}}$ to $\Sigma _{p}$ and we
let $Hom_{R\left[ \Sigma _{p},\Sigma _{p}^{\iota }\right] }\left( D\otimes
E,R\left( \widetilde{\mathbf{k}}\right) \right) $ be the set of those
pairings such that%
\begin{equation*}
\left\langle v\sigma ,w\right\rangle =\mathrm{nrd}_{p}^{\widetilde{\mathbf{k}%
}-\widetilde{\kappa _{E}}}\left( \sigma \right) \left\langle v,w\sigma
^{\iota }\right\rangle \text{ for every }\sigma \in \Sigma _{p}\text{.}
\end{equation*}%
We remark that, for every element $u\in K_{p}^{\diamond }$,%
\begin{equation*}
\left\langle vu,wu\right\rangle =\mathrm{nrd}_{p}^{\mathbf{k}-\kappa
_{E}}\left( u\right) \left\langle v,wuu^{\iota }\right\rangle =\mathrm{nrd}%
_{p}^{\mathbf{k}}\left( u\right) \left\langle v,w\right\rangle \text{,}
\end{equation*}%
so that $Hom_{R\left[ \Sigma _{p},\Sigma _{p}^{\iota }\right] }\left(
D\otimes E,R\left( \widetilde{\mathbf{k}}\right) \right) \subset Hom_{R\left[
K_{p}^{\diamond }\right] }\left( D\otimes E,R\left( \mathbf{k}\right)
\right) $.

\begin{remark}
\label{R1}Suppose now that $D\subset \widetilde{D}$ and $E\subset \widetilde{%
E}$, where $\widetilde{D}$ and $\widetilde{E}$ are $B_{p}^{\times }$%
-modules, the above inclusions are $\Sigma _{p}$ and, respectively, $\Sigma
_{p}^{\iota }$-equivariant and that $\widetilde{E}$ has central character $%
\kappa _{\widetilde{E}}=\widetilde{\kappa _{E}}$ extending $\kappa _{E}$. If 
$\left\langle \cdot ,\cdot \right\rangle \in Hom_{\mathcal{O}\left[
K_{p}^{\diamond }\right] }\left( D\otimes E,R\left( \mathbf{k}\right)
\right) $ extends to $\left\langle \cdot ,\cdot \right\rangle ^{\sim }\in
Hom_{\mathcal{O}\left[ B_{p}^{\times }\right] }\left( \widetilde{D}\otimes 
\widetilde{E},R\left( \widetilde{\mathbf{k}}\right) \right) $ then $%
\left\langle \cdot ,\cdot \right\rangle \in Hom_{\mathcal{O}\left[ \Sigma
_{p},\Sigma _{p}^{\iota }\right] }\left( D\otimes E,R\left( \widetilde{%
\mathbf{k}}\right) \right) $:%
\begin{equation*}
\left\langle v\sigma ,w\right\rangle =\left\langle v\sigma ,w\sigma
^{-1}\sigma \right\rangle ^{\sim }=\mathrm{nrd}_{p}^{\widetilde{\mathbf{k}}%
}\left( \sigma \right) \left\langle v,w\sigma ^{-1}\right\rangle ^{\sim }=%
\mathrm{nrd}_{p}^{\widetilde{\mathbf{k}}-\widetilde{\kappa _{E}}}\left(
\sigma \right) \left\langle v,w\sigma ^{\iota }\right\rangle \text{.}
\end{equation*}
\end{remark}

\bigskip

In the following proposition, we suppose that we are placed in the setting
pictured above and that $\pi $ is contentrated in $\pi _{p}\in \Sigma _{p}$.
Also, we suppose that $f\in M_{p}\left( G_{f},D,\omega _{0,p,D}\right)
^{K_{p}^{\diamond }}$ and $g\in M_{p}\left( G_{f},E,\omega _{0,p,E}\right)
^{K_{p}^{\diamond }}$ (and make a similar assumption for classical modular
forms $M$s to which the statement make reference). Finally, we assume that%
\begin{equation*}
\left\langle -,-\right\rangle \in Hom_{R\left[ \Sigma _{p},\Sigma
_{p}^{\iota }\right] }\left( D\otimes E,R\left( \widetilde{\mathbf{k}}%
\right) \right)
\end{equation*}%
(but for classical modular forms, we allow $\left\langle -,-\right\rangle
\in Hom_{R\left[ B_{p}^{\times }\right] }\left( D\otimes _{R}E,R\left( 
\mathbf{k}\right) \right) $ for $\mathbf{k}=$\ and does not require $E$ to
have central character $\kappa _{E}:\mathbb{Q}^{\times }\rightarrow
R^{\times }$)

\begin{proposition}
\label{P:heckeadjointprop}We have the following formulas, in the $p$-adic
case:%
\begin{equation*}
\left\langle f\mid T_{\pi },g\right\rangle =\mathrm{Nrd}_{f}^{\widetilde{%
\mathbf{k}}}\left( \pi \right) _{p}\mathrm{nrd}_{p}^{-\widetilde{\kappa _{E}}%
}\left( \pi _{p}\right) \mathrm{nrd}_{f}^{-\omega _{0,p,E}}\left( \pi
\right) \left\langle f,g\mid T_{\pi ^{\iota }}\right\rangle \text{.}
\end{equation*}%
For classical modular foems, $\left\langle f\mid T_{\pi },g\right\rangle =%
\mathrm{Nrd}_{f}^{\widetilde{\mathbf{k}}}\left( \pi \right) \left\langle
f,g\mid T_{\pi ^{-1}}\right\rangle $ and, whenever $E$ has central character 
$\kappa _{E}$, $\left\langle f\mid T_{\pi },g\right\rangle =\mathrm{Nrd}%
_{f}^{\widetilde{\mathbf{k}}}\left( \pi \right) \mathrm{nrd}_{f}^{-\kappa
_{E}}\left( \pi \right) \left\langle f,g\mid T_{\pi ^{\iota }}\right\rangle $%
.
\end{proposition}

\begin{proof}
We have $T_{\pi }=\left[ K_{1}\pi K_{2}\right] $ with $K_{1}=K_{2}\in 
\mathcal{K}^{\diamond \diamond }$\ and we leave to the reader to check that
we may assume that $\left\vert \Gamma _{K_{i}}\left( x\right) \right\vert =1$
for all $x\in B_{f}^{\times }$ and $i=1,2$ and that we have $\mathrm{Nrd}%
_{p}^{\mathbf{k}}\left( x\right) =\mathrm{Nrd}_{p}^{\mathbf{k}}\left(
ux\right) \mathrm{nrd}_{p}^{\mathbf{k}}\left( u_{p}\right) $ for all $u\in
K_{i}$ and $i=1,2$: this is possible thanks to our assumptions that $\pi \in
\Sigma _{p}$ and that $f$, $g$ and $\mathrm{Nrd}_{p}$ are fixed by $%
K_{p}^{\diamond }$ (because $K_{i}\in \mathcal{K}^{\diamond }$). Having made
this reduction, we compute, for $p$-adic modular forms,%
\begin{eqnarray*}
\mu \left( K_{2}\right) ^{-1}\left\langle f\mid \left[ K_{1}\pi K_{2}\right]
,g\right\rangle &=&\sum\nolimits_{x\in K_{2}\backslash B_{f}^{\times
}/B^{\times }}\frac{\left\langle \left( f\mid \left[ K_{1}\pi K_{2}\right]
\right) \left( x\right) ,g\left( x\right) \right\rangle }{\mathrm{Nrd}_{p}^{%
\mathbf{k}}\left( x\right) } \\
&=&\sum\nolimits_{u\in \left( K_{2}\cap \pi ^{-1}K_{1}\pi \right) \backslash
K_{2},x\in K_{2}\backslash B_{f}^{\times }/B^{\times }}\frac{\left\langle
f\left( \pi ux\right) \pi _{p}u_{p},g\left( x\right) \right\rangle }{\mathrm{%
Nrd}_{p}^{\mathbf{k}}\left( x\right) } \\
&=&\sum\nolimits_{u\in \left( K_{2}\cap \pi ^{-1}K_{1}\pi \right) \backslash
K_{2},x\in K_{2}\backslash B_{f}^{\times }/B^{\times }}\frac{\left\langle
f\left( \pi ux\right) \pi _{p}u_{p},g\left( ux\right) u_{p}\right\rangle }{%
\mathrm{Nrd}_{p}^{\mathbf{k}}\left( ux\right) \chi _{p}\left( u_{p}\right) }
\\
&=&\sum\nolimits_{u\in \left( K_{2}\cap \pi ^{-1}K_{1}\pi \right) \backslash
K_{2},x\in K_{2}\backslash B_{f}^{\times }/B^{\times }}\frac{\left\langle
f\left( \pi ux\right) \pi _{p},g\left( ux\right) \right\rangle }{\mathrm{Nrd}%
_{p}^{\mathbf{k}}\left( ux\right) } \\
&=&\sum\nolimits_{y\in \left( K_{2}\cap \pi ^{-1}K_{1}\pi \right) \backslash
B_{f}^{\times }/B^{\times }}\frac{\left\langle f\left( \pi y\right) \pi
_{p},g\left( y\right) \right\rangle }{\mathrm{Nrd}_{p}^{\mathbf{k}}\left(
y\right) }\text{.}
\end{eqnarray*}%
Here we have employed $\left( \text{\ref{F Hecke Operator}}\right) $ in the
second equality, the $K_{2}$-invariance of $g$ and $\mathrm{Nrd}_{p}^{%
\mathbf{k}}$ in the third equality and the $K_{p}^{\diamond }$-equivariance
of $\left\langle -,-\right\rangle $ in the fourth equality. Letting $g$, $f$%
, $K_{2}$, $K_{1}$ and $\pi ^{\iota }$ play the roles of $f$, $g$, $K_{1}$, $%
K_{2}$ and $\pi $ respectively, we also see that%
\begin{equation*}
\mu \left( K_{1}\right) ^{-1}\left\langle f,g\mid \left[ K_{2}\pi ^{\iota
}K_{1}\right] \right\rangle =\sum\nolimits_{z\in \left( K_{1}\cap \pi
^{-\iota }K_{2}\pi ^{\iota }\right) \backslash B_{f}^{\times }/B^{\times }}%
\frac{\left\langle f\left( z\right) ,g\left( \pi ^{\iota }z\right) \pi
_{p}^{\iota }\right\rangle }{\mathrm{Nrd}_{p}^{\mathbf{k}}\left( z\right) }%
\text{.}
\end{equation*}%
Note, however, that $y\mapsto \pi y$ induces a well defined map $H\backslash
B_{f}^{\times }/B^{\times }\rightarrow \pi H\pi ^{-1}\backslash
B_{f}^{\times }/B^{\times }$ for any subgroup $H$ and we have $\pi ^{\iota
}H\pi ^{-\iota }=\pi ^{-1}H\pi $. Taking $H=K_{2}\cap \pi ^{-1}K_{1}\pi $ we
see that $\pi H\pi ^{-1}=K_{1}\cap \pi ^{-\iota }K_{2}\pi ^{\iota }$. Making
the change of variables $z=\pi y$, we have%
\begin{eqnarray*}
\mu \left( K_{1}\right) ^{-1}\left\langle f,g\mid \left[ K_{2}\pi ^{\iota
}K_{1}\right] \right\rangle &=&\sum\nolimits_{y\in \left( K_{2}\cap \pi
^{-1}K_{1}\pi \right) \backslash B_{f}^{\times }/B^{\times }}\frac{%
\left\langle f\left( \pi y\right) ,g\left( \pi ^{\iota }\pi y\right) \pi
_{p}^{\iota }\right\rangle }{\mathrm{Nrd}_{p}^{\mathbf{k}}\left( \pi
y\right) } \\
&=&\mathrm{nrd}_{p}^{\widetilde{\kappa _{E}}-\widetilde{\mathbf{k}}}\left(
\pi _{p}\right) \mathrm{Nrd}_{p}^{-\mathbf{k}}\left( \pi \right)
\sum\nolimits_{y\in \left( K_{2}\cap \pi ^{-1}K_{1}\pi \right) \backslash
B_{f}^{\times }/B^{\times }}\frac{\left\langle f\left( \pi y\right) \pi
_{p},g\left( \pi ^{\iota }\pi y\right) \right\rangle }{\mathrm{Nrd}_{p}^{%
\mathbf{k}}\left( y\right) }\text{.}
\end{eqnarray*}%
Here we have used $\left\langle v,w\pi _{p}^{\iota }\right\rangle =\mathrm{%
nrd}_{p}^{\widetilde{\kappa _{E}}-\widetilde{\mathbf{k}}}\left( \pi
_{p}\right) \left\langle v\pi _{p},w\right\rangle $. We now remark that $\pi
^{\iota }\pi =\mathrm{nrd}\left( \pi \right) \in Z_{f}$, so that $g\left(
\pi ^{\iota }\pi y\right) =\mathrm{nrd}_{f}^{\omega _{0,p,E}}\left( \pi
\right) g\left( y\right) $. It follows that%
\begin{equation*}
\mu \left( K_{1}\right) ^{-1}\left\langle f,g\mid \left[ K_{2}\pi ^{\iota
}K_{1}\right] \right\rangle =\mathrm{nrd}_{p}^{\widetilde{\kappa _{E}}-%
\widetilde{\mathbf{k}}}\left( \pi _{p}\right) \mathrm{Nrd}_{p}^{-\mathbf{k}%
}\left( \pi \right) \mathrm{nrd}_{f}^{\omega _{0,p,E}}\left( \pi \right) \mu
\left( K_{2}\right) ^{-1}\left\langle f\mid \left[ K_{1}\pi K_{2}\right]
,g\right\rangle \text{.}
\end{equation*}%
The relation $\mathrm{Nrd}_{p}^{\kappa }\left( x\right) =\left( \frac{%
\mathrm{Nrd}_{f}\left( x\right) _{p}}{\mathrm{nrd}_{p}\left( x_{p}\right) }%
\right) ^{\kappa }$ gives the claim.%
\begin{equation*}
\mathrm{nrd}_{p}^{\widetilde{\kappa _{E}}-\widetilde{\mathbf{k}}}\left( \pi
_{p}\right) \mathrm{Nrd}_{p}^{-\mathbf{k}}\left( \pi \right) =\mathrm{nrd}%
_{p}^{\widetilde{\kappa _{E}}-\widetilde{\mathbf{k}}}\left( \pi _{p}\right) 
\frac{\mathrm{Nrd}_{f}^{-\mathbf{k}}\left( \pi \right) _{p}}{\mathrm{nrd}%
_{p}^{-\mathbf{k}}\left( \pi _{p}\right) }=\mathrm{nrd}_{p}^{\widetilde{%
\kappa _{E}}}\left( \pi _{p}\right) \mathrm{Nrd}_{f}^{-\widetilde{\mathbf{k}}%
}\left( \pi \right) _{p}
\end{equation*}

For modular forms one finds, by a similar computation,%
\begin{eqnarray*}
&&\mu \left( K_{2}\right) ^{-1}\left\langle f\mid \left[ K_{1}\pi K_{2}%
\right] ,g\right\rangle _{K_{2}}=\sum\nolimits_{y\in \left( K_{2}\cap \pi
^{-1}K_{1}\pi \right) \backslash G_{f}/\Gamma }\frac{\left\langle f\left(
\pi y\right) ,g\left( y\right) \right\rangle }{\mathrm{Nrd}_{f}^{\mathbf{k}%
}\left( y\right) }\text{,} \\
&&\mu \left( K_{1}\right) ^{-1}\left\langle f,g\mid \left[ K_{2}\pi
^{-1}K_{1}\right] \right\rangle _{K_{1}}=\sum\nolimits_{z\in \left(
K_{1}\cap \pi K_{2}\pi ^{-1}\right) \backslash G_{f}/\Gamma }\frac{%
\left\langle f\left( z\right) ,g\left( \pi z\right) \right\rangle }{\mathrm{%
Nrd}_{f}^{\mathbf{k}}\left( z\right) }\text{.}
\end{eqnarray*}%
The first equality in this case follows and the second, which can also be
proved by a similar computation as above, is actually a consequence of the
first in this setting, since one checks $g\mid \left[ K_{2}\pi ^{-1}K_{1}%
\right] =\mathrm{nrd}_{f}^{-\kappa _{E}}\left( \pi \right) \cdot g\mid \left[
K_{2}\pi ^{\iota }K_{1}\right] $ (because $\mathrm{nrd}\left( \pi \right)
\in Z_{f}=Z_{B_{f}^{\times }}$).
\end{proof}

\section{\label{SSIchino}\label{SS p-adic definite quaternion}The special
value formula and its $p$-adic avatar}

We are now going to recall the special value formula proved in \cite{GS2},
specialized to the triple product case, which can be regarded as an explicit
version of Ichino's formula \cite{Ich} and a generalization of \cite{BoeSch}.

Let $E/\mathbb{Q}$ be a Galois splitting field for $B$ and fix $\mathbf{B}%
_{/E}\simeq \mathbf{M}_{2/E}$ inducing $\mathbf{B}_{/E}^{\mathbf{\times }%
}\simeq \mathbf{GL}_{2/E}$. If $k\in \mathbb{N}$ we let $\mathbf{P}_{k/E}$
be the left $\mathbf{GL}_{2/E}$-representation on two variables polynomials
of degree $k$, the action being defined by the rule $\left( gP\right) \left(
X,Y\right) =P\left( \left( X,Y\right) g\right) $. We write $\mathbf{V}_{k}$
for the dual right representation. If $\underline{k}:=\left(
k_{1},...,k_{r}\right) \in \mathbb{N}^{r}$, we may identify $\mathbf{P}%
_{k_{1}/E}\otimes ...\otimes \mathbf{P}_{k_{r}/E}$ with the space of $2r$%
-variable polynomials $\mathbf{P}_{\underline{k}/E}$ which are homogeneous
of degree $k_{i}$ in the $i$-th couple of variables $W_{i}:=\left(
X_{i},Y_{i}\right) $. Then $\mathbf{V}_{k_{1}/E}\otimes ...\otimes \mathbf{V}%
_{k_{r}/E}$ is identified with the dual $\mathbf{V}_{\underline{k}/E}$ of $%
\mathbf{P}_{\underline{k}/E}$ and any $P\in \mathbf{P}_{\underline{k}%
/E}\left( -r\right) ^{\mathbf{GL}_{2/E}}$, i.e. such that $gP=\mathrm{det}%
\left( g\right) ^{r}P$, induces%
\begin{equation*}
\Lambda _{P}\in Hom_{\mathbf{GL}_{2/E}}\left( \mathbf{V}_{\underline{k}/E},%
\mathbf{1}_{/E}\left( r\right) \right)
\end{equation*}%
by the rule $\Lambda _{P}\left( l\right) :=l\left( P\right) $. Note also
that, if $P\neq 0$ then there is $l$ such that $l\left( P\right) =1$ and we
see that $\Lambda _{P}\neq 0$. Setting $0\neq \delta ^{k}\left(
X_{1},Y_{1},X_{2},Y_{2}\right) :=\left\vert 
\begin{array}{cc}
X_{1} & Y_{1} \\ 
X_{2} & Y_{2}%
\end{array}%
\right\vert ^{k}$, we have $\delta ^{1}\left( W_{1}g,W_{2}g\right) =\mathrm{%
det}\left( g\right) \delta ^{1}\left( W_{1},W_{2}\right) $, from which it
follows that $\delta _{k}\in \mathbf{P}_{k,k/E}$ and $g\delta ^{k}=\mathrm{%
det}\left( g\right) ^{k}\delta ^{k}$. We deduce that $\left\langle
-,-\right\rangle _{k}:=\Lambda _{\delta ^{k}}\neq 0$ satisfies the above
requirement: then the irreducibility of $\mathbf{V}_{k/E}$\ implies that it
is perfect and symmetric.

If $\underline{k}:=\left( k_{1},k_{2},k_{3}\right) \in \mathbb{N}^{3}$, we
define the quantities $\underline{k}^{\ast }:=\frac{k_{1}+k_{2}+k_{3}}{2}$, $%
\underline{k}_{1}^{\ast }:=\frac{-k_{1}+k_{2}+k_{3}}{2}$, $\underline{k}%
_{2}^{\ast }:=\frac{k_{1}-k_{2}+k_{3}}{2}$ and $\underline{k}_{3}^{\ast }:=%
\frac{k_{1}+k_{2}-k_{3}}{2}$. With a slight abuse of notation, we write $%
\mathbf{P}_{\underline{k}/E}$ and $\mathbf{V}_{\underline{k}/E}$ to denote
the external tensor product, which is a representation of $\mathbf{GL}%
_{2/E}^{3}$. When $\underline{k}^{\ast }\in \mathbb{N}$ and $\underline{k}$
is balanced, we can also define%
\begin{equation*}
\Lambda _{\underline{k}/E}\in Hom_{\mathbf{GL}_{2/E}}\left( \mathbf{V}_{%
\underline{k}/E},\mathbf{1}_{/E}\left( \underline{k}^{\ast }\right) \right)
\end{equation*}%
as follows. The balanced condition precisely means that $\underline{k}%
_{i}^{\ast }\geq 0$ for $i=1,2,3$, so that we can consider%
\begin{equation*}
0\neq \Delta _{\underline{k}/E}:=\delta ^{\underline{k}_{1}^{\ast }}\left(
W_{2},W_{3}\right) \delta ^{\underline{k}_{2}^{\ast }}\left(
W_{1},W_{3}\right) \delta ^{\underline{k}_{3}^{\ast }}\left(
W_{1},W_{2}\right) \in \mathbf{P}_{\underline{k}/E}\text{.}
\end{equation*}%
We have $g\Delta _{\underline{k}/E}=\mathrm{det}\left( g\right) ^{\underline{%
k}^{\ast }}\Delta _{\underline{k}/E}$. Hence $\Delta _{\underline{k}/E}\in 
\mathbf{P}_{\underline{k}/E}\left( -\underline{k}^{\ast }\right) ^{\mathbf{GL%
}_{2/E}}$ and we may set $\Lambda _{\underline{k}/E}:=\Lambda _{\Delta _{%
\underline{k}/E}}\neq 0$. The following result is an application of the
Clebsch-Gordan decomposition that we leave to the reader.

\begin{lemma}
\label{Special value L7}Suppose that $2\underline{k}^{\ast
}=k_{1}+k_{2}+k_{3}\in 2\mathbb{N}$ and $\underline{k}$ is balanced.

\begin{itemize}
\item[$\left( 1\right) $] There is a representation $\mathbf{V}_{\underline{k%
}}$\ of $\mathbf{B}^{\times 3}\ $such that $E\otimes \mathbf{V}_{\underline{k%
}}\simeq \mathbf{V}_{\underline{k}/E}$ via $\mathbf{B}_{/E}^{\times 3}\simeq 
\mathbf{GL}_{2/E}^{3}$ and $\left\langle -,-\right\rangle _{\underline{k}%
}\in Hom_{\mathbf{B}^{\times }}\left( \mathbf{V}_{\underline{k}}\otimes 
\mathbf{V}_{\underline{k}},\mathbf{1}\left( \underline{k}\right) \right) $
such that $E\otimes \left\langle -,-\right\rangle _{\underline{k}}\simeq
\left\langle -,-\right\rangle _{\underline{k}/E}$

\item[$\left( 2\right) $] We have, setting $\mathbf{B}_{1}^{\times }:=%
\mathrm{ker}\left( \mathrm{nrd}\right) $,%
\begin{equation*}
\mathrm{dim}\left( Hom_{\mathbf{B}_{1}^{\times }}\left( \mathbf{V}_{%
\underline{k}},\mathbf{1}\right) \right) =\mathrm{dim}\left( Hom_{\mathbf{SL}%
_{2/E}}\left( \mathbf{V}_{\underline{k}/E},\mathbf{1}_{/E}\right) \right) =1%
\text{.}
\end{equation*}
\end{itemize}
\end{lemma}

For $i=1,2,3$, let $\omega _{i}$ be an unitary Hecke character of the form $%
\omega _{i}=\omega _{f,i}\otimes sgn\left( -\right) ^{k_{i}}$ and set $%
\omega _{0,i}:=\omega _{f,i}\mathrm{Nrd}_{f}^{k_{i}/2}$. Assuming that $%
\omega _{1}\omega _{2}\omega _{3}=1$, we see that%
\begin{equation}
\underline{k}^{\ast }\in \mathbb{N}\text{ and }\mathrm{Nrd}_{f}^{\underline{k%
}^{\ast }}=\omega _{1,0}\omega _{2,0}\omega _{3,0}\text{.}
\label{F balanced}
\end{equation}%
It follows from an adelic version of the Peter-Weyl Theorem (see \cite{GS2})
that, if $\pi _{i}=\pi _{i,f}\otimes \mathbf{V}_{k_{i},\mathbb{C}}^{u}$ is
an irreducible unitary automorphic form with central character $\omega _{i}$%
\ (and $\mathbf{V}_{k_{i},\mathbb{C}}^{u}$ the unitary twist of $\mathbf{V}%
_{k_{i},\mathbb{C}}$),\ there is a canonical identification%
\begin{equation*}
f_{i}:\mathbf{V}_{k_{i},\mathbb{C}}^{\vee }\otimes _{\mathbb{C}}M\left(
B_{f}^{\times },\mathbf{V}_{k_{i},\mathbb{C}},\omega _{0,i}\right) \left[ 
\mathrm{Nrd}_{f}^{-k_{i}/2}\pi _{i,f}\right] \simeq A\left( \mathbf{B}%
^{\times }\left( \mathbb{A}\right) ,\omega _{i}\right) \left[ \pi _{i}\right]
\text{,}
\end{equation*}%
where $\left( -\right) \left[ \theta \right] $ means taking the $\theta $%
-component and $A\left( \mathbf{B}^{\times }\left( \mathbb{A}\right) ,\omega
_{i}\right) $ is the space of $K$-finite automorphic forms. We remark the we
could have considered automorphic forms for the algebraic group $\mathbf{B}%
^{\times 3}$ and, with $\Pi :=\pi _{1}\otimes \pi _{2}\otimes \pi _{3}$, so
that $\Pi =\Pi _{f}\otimes \mathbf{V}_{\underline{k},\mathbb{C}}^{u}$, we
have%
\begin{equation*}
f:\mathbf{V}_{\underline{k},\mathbb{C}}^{\vee }\otimes _{\mathbb{C}}M\left(
B_{f}^{\times 3},\mathbf{V}_{\underline{k},\mathbb{C}},\omega _{0}\right) %
\left[ \mathrm{Nrd}_{f}^{-\underline{k}/2}\Pi _{f}\right] \simeq A\left( 
\mathbf{B}^{\times 3}\left( \mathbb{A}\right) ,\omega \right) \left[ \Pi %
\right] \text{,}
\end{equation*}%
where $\omega =\left( \omega _{1},\omega _{2},\omega _{3}\right) $, $\mathrm{%
Nrd}_{f}^{-\underline{k}/2}\left( x_{1},x_{2},x_{3}\right) :=\mathrm{Nrd}%
_{f}^{-k_{1}/2}\left( x_{1}\right) \mathrm{Nrd}_{f}^{-k_{2}/2}\left(
x_{2}\right) \mathrm{Nrd}_{f}^{-k_{3}/2}\left( x_{3}\right) $ and $\omega
_{0}:=\omega _{f}\mathrm{Nrd}_{f}^{\underline{k}/2}$.

It follows from $\left( \text{\ref{F balanced}}\right) $\ that we can
consider the quantity $t_{\underline{k}}:=J\left( \Lambda _{\underline{k}%
/E}\right) $ defined by $\left( \text{\ref{p-adic F Prof n-Form2}}\right) $:%
\begin{equation*}
t_{\underline{k}}:M\left( B_{f}^{\times 3},\mathbf{V}_{\underline{k}%
,E},\omega _{0}\right) =M\left( B_{f}^{\times },\mathbf{V}_{k_{1},E},\omega
_{0,1}\right) \otimes _{E}M\left( B_{f}^{\times },\mathbf{V}%
_{k_{2},E},\omega _{0,2}\right) \otimes _{E}M\left( B_{f}^{\times },\mathbf{V%
}_{k_{3},E},\omega _{0,3}\right) \rightarrow E\text{.}
\end{equation*}%
The choice of $\Lambda _{\underline{k}/E}\in \mathbf{V}_{\underline{k}%
,E}^{\vee }$ yields%
\begin{equation*}
f_{\Lambda _{\underline{k}/E}}:M\left( B_{f}^{\times 3},\mathbf{V}_{%
\underline{k},E},\omega _{0}\right) \left[ \mathrm{Nrd}_{f}^{-\underline{k}%
/2}\Pi _{f}\right] \hookrightarrow A\left( \mathbf{B}^{\times 3}\left( 
\mathbb{A}\right) ,\omega \right) \left[ \Pi \right] \text{.}
\end{equation*}%
The following result is deduced in \cite{GS2} from \cite{Ich} or \cite{HK}\
and the Jacquet conjecture proved in \cite{HK}.

\begin{theorem}
\label{Special value T Ichino}Suppose that $\underline{k}$ is balanced and
that $\omega _{i}=\omega _{i,f}\otimes sgn\left( -\right) ^{k_{i}}$ are
unitary Hecke characters such that $\omega _{1}\omega _{2}\omega _{3}=1$,
implying $\underline{k}^{\ast }\in \mathbb{N}$. Consider the quantity%
\begin{equation*}
t_{\underline{k}}\left( \varphi \right) =\mu \left( K_{\varphi }\right)
\tsum\limits_{x\in K_{\varphi }\backslash B_{f}^{\times }/B^{\times }}\frac{%
\Lambda _{\underline{k}}\left( \varphi _{1}\left( x\right) \otimes \varphi
_{2}\left( x\right) \otimes \varphi _{3}\left( x\right) \right) }{\left\vert
\Gamma _{K_{\varphi }}\left( x\right) \right\vert \mathrm{Nrd}_{f}^{%
\underline{k}^{\ast }}\left( x\right) }\text{,}
\end{equation*}%
where $K_{\varphi }\in \mathcal{K}$ is such that $K_{\varphi }\subset
K_{\varphi _{1}}\cap K_{\varphi _{2}}\cap K_{\varphi _{3}}$ and%
\begin{equation*}
\varphi =\varphi _{1}\otimes \varphi _{2}\otimes \varphi _{3}\in
\tbigotimes\nolimits_{i=1}^{3}M\left( B_{f}^{\times },\mathbf{V}_{k_{i},%
\mathbb{C}},\omega _{i,0}\right) ^{K_{\varphi _{i}}}=M\left( B_{f}^{\times
3},\mathbf{V}_{\underline{k},E},\omega _{0}\right) ^{K_{\varphi _{1}}\times
K_{\varphi _{2}}\times K_{\varphi _{3}}}\text{.}
\end{equation*}%
Suppose that $B=B_{\Pi }$ is the quaternion algebra predicted by \cite{Pr}.

\begin{itemize}
\item[$\left( 1\right) $] We have the equality%
\begin{equation*}
t_{\underline{k}}^{2}=\frac{1}{2^{3}m_{\mathbf{Z}_{\mathbf{B}}\mathbf{%
\backslash B},\infty }^{2}}\frac{\zeta _{\mathbb{Q}}^{2}\left( 2\right)
L\left( 1/2,\Pi \right) }{L\left( 1,\Pi ,\mathrm{Ad}\right) }%
\prod\nolimits_{v}\alpha _{v}\left( -\right)
\end{equation*}%
as functionals on%
\begin{equation*}
f_{\Lambda _{\underline{k}/E}}:M\left( B_{f}^{\times 3},\mathbf{V}_{%
\underline{k},E},\omega _{0}\right) \left[ \mathrm{Nrd}_{f}^{-\underline{k}%
/2}\Pi _{f}\right] \hookrightarrow A\left( \mathbf{B}^{\times 3}\left( 
\mathbb{A}\right) ,\omega \right) \left[ \Pi \right] \text{.}
\end{equation*}%
Here the quantities appearing in right hand side have a similar nature as
those in $\left( \text{\ref{Intro F special value}}\right) $ (see \cite{Ich}%
).

\item[$\left( 2\right) $] Define $\widetilde{t}_{\underline{k}}\left(
\Lambda \otimes _{\mathbb{C}}\varphi \right) :=\lambda t_{\underline{k}%
}\left( \varphi \right) $ when $\Lambda =\lambda \Lambda _{\underline{k}}$
and $\widetilde{t}_{\underline{k}}\left( \Lambda \otimes _{\mathbb{C}%
}\varphi \right) :=0$ for $\Lambda $ orthogonal to $\Lambda _{\underline{k}}$
in $\mathbf{V}_{\underline{k},\mathbb{C}}^{\vee }$ (with respect to $%
\left\langle -,-\right\rangle _{\underline{k}}$), we have%
\begin{equation*}
\widetilde{t}_{\underline{k}}^{2}=\frac{1}{2^{3}m_{\mathbf{Z}_{\mathbf{B}}%
\mathbf{\backslash B},\infty }^{2}}\frac{\zeta _{\mathbb{Q}}^{2}\left(
2\right) L\left( 1/2,\Pi \right) }{L\left( 1,\Pi ,\mathrm{Ad}\right) }%
\prod\nolimits_{v}\alpha _{v}\left( -\right)
\end{equation*}%
as functionals on%
\begin{equation*}
f:\mathbf{V}_{\underline{k},\mathbb{C}}^{\vee }\otimes _{\mathbb{C}}M\left(
B_{f}^{\times 3},\mathbf{V}_{\underline{k},\mathbb{C}},\omega _{0}\right) %
\left[ \mathrm{Nrd}_{f}^{-\underline{k}/2}\Pi _{f}\right] \simeq A\left( 
\mathbf{B}^{\times 3}\left( \mathbb{A}\right) ,\omega \right) \left[ \Pi %
\right]
\end{equation*}%
and this rule extends to a morphism of functors from modular forms with
coefficients in $\mathbb{Q}\left( \omega _{f}\right) $-algebras to $\mathbf{A%
}^{1}$.

\item[$\left( 3\right) $] Suppose that $\Pi ^{\prime }$ is an automorphic
representation of $\mathbf{GL}_{2}^{3}$ and that $B=B_{\Pi ^{\prime }}$ is
the quaternion algebra predicted by \cite{Pr}. Then $L\left( \Pi ^{\prime
},1/2\right) \neq 0$ if and only if $t_{\underline{k}}\neq 0$ on $M\left(
B_{f}^{\times 3},\mathbf{V}_{\underline{k},E},\omega _{0}\right) \left[ 
\mathrm{Nrd}_{f}^{-\underline{k}/2}\Pi _{f}\right] $ with $\Pi =\Pi
_{f}\otimes \mathbf{V}_{\underline{k},\mathbb{C}}^{u}$ corresponding to $\Pi
^{\prime }$ by the Jacquet-Langlands correspondence (hence $B=B_{\Pi }$).
\end{itemize}
\end{theorem}

\bigskip

Regarding $\widetilde{t}_{\underline{k}}^{2}$ as the algebraic part of $%
L\left( 1/2,\Pi \right) $ (see \cite{GS2} for a justification), it follows
from Theorem \ref{Special value T Ichino} that the relevant part to be
interpolated is $t_{\underline{k}}$. Applying Proposition \ref{P1}, we place 
$t_{\underline{k}}$ in a $p$-adic setting, making it correspond to $t_{%
\underline{k}}:=J_{p}\left( \Lambda _{\underline{k}/\mathbb{Q}_{p}}\right) $%
\begin{equation}
t_{\underline{k}}\left( \varphi _{1},\varphi _{2},\varphi _{3}\right) =\mu
\left( K_{\varphi }\right) \tsum\limits_{x\in K_{\varphi }\backslash \mathbf{%
B}\left( \mathbb{A}_{f}\right) /\mathbf{B}\left( F\right) }\frac{\Lambda _{%
\underline{k}/\mathbb{Q}_{p}}\left( \varphi _{1}\left( x\right) \otimes
\varphi _{2}\left( x\right) \otimes \varphi _{3}\left( x\right) \right) }{%
\left\vert \Gamma _{K_{\varphi }}\left( x\right) \right\vert \mathrm{Nrd}%
_{p}^{\underline{k}^{\ast }}\left( x\right) }\text{.}
\label{F trilinear alg}
\end{equation}%
We have already interpolated the association $\underline{k}\mapsto \mathrm{%
Nrd}_{p}^{\underline{k}^{\ast }}\left( x\right) $ in \S \ref{SS Norm forms}
and we will now proceed to interpolate the association $\underline{k}\mapsto
\Lambda _{\underline{k}/\mathbb{Q}_{p}}$. To this end, we first review and
prove some facts on distribution modules, by means of which $p$-adic
families of modular forms are defined.

\section{Spaces of homogeneous $p$-adic distribution spaces}

\subsection{Locally analytic homogeneous distributions}

By a $p$-adic manifold $X$ we always mean a locally compact and paracompact
manifold over a fixed spherically complete non-archimedean $p$-adic field.
For a Banach algebra $\mathcal{O}$, we let $\mathcal{A}\left( X,\mathcal{O}%
\right) $ be the space of $\mathcal{O}$-valued locally analytic functions on 
$X$ and set $\mathcal{D}\left( X,\mathcal{O}\right) :=\mathcal{L}_{\mathcal{O%
}}\left( \mathcal{A}\left( X,\mathcal{O}\right) ,\mathcal{O}\right) \subset 
\mathcal{L}\left( \mathcal{A}\left( X,\mathcal{O}\right) ,\mathcal{O}\right) 
$, the strong $\mathcal{O}$-dual of $\mathcal{A}\left( X,\mathcal{O}\right) $%
. If $f:X\rightarrow Y$ is a morphism of $p$-adic manifolds, we have%
\begin{equation*}
f_{\mathcal{O}}^{\ast }:\mathcal{A}\left( Y,\mathcal{O}\right) \rightarrow 
\mathcal{A}\left( X,\mathcal{O}\right) \text{ and }f_{\ast }^{\mathcal{O}}:%
\mathcal{D}\left( X,\mathcal{O}\right) \rightarrow \mathcal{D}\left( Y,%
\mathcal{O}\right) \text{,}
\end{equation*}%
the first being the pull-back of functions $f_{\mathcal{O}}^{\ast }\left(
F\right) :=F\circ f$ and the second operation being the strong $\mathcal{O}$%
-dual of the first. We note that%
\begin{equation}
f_{\ast }^{\mathcal{O}}\left( \delta _{x}^{\mathcal{O}}\right) =\delta
_{f\left( x\right) }^{\mathcal{O}}\text{, for every }x\in X
\label{Distributions F0}
\end{equation}%
if $\delta _{\cdot }^{\mathcal{O}}:X\rightarrow \mathcal{D}\left( X,\mathcal{%
O}\right) $ denotes the Dirac distribution map. It is useful to remark that
the $\mathcal{O}$-linear span of $\left\{ \delta _{x}^{\mathcal{O}}:x\in
X\right\} $ is dense in $\mathcal{D}\left( X,\mathcal{O}\right) $ (see \cite[%
\S 6.1]{GS}): we refer to this fact using the set phrase "by density of
Dirac distributions". It can be shown that there are topological
identifications\footnote{%
We write $V\otimes _{\iota }W$ (resp. $V\otimes W$)\ to denote $V\otimes W$
with the inductive (resp. projective)\ tensor topology.}%
\begin{equation*}
\mathbf{T}_{\mathcal{D}\left( X\right) }^{\mathcal{O}}:\mathcal{O}\widehat{%
\otimes }\mathcal{D}\left( X\right) \overset{\sim }{\rightarrow }\mathcal{D}%
\left( X,\mathcal{O}\right)
\end{equation*}%
and%
\begin{equation*}
\mathbf{P}_{\mathcal{D}\left( X_{1}\right) ,\mathcal{D}\left( X_{2}\right)
}^{\mathcal{O}_{1},\mathcal{O}_{2}}:\mathcal{D}\left( X_{1},\mathcal{O}%
_{1}\right) \widehat{\otimes }_{\iota }\mathcal{D}\left( X_{2},\mathcal{O}%
_{2}\right) \overset{\sim }{\rightarrow }\mathcal{D}\left( X_{1}\times X_{2},%
\mathcal{O}_{1}\widehat{\otimes }\mathcal{O}_{2}\right) \text{.}
\end{equation*}%
They are characterized by the equalities:%
\begin{equation}
\mathbf{T}_{\mathcal{D}\left( X\right) }^{\mathcal{O}}\left( 1_{\mathcal{O}}%
\widehat{\otimes }\delta _{x}\right) =\delta _{x}^{\mathcal{O}}\text{ and }%
\mathbf{P}_{\mathcal{D}\left( X_{1}\right) ,\mathcal{D}\left( X_{2}\right)
}^{\mathcal{O}_{1},\mathcal{O}_{2}}\left( \delta _{x_{1}}^{\mathcal{O}_{1}}%
\widehat{\otimes }_{\iota }\delta _{x_{2}}^{\mathcal{O}_{2}}\right) =\delta
_{\left( x_{1},x_{2}\right) }^{\mathcal{O}_{1}\widehat{\otimes }\mathcal{O}%
_{2}}\text{.}  \label{Distributions F1}
\end{equation}%
We will usually suppress the reference to the Banach algebra when this is
the fixed $p$-adic field.

\bigskip

Suppose from now on that $X$ is endowed with the action of a $p$-adic Lie
group $T$, meaning that the action is given by a locally analytic map $%
a:T\times X\rightarrow X$. Then $T$ naturally acts from the right on $%
\mathcal{A}\left( X,\mathcal{O}\right) $ and from the left on $\mathcal{D}%
\left( X,\mathcal{O}\right) $. The left action of $T$ on $\mathcal{D}\left(
X\right) $ can be extended, with respect to $\delta _{\cdot }:T\rightarrow 
\mathcal{D}\left( T\right) $, to a left action of $\mathcal{D}\left(
T\right) $ making $\mathcal{D}\left( X\right) $ a $\mathcal{D}\left(
T\right) $-module by the convolution product:%
\begin{equation}
\mathcal{D}\left( T\right) \otimes _{\iota }\mathcal{D}\left( X\right) 
\overset{\mathbf{P}_{\mathcal{D}\left( T\right) ,\mathcal{D}\left( X\right) }%
}{\rightarrow }\mathcal{D}\left( T\times X\right) \overset{a_{\ast }}{%
\rightarrow }\mathcal{D}\left( X\right) \text{.}  \label{Distributions D4}
\end{equation}%
We note the formula%
\begin{equation}
\delta _{t}\cdot \delta _{x}=\delta _{tx}\text{ for }t\in T\text{ and }x\in X%
\text{,}  \label{Distributions F2}
\end{equation}%
which indeed characterizes the multiplication law by density of the Dirac
distributions. Also we remark that the multiplication map is in general
separately continuous, while it is continuous if we assume that $T$ and $X$
are compact. In particular, one checks that $\mathcal{D}\left( T\right) $
becomes an algebra in this way. We write $Hom_{\mathcal{A}}\left( T,\mathcal{%
O}^{\times }\right) $ to denote the group of those group homomorphisms such
that their composition with the inclusion $\mathcal{O}^{\times }\subset 
\mathcal{O}$ belongs to $\mathcal{A}\left( T,\mathcal{O}\right) $. We also
write $Hom_{\mathcal{L}}\left( \mathcal{D}\left( T\right) ,\mathcal{O}%
\right) $ to denote the space of those morphisms of locally convex spaces
that are morphisms of algebras. Then there is a bijection (see \cite[Lemma
6.2]{GS})%
\begin{equation}
C^{\mathcal{O}}:Hom_{\mathcal{L}}\left( \mathcal{D}\left( T\right) ,\mathcal{%
O}\right) \overset{\sim }{\rightarrow }Hom_{\mathcal{A}}\left( T,\mathcal{O}%
^{\times }\right) \text{, via }C^{\mathcal{O}}\left( \mathbf{k}\right)
\left( t\right) :=\mathbf{k}\left( \delta _{t}\right) \text{.}
\label{Distributions F Weight Id}
\end{equation}%
We will abuse of notations, when there will be no risk of confusion, and
identify these two sets, deserving the exponential notation to the group
homomorphisms and calling the elements of these sets weights.

If $\mathbf{k}$ is a weight, we may consider the space of locally analytic
homogeneous functions:%
\begin{equation*}
\mathcal{A}_{\mathbf{k}}\left( X\right) =\mathcal{A}\left( X,\mathbf{k}%
\right) =\left\{ F\in \mathcal{A}\left( X,\mathcal{O}\right) :F\left(
tx\right) =t^{\mathbf{k}}F\left( x\right) \right\} \text{.}
\end{equation*}%
It is indeed a closed $\mathcal{O}$-submodule of $\mathcal{A}\left( X,%
\mathcal{O}\right) $. Viewing both $\mathcal{O}$ and $\mathcal{D}\left(
X\right) $ as $\mathcal{D}\left( T\right) $-modules by means of $\mathbf{k}$
and, respectively, the convolution product, we may define%
\begin{equation*}
\mathcal{D}_{\mathbf{k}}\left( X\right) :=\mathcal{O}\widehat{\otimes }_{%
\mathbf{k}}\mathcal{D}\left( X\right) \text{ and }\mathcal{D}\left( X,%
\mathbf{k}\right) :=\mathcal{L}_{\mathcal{O}}\left( \mathcal{A}_{\mathbf{k}%
}\left( X\right) ,\mathcal{O}\right) \text{.}
\end{equation*}%
We also assume from no on that $X$ is endowed with a right action by a
semigroup $\Sigma $ such that $\sigma :\Sigma \rightarrow \Sigma $ is
locally analytic for every $\sigma \in \Sigma $, which is compatible with
the left $T$-action in the sense that $t(x\sigma )=(tx)\sigma $ for all $%
t\in T$, $x\in X$, and $\sigma \in \Sigma $. It follows that $\sigma $
induces a well defined action on $\mathcal{A}_{\mathbf{k}}\left( X\right) $, 
$\mathcal{D}_{\mathbf{k}}\left( X\right) $ and $\mathcal{D}\left( X,\mathbf{k%
}\right) $. The relation between the space $\mathcal{D}_{\mathbf{k}}\left(
X\right) $ and $\mathcal{D}\left( X,\mathbf{k}\right) $ is expressed by
means of an $\left( \mathcal{O},\Sigma \right) $-equivariant morphism of
locally convex spaces%
\begin{equation}
\mathbf{T}_{\mathcal{D}\left( X\right) }^{\mathbf{k}}:\mathcal{D}_{\mathbf{k}%
}\left( X\right) \rightarrow \mathcal{D}\left( X,\mathbf{k}\right)
\label{Distributions F Int}
\end{equation}%
which is an isomorphism when $X$ is a trivial (equivalently locally trivial)$%
\ T$-bundle. It is characterized by the property that%
\begin{equation*}
\mathbf{T}_{\mathcal{D}\left( X\right) }^{\mathbf{k}}\left( 1\widehat{%
\otimes }_{\mathbf{k}}\delta _{x}\right) =\delta _{x}^{\mathbf{k}}\text{ for
every }x\in X\text{,}
\end{equation*}%
if $\delta _{x}^{\mathbf{k}}$ is the image of $\delta _{x}^{\mathcal{O}}$.
We refer the reader to \cite[Lemma 6.3 and Proposition 6.6]{GS} for details.

It follows from $\left( \text{\ref{Distributions F Int}}\right) $ that the
elements of $\mathcal{D}_{\mathbf{k}}\left( X\right) $ naturally integrates
functions in $\mathcal{A}_{\mathbf{k}}\left( X\right) $. Furthermore they
are endowed with natural specialization maps, not possessed by the spaces $%
\mathcal{D}\left( X,\mathbf{k}\right) $, defined as follows. If we have
given $\mathbf{k}_{i}\in Hom_{\mathcal{L}}\left( \mathcal{D}\left( T\right) ,%
\mathcal{O}_{i}\right) $, we say that $\mathbf{k}_{1}$ specializes via $\phi 
$ to $\mathbf{k}_{2}$, and we write $\mathbf{k}_{1}\overset{\phi }{%
\rightarrow }\mathbf{k}_{2}$, if $\phi \in Hom_{\mathcal{L}}\left( \mathcal{O%
}_{1},\mathcal{O}_{2}\right) $ and $\mathbf{k}_{2}=\phi \circ \mathbf{k}_{1}$%
. Then we have an induced specialization map%
\begin{equation}
\phi _{\ast }:\mathcal{D}_{\mathbf{k}_{1}}\left( X\right) \rightarrow 
\mathcal{D}_{\mathbf{k}_{2}}\left( X\right) \text{ via }\phi _{\ast }\left(
\alpha \widehat{\otimes }_{\mathbf{k}_{1}}\mu \right) :=\phi \left( \alpha
\right) \widehat{\otimes }_{\mathbf{k}_{2}}\mu \text{.}
\label{Distributions F Specialization}
\end{equation}

\subsection{Multiplying locally analytic homogeneous distributions}

Now suppose that we have given two $p$-adic locally compact and paracompact
manifolds $X_{i}$ endowed with analytic actions of $T_{i}$ for $i=1,2$, so
that $T_{1}\times T_{2}$ act on $X_{1}\times X_{2}$ in the obvious way. Let
us be given $\mathbf{k}_{i}\in Hom_{\mathcal{L}}\left( \mathcal{D}\left(
T_{i}\right) ,\mathcal{O}_{i}\right) $. We define the continuous morphism of
locally convex spaces%
\begin{equation}
\mathbf{k}_{1}\boxplus \mathbf{k}_{2}:\mathcal{D}\left( T_{1}\times
T_{2}\right) \overset{\mathbf{P}_{\mathcal{D}\left( T_{1}\right) ,\mathcal{D}%
\left( T_{2}\right) }^{-1}}{\rightarrow }\mathcal{D}\left( T_{1}\right) 
\widehat{\otimes }_{\iota }\mathcal{D}\left( T_{2}\right) \overset{\mathbf{k}%
_{1}\widehat{\otimes }_{\iota }\mathbf{k}_{2}}{\rightarrow }\mathcal{O}_{1}%
\widehat{\otimes }_{\iota }\mathcal{O}_{2}\overset{\widehat{1}}{\rightarrow }%
\mathcal{O}_{1}\widehat{\otimes }\mathcal{O}_{2}\text{.}
\label{Distributions Fmult1}
\end{equation}%
Exploiting the effect on Dirac distributions and noticing that the
multiplications laws are separately continuous by $\left( \text{\ref%
{Distributions D4}}\right) $, it is not difficult to deduce from the density
of Dirac distributions that $\mathbf{k}_{1}\boxplus \mathbf{k}_{2}$ is a
morphism of algebras, hence%
\begin{equation*}
\mathbf{k}_{1}\boxplus \mathbf{k}_{2}\in Hom_{\mathcal{L}}\left( \mathcal{D}%
\left( T_{1}\times T_{2}\right) ,\mathcal{O}_{1}\widehat{\otimes }\mathcal{O}%
_{2}\right) .
\end{equation*}%
We assume that $X_{i}$ is further endowed with a right action by a semigroup 
$\Sigma _{i}$ having the same properties at the $\Sigma $-action considered
above.

\begin{lemma}
\label{Distributions L3}There is a unique morphism of locally convex spaces $%
\mathbf{P}_{\mathcal{D}\left( X_{1}\right) ,\mathcal{D}\left( X_{2}\right)
}^{\mathbf{k}_{1},\mathbf{k}_{2}}$ making the following diagram commutative,
which is $\left( \mathcal{O}_{1}\widehat{\otimes }\mathcal{O}_{2},\Sigma
_{1}\times \Sigma _{2}\right) $-equivariant:%
\begin{equation}
\begin{array}{ccc}
\mathcal{O}_{1}\widehat{\otimes }\mathcal{D}\left( X_{1}\right) \widehat{%
\otimes }_{\iota }\mathcal{O}_{2}\widehat{\otimes }\mathcal{D}\left(
X_{2}\right) & \overset{1_{\mathcal{O}_{1}\widehat{\otimes }\mathcal{O}_{2}}%
\widehat{\otimes }\mathbf{P}_{\mathcal{D}\left( X_{1}\right) ,\mathcal{D}%
\left( X_{2}\right) }}{\rightarrow } & \mathcal{O}_{1}\widehat{\otimes }%
\mathcal{O}_{2}\widehat{\otimes }\mathcal{D}\left( X_{1}\times X_{2}\right)
\\ 
\downarrow &  & \downarrow \\ 
\mathcal{D}_{\mathbf{k}_{1}}\left( X_{1}\right) \widehat{\otimes }_{\iota }%
\mathcal{D}_{\mathbf{k}_{2}}\left( X_{2}\right) & \overset{\mathbf{P}_{%
\mathcal{D}\left( X_{1}\right) ,\mathcal{D}\left( X_{2}\right) }^{\mathbf{k}%
_{1},\mathbf{k}_{2}}}{\rightarrow } & \mathcal{D}_{\mathbf{k}_{1}\boxplus 
\mathbf{k}_{2}}\left( X_{1}\times X_{2}\right) \text{.}%
\end{array}
\label{Distributions L3 Diagram}
\end{equation}
\end{lemma}

\begin{proof}
Let $B$ be the composition of $1_{\mathcal{O}_{1}\widehat{\otimes }\mathcal{O%
}_{2}}\widehat{\otimes }\mathbf{P}_{\mathcal{D}\left( X_{1}\right) ,\mathcal{%
D}\left( X_{2}\right) }$ with the right vertical morphism. Since we know
that $B$ is continuous and $\mathcal{D}_{\mathbf{k}_{1}\boxplus \mathbf{k}%
_{2}}\left( X_{1}\times X_{2}\right) $ is Hausdorff and complete, we first
need to show that, for every $\alpha _{i}\in \mathcal{O}_{i}$, $\mu _{i}\in 
\mathcal{D}\left( X_{i}\right) $ and $\nu _{i}\in \mathcal{D}\left(
T_{i}\right) $%
\begin{equation*}
B\left( \alpha _{1}\mathbf{k}_{1}\left( \nu _{1}\right) \widehat{\otimes }%
\mu _{1}\widehat{\otimes }_{\iota }\alpha _{2}\mathbf{k}_{1}\left( \nu
_{2}\right) \widehat{\otimes }\mu _{2}\right) =B\left( \alpha _{1}\widehat{%
\otimes }\left( \nu _{1}\cdot \mu _{1}\right) \widehat{\otimes }_{\iota
}\alpha _{2}\widehat{\otimes }\left( \nu _{2}\cdot \mu _{2}\right) \right) 
\text{.}
\end{equation*}%
It turns out that this is equivalent to checking the equalities%
\begin{equation}
\mathbf{P}_{\mathcal{D}\left( T_{1}\right) ,\mathcal{D}\left( T_{2}\right)
}\left( \nu _{1}\widehat{\otimes }\nu _{2}\right) \cdot \mathbf{P}_{\mathcal{%
D}\left( X_{1}\right) ,\mathcal{D}\left( X_{2}\right) }\left( \mu _{1}%
\widehat{\otimes }\mu _{2}\right) =\mathbf{P}_{\mathcal{D}\left(
X_{1}\right) ,\mathcal{D}\left( X_{2}\right) }\left( \left( \nu _{1}\cdot
\mu _{1}\right) \widehat{\otimes }\left( \nu _{2}\cdot \mu _{2}\right)
\right) \text{ in }\mathcal{D}\left( X_{1}\times X_{2}\right) \text{.}
\label{Distributions L3 Claim}
\end{equation}%
When $\mu _{i}=\delta _{x_{i}}$ and $\nu _{i}=\delta _{t_{i}}$ we have
indeed, by $\left( \text{\ref{Distributions F1}}\right) $ and $\left( \text{%
\ref{Distributions F2}}\right) $%
\begin{eqnarray*}
&&\mathbf{P}_{\mathcal{D}\left( T_{1}\right) ,\mathcal{D}\left( T_{2}\right)
}\left( \nu _{1}\widehat{\otimes }\nu _{2}\right) \cdot \mathbf{P}_{\mathcal{%
D}\left( X_{1}\right) ,\mathcal{D}\left( X_{2}\right) }\left( \mu _{1}%
\widehat{\otimes }\mu _{2}\right) =\delta _{\left( t_{1},t_{2}\right) }\cdot
\delta _{\left( x_{1},x_{2}\right) }=\delta _{\left(
t_{1}x_{1},t_{2}x_{2}\right) }\text{,} \\
&&\mathbf{P}_{\mathcal{D}\left( X_{1}\right) ,\mathcal{D}\left( X_{2}\right)
}\left( \left( \nu _{1}\cdot \mu _{1}\right) \widehat{\otimes }\left( \nu
_{2}\cdot \mu _{2}\right) \right) =\mathbf{P}_{\mathcal{D}\left(
X_{1}\right) ,\mathcal{D}\left( X_{2}\right) }\left( \delta _{t_{1}x_{1}}%
\widehat{\otimes }\delta _{t_{2}x_{2}}\right) =\delta _{\left(
t_{1}x_{1},t_{2}x_{2}\right) }\text{.}
\end{eqnarray*}%
We note that both the left and the right hand sides of $\left( \text{\ref%
{Distributions L3 Claim}}\right) $ are linear in the variables $\mu _{i}$
and $\nu _{i}$. Furthermore, if we fix three of these variables, the two
resulting functions are continuous in the remaining variable thanks to $%
\left( \text{\ref{Distributions D4}}\right) $ showing that the
multiplication laws are separately continuous. Hence the claimed equality $%
\left( \text{\ref{Distributions L3 Claim}}\right) $ follows from the density
of Dirac distributions. The existence and uniqueness of $\mathbf{P}_{%
\mathcal{D}\left( X_{1}\right) ,\mathcal{D}\left( X_{2}\right) }^{\mathbf{k}%
_{1},\mathbf{k}_{2}}$ follows and, since $\mathcal{O}_{1}\widehat{\otimes }%
\mathcal{D}\left( X_{1}\right) \widehat{\otimes }_{\iota }\mathcal{O}_{2}%
\widehat{\otimes }\mathcal{D}\left( X_{2}\right) \rightarrow \mathcal{D}_{%
\mathbf{k}_{1}}\left( X_{1}\right) \widehat{\otimes }_{\iota }\mathcal{D}_{%
\mathbf{k}_{1}}\left( X_{2}\right) $ is surjective and all the arrows other
than $\mathbf{P}_{\mathcal{D}\left( X_{1}\right) ,\mathcal{D}\left(
X_{2}\right) }^{\mathbf{k}_{1},\mathbf{k}_{2}}$ in $\left( \text{\ref%
{Distributions L3 Diagram}}\right) $ are $\left( \mathcal{O}_{1}\widehat{%
\otimes }\mathcal{O}_{2},\Sigma _{1}\times \Sigma _{2}\right) $-equivariant
(by $\left( \text{\ref{Distributions F1}}\right) $), $\mathbf{P}_{\mathcal{D}%
\left( X_{1}\right) ,\mathcal{D}\left( X_{2}\right) }^{\mathbf{k}_{1},%
\mathbf{k}_{2}}$ is equivariant as well.
\end{proof}

\bigskip

In particular we may define%
\begin{equation*}
\overline{\mathbf{P}}_{\mathcal{D}\left( X_{1}\right) ,\mathcal{D}\left(
X_{2}\right) }^{\mathbf{k}_{1},\mathbf{k}_{2}}:\mathcal{D}_{\mathbf{k}%
_{1}}\left( X_{1}\right) \widehat{\otimes }_{\iota }\mathcal{D}_{\mathbf{k}%
_{1}}\left( X_{2}\right) \overset{\mathbf{P}_{\mathcal{D}\left( X_{1}\right)
,\mathcal{D}\left( X_{2}\right) }^{\mathbf{k}_{1},\mathbf{k}_{2}}}{%
\rightarrow }\mathcal{D}_{\mathbf{k}_{1}\boxplus \mathbf{k}_{2}}\left(
X_{1}\times X_{2}\right) \overset{\mathbf{T}_{\mathcal{D}\left( X_{1}\times
X_{2}\right) }^{\mathbf{k}_{1}\boxplus \mathbf{k}_{2}}}{\rightarrow }%
\mathcal{D}\left( X_{1}\times X_{2},\mathbf{k}_{1}\boxplus \mathbf{k}%
_{2}\right) \text{.}
\end{equation*}%
If $\mu _{i}\in \mathcal{D}_{\mathbf{k}_{i}}\left( X_{i}\right) $ for $i=1,2$%
, we set%
\begin{equation*}
\mu _{1}\boxtimes \mu _{2}:=\overline{\mathbf{P}}_{\mathcal{D}\left(
X_{1}\right) ,\mathcal{D}\left( X_{2}\right) }^{\mathbf{k}_{1},\mathbf{k}%
_{2}}\left( \mu _{1}\widehat{\otimes }_{\iota }\mu _{2}\right) \in \mathcal{D%
}\left( X_{1}\times X_{2},\mathbf{k}_{1}\boxplus \mathbf{k}_{2}\right) \text{%
.}
\end{equation*}%
Of course, the formation of $\mathbf{k}_{1}\boxplus \mathbf{k}_{2}$, $%
\mathbf{P}_{\mathcal{D}\left( X_{1}\right) ,\mathcal{D}\left( X_{2}\right)
}^{\mathbf{k}_{1},\mathbf{k}_{2}}$ and $\overline{\mathbf{P}}_{\mathcal{D}%
\left( X_{1}\right) ,\mathcal{D}\left( X_{2}\right) }^{\mathbf{k}_{1},%
\mathbf{k}_{2}}$ extends to a finite number of indices and the usual
associativity constraints are satisfied, as well as the compatibility with
the commutativity constraints in the sources and the targets of these maps.
We finally remark that the equations%
\begin{eqnarray}
\mathbf{P}_{\mathcal{D}\left( X_{1}\right) ,\mathcal{D}\left( X_{2}\right)
}^{\mathbf{k}_{1},\mathbf{k}_{2}}\left( 1_{\mathcal{O}_{1}}\widehat{\otimes }%
_{\mathbf{k}_{1}}\delta _{x_{1}}\widehat{\otimes }_{\iota }1_{\mathcal{O}%
_{2}}\widehat{\otimes }_{\mathbf{k}_{2}}\delta _{x_{2}}\right) &=&1_{%
\mathcal{O}_{1}\widehat{\otimes }\mathcal{O}_{2}}\widehat{\otimes }_{\mathbf{%
k}_{1}\boxplus \mathbf{k}_{2}}\delta _{\left( x_{1},x_{2}\right) }\text{,} 
\notag \\
\overline{\mathbf{P}}_{\mathcal{D}\left( X_{1}\right) ,\mathcal{D}\left(
X_{2}\right) }^{\mathbf{k}_{1},\mathbf{k}_{2}}\left( 1_{\mathcal{O}_{1}}%
\widehat{\otimes }_{\mathbf{k}_{1}}\delta _{x_{1}}\widehat{\otimes }_{\iota
}1_{\mathcal{O}_{2}}\widehat{\otimes }_{\mathbf{k}_{2}}\delta
_{x_{2}}\right) &=&\delta _{\left( x_{1},x_{2}\right) }^{\mathbf{k}%
_{1}\boxplus \mathbf{k}_{2}}  \label{Distributions F3}
\end{eqnarray}%
characterize these maps.

\subsection{\label{SS Algebraic operations on weights}Algebraic operations
on weights}

Setting $\mathcal{X}_{T}\left( \mathcal{O}\right) :=Hom_{\mathcal{A}}\left(
T,\mathcal{O}^{\times }\right) $ defines a group functor on Banach algebras,
so that we have%
\begin{equation*}
+:\mathcal{X}_{T}\left( \mathcal{O}\right) \times \mathcal{X}_{T}\left( 
\mathcal{O}\right) \rightarrow \mathcal{X}_{T}\left( \mathcal{O}\right) 
\text{ and }-:\mathcal{X}_{T}\left( \mathcal{O}\right) \rightarrow \mathcal{X%
}_{T}\left( \mathcal{O}\right) \text{.}
\end{equation*}%
It follows from $\left( \text{\ref{Distributions F Weight Id}}\right) $ that
we may transport these operations getting%
\begin{equation*}
+:Hom_{\mathcal{L}}\left( \mathcal{D}\left( T\right) ,\mathcal{O}_{i}\right)
\times Hom_{\mathcal{L}}\left( \mathcal{D}\left( T\right) ,\mathcal{O}%
_{i}\right) \rightarrow Hom_{\mathcal{L}}\left( \mathcal{D}\left( T\right) ,%
\mathcal{O}_{i}\right)
\end{equation*}%
and%
\begin{equation*}
-:Hom_{\mathcal{L}}\left( \mathcal{D}\left( T\right) ,\mathcal{O}_{i}\right)
\rightarrow Hom_{\mathcal{L}}\left( \mathcal{D}\left( T\right) ,\mathcal{O}%
_{i}\right) \text{.}
\end{equation*}%
Our next task it to interpolate these operations.

If we have given $\mathbf{k}_{i}\in Hom_{\mathcal{L}}\left( \mathcal{D}%
\left( T\right) ,\mathcal{O}_{i}\right) $, then we define%
\begin{equation*}
\mathbf{k}_{1}\oplus \mathbf{k}_{2}:\mathcal{D}\left( T\right) \overset{%
\Delta _{\ast }}{\rightarrow }\mathcal{D}\left( T\times T\right) \overset{%
\mathbf{k}_{1}\boxplus \mathbf{k}_{2}}{\rightarrow }\mathcal{O}_{1}\widehat{%
\otimes }\mathcal{O}_{2}\text{,}
\end{equation*}%
where $\Delta :T\rightarrow T\times T$ is the diagonal map and $\mathbf{k}%
_{1}\boxplus \mathbf{k}_{2}$ is given by $\left( \text{\ref{Distributions
Fmult1}}\right) $.

If $\mathbf{k}\in Hom_{\mathcal{L}}\left( \mathcal{D}\left( T\right) ,%
\mathcal{O}\right) $, then we define%
\begin{equation*}
\ominus \mathbf{k:}\mathcal{D}\left( T\right) \overset{i_{\ast }}{%
\rightarrow }\mathcal{D}\left( T\right) \overset{\mathbf{k}}{\rightarrow }%
\mathcal{O}\text{,}
\end{equation*}%
where $i:T\rightarrow T$ is the inversion, and set%
\begin{equation*}
\mathbf{k}_{1}\ominus \mathbf{k}_{2}:=\mathbf{k}_{1}\oplus \left( \ominus 
\mathbf{k}_{2}\right)
\end{equation*}%
We note that these operations are obviously functorial and compatible with
specialization.

Exploiting the definitions and making $\left( \text{\ref{Distributions
Fmult1}}\right) $ explicit it is easy to check the following result.

\begin{lemma}
\label{Distributions L WeightsOp}Suppose that $\mathbf{k},\mathbf{k}_{i}\in
Hom_{\mathcal{L}}\left( \mathcal{D}\left( T\right) ,\mathcal{O}\right) $ and
write%
\begin{equation*}
m_{\mathcal{O}}:\mathcal{O}\widehat{\otimes }\mathcal{O}\rightarrow \mathcal{%
O}
\end{equation*}%
for the multiplication map. Then $-\mathbf{k=}\ominus \mathbf{k}$ and%
\begin{equation*}
\mathbf{k}_{1}+\mathbf{k}_{2}:\mathcal{D}\left( T\right) \overset{\mathbf{k}%
_{1}\oplus \mathbf{k}_{2}}{\rightarrow }\mathcal{O}\widehat{\otimes }%
\mathcal{O}\overset{m_{\mathcal{O}}}{\rightarrow }\mathcal{O}\text{.}
\end{equation*}
\end{lemma}

\bigskip

We now illustrate why $\mathbf{k}_{1}\oplus \mathbf{k}_{2}$ interpolates the 
$+$ operation. Suppose that $F$ is our $p$-adic working field and that $%
\mathbf{k}_{i}\in Hom_{\mathcal{L}}\left( \mathcal{D}\left( T\right)
,F\right) $ are such that $\mathbf{k}_{i}\overset{\phi _{i}}{\rightarrow }%
k_{i}$. Then%
\begin{equation*}
\mathbf{k}_{1}\oplus \mathbf{k}_{2}\overset{\phi _{1}\widehat{\otimes }\phi
_{2}}{\rightarrow }k_{1}\oplus k_{2}
\end{equation*}%
by the compatibility of the $\oplus $-operation with specializations. But we
have $F\widehat{\otimes }F=F$ canonically and the identification is given by 
$m_{F}$. Hence $\mathbf{k}_{1}\oplus \mathbf{k}_{2}$ specializes via $\phi
_{1}\widehat{\otimes }\phi _{2}$ to $k_{1}+k_{2}$, thanks to Lemma \ref%
{Distributions L WeightsOp}. In particular, suppose that $\mathcal{X}_{T}$
is representable by a rigid analytic space (for example because $T$ is
compact) and that $\mathbf{k}_{i}\overset{\phi _{i}}{\rightarrow }k_{i}$
corresponds to $k_{i}\in U_{i}$, with $U_{i}\subset \mathcal{X}_{T}$ an
affinoid neighbourhood of $k$. Then $\mathbf{k}_{1}\oplus \mathbf{k}_{2}$
corresponds to%
\begin{equation*}
U_{1}\times U_{2}\subset \mathcal{X}_{T}\times \mathcal{X}_{T}\overset{+}{%
\rightarrow }\mathcal{X}_{T}\text{.}
\end{equation*}%
We finally remark that, as a consequence of the associativity of the
operation in $T$, we have%
\begin{equation*}
\left( \mathbf{k}_{1}\oplus \mathbf{k}_{2}\right) \oplus \mathbf{k}%
_{3}\simeq \mathbf{k}_{1}\oplus \left( \mathbf{k}_{2}\oplus \mathbf{k}%
_{3}\right)
\end{equation*}%
up to%
\begin{equation*}
\left( \mathcal{O}_{1}\otimes \mathcal{O}_{2}\right) \otimes \mathcal{O}%
_{3}\simeq \mathcal{O}_{1}\otimes \left( \mathcal{O}_{2}\otimes \mathcal{O}%
_{3}\right) \text{.}
\end{equation*}%
A similar compatibility holds true for the commutativity, when $T$ is
commutative as in our applications.

\bigskip

Suppose now $T\simeq \Delta \times \left( 1+p\mathbb{Z}_{p}\right) ^{r}$,
where $\Delta $ is the torsion part of $T$, and consider the multiplication
by $2$ map $t\mapsto t^{2}$ (we write $T$ multiplicatively). We say that $%
\mathbf{k}\in \mathcal{X}_{T}\left( \mathcal{O}\right) $ is even if it is in
the image of $2^{\ast }:\mathcal{X}_{T}\left( \mathcal{O}\right) \rightarrow 
\mathcal{X}_{T}\left( \mathcal{O}\right) $ and set $\frac{\mathbf{k}}{2}$
for an element in the inverse image of $\mathbf{k}$. For example, suppose
that $p\neq 2$ and $T=\mathbb{Z}_{p}^{\times }\simeq \mathbb{F}_{p}^{\times
}\times \left( 1+p\mathbb{Z}_{p}\right) $. We can decompose every $\mathbf{k}%
\in \mathcal{X}_{T}\left( \mathcal{O}\right) $ in the form $\mathbf{k}%
=\left( \left[ \mathbf{k}\right] ,\left\langle \mathbf{k}\right\rangle
\right) $, where $\left[ \mathbf{k}\right] \in \mathbb{F}_{p}^{\times }$ and 
$\left\langle \mathbf{k}\right\rangle \in \mathcal{X}_{1+p\mathbb{Z}%
_{p}}\left( \mathcal{O}\right) $. Since $t\mapsto t^{2}$ is invertible on $%
1+p\mathbb{Z}_{p}$, $\mathbf{k}=\left( \left[ \mathbf{k}\right]
,\left\langle \mathbf{k}\right\rangle \right) $ is even if and only if $%
\left[ \mathbf{k}\right] \in \mathbb{F}_{p}^{\times 2}$ and then $\frac{%
\mathbf{k}}{2}\in \left\{ \left( \frac{\left[ \mathbf{k}\right] }{2},\frac{%
\left\langle \mathbf{k}\right\rangle }{2}\right) ,\left( -\frac{\left[ 
\mathbf{k}\right] }{2},\frac{\left\langle \mathbf{k}\right\rangle }{2}%
\right) \right\} $; if $\left[ \mathbf{k}\right] =\left[ k_{0}\right] $ for
some integer $k_{0}$ our convention is to choose $\frac{\mathbf{k}}{2}%
=\left( \frac{\left[ k_{0}\right] }{2},\frac{\left\langle \mathbf{k}%
\right\rangle }{2}\right) $. Then $\mathbf{k}\overset{\phi }{\rightarrow }%
k\in \mathbb{N}$ implies $k\in 2\mathbb{N}$ and $\frac{\mathbf{k}}{2}\overset%
{\phi }{\rightarrow }\frac{k}{2}$. The elements of $Hom_{\mathcal{L}}\left( 
\mathcal{D}\left( T\right) ,\mathcal{O}\right) \simeq \mathcal{X}_{T}\left( 
\mathcal{O}\right) $ are called $\mathcal{O}$-weights; we will freely
identify $\mathbf{k}_{1}\boxplus \mathbf{k}_{2}\simeq \left( \mathbf{k}_{1},%
\mathbf{k}_{2}\right) $.

\section{The $p$-adic trilinear form}

The semigroup $\Sigma _{0}\left( p\mathbb{Z}_{p}\right) \subset \mathbf{M}%
_{2}\left( \mathbb{Z}_{p}\right) $ acts from the right on the set $W:=%
\mathbb{Z}_{p}^{\times }\times \mathbb{Z}_{p}$. Setting $\omega _{p}:=\left( 
\begin{array}{cc}
0 & -1 \\ 
p & 0%
\end{array}%
\right) $, we have $\widehat{W}:=W\omega _{p}=p\mathbb{Z}_{p}\times \mathbb{Z%
}_{p}^{\times }$, on which $\Sigma _{0}\left( p\mathbb{Z}_{p}\right) ^{\iota
}=\omega _{p}^{-1}\Sigma _{0}\left( p\mathbb{Z}_{p}\right) \omega _{p}$ acts
from the right. Hence, for a $\mathcal{O}$-weight $\mathbf{k}$, we may form
the right $\Sigma _{0}\left( p\mathbb{Z}_{p}\right) $-module (resp. $\Sigma
_{0}\left( p\mathbb{Z}_{p}\right) ^{\iota }$-module) $\mathcal{D}_{\mathbf{k}%
}\left( W\right) $ (resp. $\mathcal{D}_{\mathbf{k}}\left( \widehat{W}\right) 
$). We take $K_{p}^{\diamond }=\Gamma _{0}\left( p\mathbb{Z}_{p}\right)
:=\Sigma _{0}\left( p\mathbb{Z}_{p}\right) \cap \mathbf{GL}_{2}\left( 
\mathbb{Z}_{p}\right) $. Then we may form the spaces of $p$-adic families of
modular forms on $\mathbf{B}^{\times }$:%
\begin{equation*}
M_{p}\left( \mathcal{D}_{\mathbf{k}}\left( W\right) ,\omega _{0,p}\right)
:=M_{p}\left( B_{f}^{\times },\mathcal{D}_{\mathbf{k}}\left( W\right)
,\omega _{0,p}\right) \text{.}
\end{equation*}%
Recall we work over a $p$-adic field $F$ and consider Banach $F$-algebras:
we set $M_{p}\left( \mathbf{V}_{k,F},\omega _{0,p}\right) :=M_{p}\left(
B_{f}^{\times },\mathbf{V}_{k,F},\omega _{0,p}\right) $ and $M\left( \mathbf{%
V}_{k,F},\omega _{0}\right) :=M\left( B_{f}^{\times },\mathbf{V}%
_{k,F},\omega _{0}\right) $ and, in general, we remove the group from the
notation for $p$-adic forms when it will be clear.

\begin{example}
\label{Example Nebetype1}Fix an identification $\mathbf{B}\left( \mathbb{A}%
_{f}^{\mathrm{Disc}\left( B\right) }\right) \simeq \mathbf{M}_{2}\left( 
\mathbb{A}_{f}^{\mathrm{Disc}\left( B\right) }\right) $ and, for an integer $%
N$ such that $\left( N,\mathrm{Disc}\left( B\right) \right) =1$, write $%
K_{0}^{\mathrm{Disc}\left( B\right) }\left( N\right) \subset \mathbf{B}%
^{\times }\left( \mathbb{A}_{f}^{\mathrm{Disc}\left( B\right) }\right) $
(resp. $K_{1}^{\mathrm{Disc}\left( B\right) }\left( N\right) \subset K_{0}^{%
\mathrm{Disc}\left( B\right) }\left( N\right) $) for the subgroup which
corresponds to matrices with integral coefficients having lower left entry $%
c\equiv 0$ $\mathrm{mod}\left( N\right) $ (resp. upper left entry $a=1$).
Setting $\mathcal{O}_{\mathrm{Disc}\left( B\right) }:=\tprod\nolimits_{l\mid 
\mathrm{Disc}\left( B\right) }\mathcal{O}_{B_{v}}^{\times }$ we can define%
\begin{equation*}
K_{0}\left( N\right) :=K_{0}^{\mathrm{Disc}\left( B\right) }\left( N\right)
\times \mathcal{O}_{\mathrm{Disc}\left( B\right) }^{\times }\text{ and }%
K_{0}\left( N\right) :=K_{0}^{\mathrm{Disc}\left( B\right) }\left( N\right)
\times \mathcal{O}_{\mathrm{Disc}\left( B\right) }^{\times }\text{.}
\end{equation*}%
Assuming that $\mu _{\varphi \left( N\right) }\subset F$, we can decompose%
\begin{equation*}
M_{p}\left( \mathcal{D}_{\mathbf{k}}\left( W\right) \right) ^{K_{1}\left(
N\right) }=\tbigoplus\nolimits_{\epsilon :\left( \frac{\mathbb{Z}}{N\mathbb{Z%
}}\right) ^{\times }\rightarrow \mathcal{O}^{\times }}M_{p}\left( \mathcal{D}%
_{\mathbf{k}}\left( W\right) \right) ^{K_{1}\left( N\right) }\left( \epsilon
\right) \text{,}
\end{equation*}%
where $M\left( \epsilon \right) $ is the submodule of elements $x\in M$ such
that $xu=\epsilon \left( u\right) x$ if we define $\epsilon \left( u\right)
:=\epsilon \left( a_{u}\right) $ for $a_{u}$ the upper left entry of $u\in
K_{0}\left( N\right) $. Setting $\omega _{0,p}^{\epsilon ,\mathbf{k}}\left(
z\right) :=\epsilon \left( \frac{z}{\mathrm{N}_{f}\left( z\right) }\right)
\left( \frac{z}{\mathrm{N}_{f}\left( z\right) }\right) _{p}^{-\mathbf{k}}$,
we have%
\begin{equation*}
M_{p}\left( \mathcal{D}_{\mathbf{k}}\left( W\right) \right) ^{K_{1}\left(
N\right) }\left( \epsilon \right) =M_{p}\left( \mathcal{D}_{\mathbf{k}%
}\left( W\right) ,\omega _{0,p}^{\epsilon ,\mathbf{k}}\right) ^{K_{1}\left(
N\right) }\subset M_{p}\left( \mathcal{D}_{\mathbf{k}}\left( W\right)
,\omega _{0,p}^{\epsilon ,\mathbf{k}}\right) \text{.}
\end{equation*}
\end{example}

\bigskip

If $\mathbf{k}\overset{\phi }{\rightarrow }k\in \mathbb{N}$, there is a
specialization map%
\begin{equation*}
\phi _{\ast }^{\mathrm{al}\text{\textrm{g}}}:M_{p}\left( \mathcal{D}_{%
\mathbf{k}}\left( W\right) \right) \overset{\phi _{\ast }}{\rightarrow }%
M_{p}\left( \mathcal{D}_{k}\left( W\right) \right) \overset{\nu _{k}}{%
\rightarrow }M_{p}\left( \mathbf{V}_{k,F}\right)
\end{equation*}%
where the first arrow is induced by $\left( \text{\ref{Distributions F
Specialization}}\right) $ and the second is the restriction to polynomials
map. We will usually write $\varphi _{k}:=\phi _{\ast }^{\mathrm{al}\text{%
\textrm{g}}}\left( \varphi \right) $ when $\varphi \in M_{p}\left( \mathcal{D%
}_{\mathbf{k}}\left( W\right) \right) $.

\begin{example}
\label{Example Nebetype2}Suppose that we are in the setting of Example \ref%
{Example Nebetype1}, so that $\mu _{\varphi \left( N\right) }\subset F$.
Then we can decompose $M\left( \mathbf{V}_{k,F}\right) ^{K_{1}\left(
N\right) }$ as we did for $p$-adic forms. Setting $\omega _{0}^{\epsilon
,k}\left( z\right) :=\epsilon \left( \frac{z}{\mathrm{N}_{f}\left( z\right) }%
\right) \mathrm{N}_{f}^{k}\left( z\right) $, we have%
\begin{equation*}
M\left( \mathbf{V}_{k,F}\right) ^{K_{1}\left( N\right) }\left( \epsilon
\right) =M\left( \mathbf{V}_{k,F},\omega _{0}^{\epsilon ,k}\right)
^{K_{1}\left( N\right) }\subset M\left( \mathbf{V}_{k,F},\omega
_{0}^{\epsilon ,k}\right) \text{.}
\end{equation*}%
Furthermore, when $\epsilon :\left( \frac{\mathbb{Z}}{N\mathbb{Z}}\right)
^{\times }\rightarrow \mu _{\varphi \left( N\right) }\subset F^{\times }$,
the specialization map induces%
\begin{equation*}
\phi _{\ast }^{\mathrm{al}\text{\textrm{g}}}:M_{p}\left( \mathcal{D}_{%
\mathbf{k}}\left( W\right) ,\omega _{0,p}^{\epsilon ,\mathbf{k}}\right)
\rightarrow M_{p}\left( \mathbf{V}_{k,F},\omega _{0,p}^{\epsilon ,k}\right)
\simeq M\left( \mathbf{V}_{k,F},\omega _{0}^{\epsilon ,k}\right) \text{.}
\end{equation*}
\end{example}

\bigskip

We identify $\psi :\mathbb{Q}_{p}^{2}\times \mathbb{Q}_{p}^{2}\overset{\sim }%
{\rightarrow }\mathbf{M}_{2}\left( \mathbb{Q}_{p}\right) $ by the rule $\psi
\left( x_{1},y_{1},x_{2},y_{2}\right) :=\left( 
\begin{array}{cc}
x_{1} & y_{1} \\ 
x_{2} & y_{2}%
\end{array}%
\right) $ and, for a subset $S\subset \mathbb{Q}_{p}^{2}\times \mathbb{Q}%
_{p}^{2}$, we define%
\begin{equation*}
\delta _{S}:S\rightarrow \mathbb{Q}_{p}\text{, }\delta _{S}\left( s\right) :=%
\mathrm{det}\left( \psi \left( s\right) \right) \text{ and }S_{n}:=\delta
_{S}^{-1}\left( p^{n}\mathbb{Z}_{p}^{\times }\right) \text{.}
\end{equation*}%
For a continuous\ group homomorphism $\mathbf{k}:\mathbb{Z}_{p}^{\times
}\rightarrow \mathcal{O}^{\times }$ valued in a $\mathbb{Q}_{p}$-Banach
algebra $\mathcal{O}$, we may consider%
\begin{equation*}
\delta _{S_{0}}^{\mathbf{k}}:S_{0}\rightarrow \mathbb{Z}_{p}^{\times }%
\overset{\mathbf{k}}{\rightarrow }\mathcal{O}^{\times }\text{, }\delta _{S}^{%
\mathbf{k}}\left( s\right) :=\delta _{S}\left( s\right) ^{\mathbf{k}}\text{.}
\end{equation*}%
It is a locally analytic function, when $S\subset \mathbb{Q}_{p}^{2}\times 
\mathbb{Q}_{p}^{2}$ is a submanifold, because $\mathbf{k}$ is locally
analytic and $S_{0}\subset S$ is a submanifold.

Since $\psi \left( w_{1}g,w_{2}g\right) =\psi \left( w_{1},w_{2}\right) g$
for any $w_{i}=\left( x_{i},y_{i}\right) $ with $i=1,2$ and $g\in \mathbf{GL}%
_{2}\left( \mathbb{Q}_{p}\right) $, we have $\delta _{Sg}\left( sg\right)
=\delta _{S}\left( s\right) \mathrm{det}\left( g\right) $ for any $s\in S$
and $\left( Sg\right) _{n}=S_{n-\nu _{p}\left( g\right) }g$ where $\nu _{p}:=%
\mathrm{ord}_{p}\circ \mathrm{det}$. In particular, if $\Sigma \subset 
\mathbf{GL}_{2}\left( \mathbb{Q}_{p}\right) $ is a subsemigroup acting on $S$%
, then $\Gamma _{\Sigma }:=\Sigma \cap \mathrm{det}^{-1}\left( \mathbb{Z}%
_{p}^{\times }\right) $ acts on $S_{n}$ for every $n$. We have%
\begin{equation}
\delta _{S_{0}}^{\mathbf{k}}\left( sg\right) =\mathrm{det}\left( g\right) ^{%
\mathbf{k}}\delta _{S_{0}}^{\mathbf{k}}\left( s\right)  \label{F det 1}
\end{equation}%
and, in particular,%
\begin{equation}
\delta _{S_{0}}^{\mathbf{k}}\left( t_{1}s_{1},t_{2}s_{2}\right) =t_{1}^{%
\mathbf{k}}t_{2}^{\mathbf{k}}\delta _{S_{0}}^{\mathbf{k}}\left( s\right) 
\text{.}  \label{F det 2}
\end{equation}%
Noticing that $\left( W\times \widehat{W}\right) _{0}=W\times \widehat{W}$
and $\left( \widehat{W}\times W\right) _{0}=\widehat{W}\times W$\ we may
consider the locally analytic functions%
\begin{equation*}
\delta _{W\times \widehat{W}}^{\mathbf{k}}:W\times \widehat{W}\rightarrow 
\mathcal{O}^{\times }\text{, }\delta _{\widehat{W}\times W}^{\mathbf{k}}:%
\widehat{W}\times W\rightarrow \mathcal{O}^{\times }\text{ and }\delta
_{\left( W\times W\right) _{0}}^{\mathbf{k}}:W\times W\rightarrow \mathcal{O}%
^{\times }\text{.}
\end{equation*}%
Recall our notation for the twists by the norm. Since $\Gamma _{0}\left( p%
\mathbb{Z}_{p}\right) \subset \mathbf{GL}_{2}\left( \mathbb{Q}_{p}\right) $
is compact, $\mathrm{nrd}_{p}$ maps it into the maximal open compact
subgroup $\mathbb{Z}_{p}^{\times }\subset \mathbb{Q}_{p}^{\times }$. Hence,
if $D$ is a $\Gamma _{0}\left( p\mathbb{Z}_{p}\right) $-module with
coefficients in $\mathcal{O}$, it makes sense to consider $D\left( \mathbf{k}%
\right) :=D\left( \mathrm{nrd}_{p}^{\mathbf{k}}\right) $, the same
representation with action $v\cdot _{\mathbf{k}}g:=\mathrm{nrd}_{p}^{\mathbf{%
k}}\left( g\right) vg$. Also, recall we have $\mathrm{Nrd}_{p}^{\mathbf{k}%
}\in M_{p}\left( \mathcal{O}\left( \mathbf{k}\right) ,\mathrm{N}_{p}^{2%
\mathbf{k}}\right) ^{K}$ for every $K\in \mathcal{K}^{\diamond }$ (indeed $%
\mathrm{N}_{p}^{2\mathbf{k}}=\mathrm{Nrd}_{p}^{\mathbf{k}}$ on $Z_{f}=%
\mathbb{A}_{f}^{\times }$).

\bigskip

If $\underline{\mathbf{k}}=\left( \mathbf{k}_{1},\mathbf{k}_{2},\mathbf{k}%
_{3}\right) $ where $\mathbf{k}_{i}:\mathbb{Z}_{p}^{\times }\rightarrow 
\mathcal{O}_{i}^{\times }$ are $\mathcal{O}_{i}$-valued weights such that $%
\mathbf{k}_{1}\oplus \mathbf{k}_{2}\oplus \mathbf{k}_{3}$ is even, set $%
\underline{\mathbf{k}}^{\ast }:=\frac{\mathbf{k}_{1}\oplus \mathbf{k}%
_{2}\oplus \mathbf{k}_{3}}{2}$, $\underline{\mathbf{k}}_{1}^{\ast }:=\frac{%
\ominus \mathbf{k}_{1}\oplus \mathbf{k}_{2}\oplus \mathbf{k}_{3}}{2}$, $%
\underline{\mathbf{k}}_{2}^{\ast }:=\frac{\mathbf{k}_{1}\ominus \mathbf{k}%
_{2}\oplus \mathbf{k}_{3}}{2}$ and $\underline{\mathbf{k}}_{3}^{\ast }:=%
\frac{\mathbf{k}_{1}\oplus \mathbf{k}_{2}\ominus \mathbf{k}_{3}}{2}$, so
that $\underline{\mathbf{k}}_{i}^{\ast }:\mathbb{Z}_{p}^{\times }\rightarrow 
\mathcal{O}_{\underline{\mathbf{k}}}^{\times }$ for $\mathcal{O}_{\underline{%
\mathbf{k}}}:=\mathcal{O}_{1}\widehat{\otimes }\mathcal{O}_{2}\widehat{%
\otimes }\mathcal{O}_{3}$. We define $\underline{W}:=W\times W\times W$, $%
\underline{W}_{1}:=\widehat{W}\times W\times W$, $\underline{W}_{2}:=W\times 
\widehat{W}\times W$ and $\underline{W}_{3}:=W\times W\times \widehat{W}$.
Also, if $p_{i}:\underline{W}_{i}\rightarrow W\times W$ denotes the
projection onto the components which are different from $i$, we define $%
\underline{W}_{i}^{\circ }:=p_{i}^{-1}\left( \left( W\times W\right)
_{0}\right) $ (for example, $\underline{W}_{3}^{\circ }:=\left( W\times
W\right) _{0}\times \widehat{W}$). Then we define the locally analytic
functions%
\begin{equation*}
\Delta _{i,\underline{\mathbf{k}}}^{\circ }:\underline{W}_{i}^{\circ
}\rightarrow \mathcal{O}_{\underline{\mathbf{k}}}^{\times }
\end{equation*}%
by the rule%
\begin{eqnarray*}
\Delta _{1,\underline{\mathbf{k}}}^{\circ }\left( w_{1},w_{2},w_{3}\right)
&:&=\delta _{\left( W\times W\right) _{0}}^{\underline{\mathbf{k}}_{1}^{\ast
}}\left( w_{2},w_{3}\right) \delta _{\widehat{W}\times W}^{\underline{%
\mathbf{k}}_{2}^{\ast }}\left( w_{1},w_{3}\right) \delta _{\widehat{W}\times
W}^{\underline{\mathbf{k}}_{3}^{\ast }}\left( w_{1},w_{2}\right) \text{,} \\
\Delta _{2,\underline{\mathbf{k}}}^{\circ }\left( w_{1},w_{2},w_{3}\right)
&:&=\delta _{\widehat{W}\times W}^{\underline{\mathbf{k}}_{1}^{\ast }}\left(
w_{2},w_{3}\right) \delta _{\left( W\times W\right) _{0}}^{\underline{%
\mathbf{k}}_{2}^{\ast }}\left( w_{1},w_{3}\right) \delta _{W\times \widehat{W%
}}^{\underline{\mathbf{k}}_{3}^{\ast }}\left( w_{1},w_{2}\right) \text{,} \\
\Delta _{3,\underline{\mathbf{k}}}^{\circ }\left( w_{1},w_{2},w_{3}\right)
&:&=\delta _{W\times \widehat{W}}^{\underline{\mathbf{k}}_{1}^{\ast }}\left(
w_{2},w_{3}\right) \delta _{W\times \widehat{W}}^{\underline{\mathbf{k}}%
_{2}^{\ast }}\left( w_{1},w_{3}\right) \delta _{\left( W\times W\right)
_{0}}^{\underline{\mathbf{k}}_{3}^{\ast }}\left( w_{1},w_{2}\right) \text{.}
\end{eqnarray*}%
We remark that $\Gamma _{0}\left( p\mathbb{Z}_{p}\right) =\Gamma _{0}\left( p%
\mathbb{Z}_{p}\right) ^{\iota }\subset \Sigma _{0}\left( p\mathbb{Z}%
_{p}\right) ^{\iota }$ acts diagonally on $\underline{W}_{i}^{\circ }$. The
following lemma is an application of $\left( \text{\ref{F det 1}}\right) $, $%
\left( \text{\ref{F det 2}}\right) $ and the definitions of \S \ref{SS
Algebraic operations on weights}.

\begin{lemma}
\label{L 3-forms}We have $\Delta _{i,\underline{\mathbf{k}}}^{\circ }\in 
\mathcal{A}_{\mathbf{k}_{1}\boxplus \mathbf{k}_{2}\boxplus \mathbf{k}%
_{3}}\left( \underline{W}_{i}^{\circ }\right) \left( -\underline{\mathbf{k}}%
^{\ast }\right) ^{\Gamma _{0}\left( p\mathbb{Z}_{p}\right) }$.
\end{lemma}

\bigskip

We will now focus on the $i=3$ index, the other cases being similar. Since $%
\underline{W}_{3}=\underline{W}_{3}^{\circ }\sqcup \left( \underline{W}_{3}-%
\underline{W}_{3}^{\circ }\right) $ (resp. $W^{2}=\left( W^{2}\right)
_{0}\sqcup \left( W^{2}-\left( W^{2}\right) _{0}\right) $)\ is a disjoint
decomposition in open subsets, we have an extension by zero map $\cdot
_{\circ }:\mathcal{A}_{\underline{\mathbf{k}}}\left( \underline{W}%
_{3}^{\circ }\right) \rightarrow \mathcal{A}_{\underline{\mathbf{k}}}\left( 
\underline{W}_{3}\right) $ (resp. $\cdot _{0}:\mathcal{A}_{\mathbf{k}%
_{1}\boxplus \mathbf{k}_{2}}\left( \left( W^{2}\right) _{0}\right)
\rightarrow \mathcal{A}_{\mathbf{k}_{1}\boxplus \mathbf{k}_{2}}\left(
W^{2}\right) $). By duality, we obtain a map%
\begin{equation*}
\cdot ^{\circ }:\mathcal{D}\left( \underline{W}_{3},\underline{\mathbf{k}}%
\right) \rightarrow \mathcal{D}\left( \underline{W}_{3}^{\circ },\underline{%
\mathbf{k}}\right) \text{ (resp. }\cdot ^{0}:\mathcal{D}_{\mathbf{k}%
_{1}\boxplus \mathbf{k}_{2}}\left( W\times W\right) \rightarrow \mathcal{D}_{%
\mathbf{k}_{1}\boxplus \mathbf{k}_{2}}\left( \left( W\times W\right)
_{0}\right) \text{).}
\end{equation*}%
It follows from Lemma \ref{L 3-forms} that we may consider%
\begin{equation*}
\Lambda _{3,\underline{\mathbf{k}}}^{\circ }\left( \mu _{1},\mu _{2},\mu
_{3}\right) :=\left( \mu _{1}\boxtimes \mu _{2}\boxtimes \mu _{3}\right)
^{\circ }\left( \Delta _{3,\underline{\mathbf{k}}}^{\circ }\right) \in 
\mathcal{O}_{\underline{\mathbf{k}}}\text{ (}\mu _{i}\in \mathcal{D}_{%
\mathbf{k}_{i}}\left( W\right) \text{ for }i=1,2\text{ and }\mu _{3}\in 
\mathcal{D}_{\mathbf{k}_{3}}\left( \widehat{W}\right) \text{)}
\end{equation*}%
and that we have%
\begin{equation*}
\Lambda _{3,\underline{\mathbf{k}}}^{\circ }\in Hom_{\mathcal{O}\left[
\Gamma _{0}\left( p\mathbb{Z}_{p}\right) \right] }\left( \mathcal{D}_{%
\mathbf{k}_{1}}\left( W\right) \otimes _{\mathcal{O}}\mathcal{D}_{\mathbf{k}%
_{2}}\left( W\right) \otimes _{\mathcal{O}}\mathcal{D}_{\mathbf{k}%
_{3}}\left( \widehat{W}\right) ,\mathcal{O}_{\underline{\mathbf{k}}}\left( 
\underline{\mathbf{k}}^{\ast }\right) \right) .
\end{equation*}%
Suppose that we have given characters $\omega _{0,p}^{\mathbf{k}_{i}}$ for $%
i=1,2,3$ such that $\mathrm{Nrd}_{p}^{\underline{\mathbf{k}}^{\ast }}=\omega
_{0,p}^{\mathbf{k}_{1}}\omega _{0,p}^{\mathbf{k}_{2}}\omega _{0,p}^{\mathbf{k%
}_{3}}$. Taking $\Lambda =\Lambda _{\underline{\mathbf{k}}}^{\circ }$ in $%
\left( \text{\ref{p-adic F Prof n-Form2}}\right) $) gives the trilinear form%
\begin{equation*}
t_{3,\underline{\mathbf{k}}}^{\circ }:M_{p}\left( \mathcal{D}_{\mathbf{k}%
_{1}}\left( W\right) ,\omega _{0,p}^{\mathbf{k}_{1}}\right) \otimes
M_{p}\left( \mathcal{D}_{\mathbf{k}_{2}}\left( W\right) ,\omega _{0,p}^{%
\mathbf{k}_{2}}\right) \otimes M_{p}\left( \mathcal{D}_{\mathbf{k}%
_{3}}\left( \widehat{W}\right) ,\omega _{0,p}^{\mathbf{k}_{3}}\right)
\rightarrow \mathcal{O}_{\underline{\mathbf{k}}}\text{.}
\end{equation*}%
Suppose, for example, that we may write $\omega _{0,p}^{\mathbf{k}%
_{i}}\left( z\right) =\omega _{0,i}\left( \frac{z}{\mathrm{N}_{f}\left(
z\right) }\right) \left( \frac{z}{\mathrm{N}_{f}\left( z\right) }\right)
_{p}^{-\mathbf{k}_{i}}$ with $\omega _{0,i}$\ taking values in $F$ (as in
Example \ref{Example Nebetype1}) with $\omega _{0,1}\omega _{0,2}\omega
_{0,3}=1$; noticing that $\left( \frac{z}{\mathrm{N}_{f}\left( z\right) }%
\right) _{p}^{-\left( \mathbf{k}_{1}\oplus \mathbf{k}_{2}\oplus \mathbf{k}%
_{3}\right) }=\mathrm{Nrd}_{p}^{\underline{\mathbf{k}}^{\ast }}\left(
z\right) $ the above definition applies. When $\mathbf{k}_{i}$ extends to a
character $\mathbf{k}_{i}:\mathbb{Q}_{p}^{\times }\rightarrow \mathcal{O}%
_{i}^{\times }$, $\Delta _{3,\underline{\mathbf{k}}}^{\circ }$ makes sense
as an element $\Delta _{3,\underline{\mathbf{k}}}\in \mathcal{A}_{\mathbf{k}%
_{1}\boxplus \mathbf{k}_{2}\boxplus \mathbf{k}_{3}}\left( \underline{W}%
_{3}\right) \left( -\underline{\mathbf{k}}^{\ast }\right) ^{\Gamma
_{0}\left( p\mathbb{Z}_{p}\right) }$. We can therefore integrate this
function without first applying $\cdot ^{0}$ to the measures involved. The
result is a trilinear form%
\begin{equation*}
t_{3,\underline{\mathbf{k}}}:M_{p}\left( \mathcal{D}_{\mathbf{k}_{1}}\left(
W\right) ,\omega _{0,p}^{\mathbf{k}_{1}}\right) \otimes M_{p}\left( \mathcal{%
D}_{\mathbf{k}_{2}}\left( W\right) ,\omega _{0,p}^{\mathbf{k}_{2}}\right)
\otimes M_{p}\left( \mathcal{D}_{\mathbf{k}_{3}}\left( \widehat{W}\right)
,\omega _{0,p}^{\mathbf{k}_{3}}\right) \rightarrow \mathcal{O}_{\underline{%
\mathbf{k}}}\text{.}
\end{equation*}%
It follows from $\left( \text{\ref{Distributions F3}}\right) $ that, if $%
\varphi _{i,k_{i}}=\phi _{i\ast }^{\mathrm{al}\text{\textrm{g}}}\left(
\varphi _{i}\right) $ where $\mathbf{k}_{i}\overset{\phi _{i}}{\rightarrow }%
k_{i}\in \mathbb{N}$, then%
\begin{equation}
t_{3,\underline{k}}\left( \varphi _{1}\otimes \varphi _{2}\otimes \varphi
_{3}\mid W_{3}\right) =t_{\underline{k}}\left( \varphi _{1,k_{1}},\varphi
_{2,k_{2}},\varphi _{3,k_{3}}\mid W_{p}\right) \text{.}
\label{F Interpolation}
\end{equation}

\bigskip

The following key calculation relates the trilinear form $t_{3,\underline{k}%
}^{\circ }$ to $t_{3,\underline{k}}$. Write $\widehat{\pi }_{p}$ for the
idele concentrated at $p$, where we have $\left( \widehat{\pi }_{p}\right)
_{p}=\pi _{p}:=\left( 
\begin{array}{cc}
1 & 0 \\ 
0 & p%
\end{array}%
\right) $. If $\varphi \in M_{p}(\mathcal{D}_{k}(W))^{\Gamma _{0}\left( p%
\mathbb{Z}_{p}\right) }$, the $U_{p}$-operator is defined by the double
coset $K\widehat{\pi }_{p}K$, where $K=K^{p}\Gamma _{0}\left( p\mathbb{Z}%
_{p}\right) $ and $\varphi \in M_{p}(\mathcal{D}_{k}(W))^{K^{p}\Gamma
_{0}\left( p\mathbb{Z}_{p}\right) }$ with $K^{p}\subset B^{\times p}$, the
prime to $p$ subgroup of $B_{f}^{\times }$.

\begin{proposition}
\label{P:afafsdsf} For $i=1,2$, suppose $\varphi _{i}\in M_{p}(\mathcal{D}%
_{k_{i}}(W))^{\Gamma _{0}\left( p\mathbb{Z}_{p}\right) }$ is a $U_{p}$%
-eigenvector with $\varphi _{i}|U_{p}=a_{i}\varphi _{i}$, and view $\varphi
_{1}\otimes \varphi _{2}$ as an element of $M_{p}(\mathcal{D}_{k_{1}\boxplus
k_{2}}(W\times W))$. Then 
\begin{equation*}
(\varphi _{1}\otimes \varphi _{2})|U_{p}=a_{1}a_{2}\big(\varphi _{1}\otimes
\varphi _{2}-i_{\ast }(\varphi _{1}\otimes \varphi _{2})^{0}\big),
\end{equation*}%
where $i_{\ast }:M_{p}(\mathcal{D}_{k_{1}\boxplus k_{2}}(\left( W\times
W\right) _{0}))\rightarrow M_{p}(\mathcal{D}_{k_{1}\boxplus k_{2}}(W\times
W))$ is induced by the map $i_{\ast }:\mathcal{D}_{k_{1}\boxplus
k_{2}}(\left( W\times W\right) _{0})\rightarrow \mathcal{D}_{k_{1}\boxplus
k_{2}}(W\times W)$ induced by the inclusion $i:\left( W\times W\right)
_{0}\hookrightarrow W\times W$.
\end{proposition}

\begin{proof}
Consider the decomposition 
\begin{equation*}
W=\bigsqcup_{i=0}^{p-1}W_{i},\quad \text{where}\quad W_{i}=W\pi
_{i}=\{w=\left( x,y\right) \in W:y\equiv ix\pmod{p}\}\text{.}
\end{equation*}%
Then $K\widehat{\pi }_{p}K=\bigsqcup_{i=0}^{p-1}K\pi _{i}$ and we can
compute:%
\begin{align*}
a_{1}a_{2}(\varphi _{1}\otimes \varphi _{2})(w)& =(\varphi _{1}|U_{p}\otimes
\varphi _{2}|U_{p})(w) \\
& =\sum_{i,j=0}^{p-1}\varphi _{1}(\pi _{i}w)\pi _{i}\otimes \varphi _{2}(\pi
_{j}w)\pi _{j} \\
& =\sum_{i=0}^{p-1}\varphi _{1}(\pi w)\pi _{i}\otimes \varphi _{3}(\pi
_{i}w)\pi _{i}+\sum_{%
\begin{smallmatrix}
i,j=0 \\ 
i\neq j%
\end{smallmatrix}%
}^{p-1}\varphi _{1}(\pi _{i}w)\pi _{i}\otimes \varphi _{2}(\pi _{j}w)\pi _{j}
\\
& =((\varphi _{1}\otimes \varphi _{2})|U_{p})(w)+A.
\end{align*}%
It remains to show that $A=a_{1}a_{2}i_{\ast }((\varphi _{1}\otimes \varphi
_{2})^{0}(w))$. To this end, note that we may write%
\begin{equation*}
W^{2}=\bigcup_{i,j=0}^{p-1}W_{i}\times W_{j}=\bigcup_{i=0}^{p-1}W_{i}\times
W_{i}\cup \bigcup_{%
\begin{smallmatrix}
i,j=0 \\ 
i\neq j%
\end{smallmatrix}%
}^{p-1}W_{i}\times W_{j}.
\end{equation*}%
Subordinate to this decomposition of spaces, we have a corresponding
decomposition of $\mathcal{D}_{k_{1}\boxplus k_{2}}(W^{2})$: 
\begin{equation*}
\mathcal{D}_{k_{1}\boxplus k_{2}}(W^{2})=\bigoplus_{i=0}^{p-1}\mathcal{D}%
_{k_{1}\boxplus k_{2}}(W_{i}\times W_{i})\oplus \bigoplus_{%
\begin{smallmatrix}
i,j=0 \\ 
i\neq j%
\end{smallmatrix}%
}^{p-1}\mathcal{D}_{k_{1}\boxplus k_{2}}(W_{i}\times W_{j}).
\end{equation*}%
Note that the spaces $W_{i}$ are all $\mathbb{Z}_{p}^{\times }$-stable, so
that these spaces of distributions are defined. Writing $\proj_{i,j}:%
\mathcal{D}_{k_{1}\boxplus k_{2}}(W^{2})\rightarrow \mathcal{D}%
_{k_{1}\boxplus k_{2}}(W_{i}\times W_{j})$ for the associated projections,
we have 
\begin{equation*}
+a_{1}a_{2}\sum_{i,j=0}^{p-1}\proj_{i,j}(\varphi _{1}(w)\otimes \varphi
_{2}(w))=a_{1}a_{2}(\varphi _{1}(w)\otimes \varphi
_{2}(w))=\sum_{i,j=0}^{p-1}\varphi _{1}(\pi _{i}w)\pi _{i}\otimes \varphi
_{2}(\pi _{j}w)\pi _{j}.
\end{equation*}%
Since 
\begin{equation*}
\varphi _{1}(\pi _{i}w)\pi _{i}\otimes \varphi _{2}(\pi _{j}w)\pi _{j}\in 
\mathcal{D}_{k_{1}\boxplus k_{2}}(W_{i}\times W_{j}),
\end{equation*}%
it follows that 
\begin{equation*}
a_{1}a_{2}\proj_{i,j}(\varphi _{1}(w)\otimes \varphi _{2}(w))=\varphi
_{1}(\pi _{i}w)\pi _{i}\otimes \varphi _{2}(\pi _{j}w).
\end{equation*}%
for all $i,j$. Therefore, 
\begin{equation}
A=a_{1}a_{2}\sum_{%
\begin{smallmatrix}
i,j=0 \\ 
i\neq j%
\end{smallmatrix}%
}^{p-1}\proj_{i,j}(\varphi _{1}(w)\otimes \varphi _{2}(w)).
\label{E:nskwykf}
\end{equation}%
One easily verifies the equality 
\begin{equation*}
\left( W\times W\right) ^{0}=\bigcup_{%
\begin{smallmatrix}
i,j=0 \\ 
i\neq j%
\end{smallmatrix}%
}^{p-1}W_{i}\times W_{j},
\end{equation*}%
implying that 
\begin{equation}
\sum_{%
\begin{smallmatrix}
i,j=0 \\ 
i\neq j%
\end{smallmatrix}%
}^{p-1}\proj_{i,j}(\varphi _{1}(w)\otimes \varphi _{2}(w))=i_{\ast
}((\varphi _{1}\otimes \varphi _{2})^{0}(w)).  \label{E:sdkfky}
\end{equation}%
Now substitute~\eqref{E:sdkfky} into~\eqref{E:nskwykf}.
\end{proof}

\bigskip

We are going to apply the results of \S \ref{SSS Adjointness}. We take $%
\Sigma _{p}=\Sigma _{0}\left( p\mathbb{Z}_{p}\right) $, $D=\mathcal{D}%
_{k_{1}\boxplus k_{2}}(W\times W)$ and $E=\mathcal{D}_{k_{3}}(\widehat{W})$
and $\mathrm{n}_{p}=\mathrm{nrd}_{p}$. Then $\mathbf{k}=\underline{k}^{\ast
} $ (resp. the central character $\kappa _{E}=k_{3}$ of $E=\mathcal{D}%
_{k_{3}}(\widehat{W})$) extends to the character $\underline{k}^{\ast }$ of $%
\mathbb{Q}_{p}^{\times }$ (resp. the character $\widetilde{\kappa _{E}}%
=k_{3} $\ of $\mathbb{Q}_{p}^{\times }$). Finally, justified by Examples \ref%
{Example Nebetype1} and \ref{Example Nebetype2},\ we suppose that we may
write $\omega _{0,p}^{k_{i}}\left( z\right) =\omega _{0,i}\left( \frac{z}{%
\mathrm{N}_{f}\left( z\right) }\right) \left( \frac{z}{\mathrm{N}_{f}\left(
z\right) }\right) _{p}^{-k_{i}}$ with $\omega _{0,1}\omega _{0,2}\omega
_{0,3}=1$. Then $\omega _{0,p,D}\left( z\right) =\omega _{0,1}\left( \frac{z%
}{\mathrm{N}_{f}\left( z\right) }\right) \omega _{0,2}\left( \frac{z}{%
\mathrm{N}_{f}\left( z\right) }\right) \left( \frac{z}{\mathrm{N}_{f}\left(
z\right) }\right) _{p}^{-k_{1}-k_{2}}$, $\omega _{0,p,E}\left( z\right)
=\omega _{0,3}\left( \frac{z}{\mathrm{N}_{f}\left( z\right) }\right) \left( 
\frac{z}{\mathrm{N}_{f}\left( z\right) }\right) _{p}^{-k_{3}}$ and we have $%
\omega _{0,p,D}\omega _{0,p,E}=\omega _{0,p}$ with $\omega _{0,p}\left(
z\right) =\left( \frac{z}{\mathrm{N}_{f}\left( z\right) }\right)
_{p}^{-k_{1}-k_{2}-k_{3}}=\mathrm{Nrd}_{p}^{\underline{k}^{\ast }}\left(
z\right) $.

\begin{lemma}
\label{L:afafsdsf}With these notations the trilinear form $t_{3,\underline{k}%
}$ defines an element of $\Hom_{\mathcal{O}\left[ \Sigma _{p},\Sigma
_{p}^{\iota }\right] }(D\otimes E,\mathcal{O}(\underline{k}^{\ast }))$.
Furthermore, we have $\mathrm{Nrd}_{f}^{\widetilde{\mathbf{k}}}\left( \pi
\right) _{p}\mathrm{nrd}_{p}^{-\widetilde{\kappa _{E}}}\left( \pi
_{p}\right) \mathrm{nrd}_{f}^{-\omega _{0,p,E}}\left( \pi \right) =\omega
_{0,3}\left( \frac{\mathrm{Nrd}_{f}\left( \pi \right) }{\mathrm{nrd}\left(
\pi \right) }\right) \mathrm{Nrd}_{f}^{\underline{k}_{3}^{\ast }}\left( \pi
\right) _{p}$ in Proposition \ref{P:heckeadjointprop}.
\end{lemma}

\begin{proof}
Note that $\Delta _{i,\underline{k}}$ defines indeed $\widetilde{\Delta }_{%
\underline{k}}\in \mathcal{A}_{k_{1}\boxplus k_{2}\boxplus k_{3}}(\mathbb{Q}%
_{p}^{2}\times \mathbb{Q}_{p}^{2}\times \mathbb{Q}_{p}^{2})$ such that $%
\widetilde{\Delta }_{\underline{k}\mid \underline{W}_{3}}=\Delta _{i,%
\underline{k}}$. We take $\widetilde{D}:=\mathcal{D}_{k_{1}\boxplus k_{2}}(%
\mathbb{Q}_{p}^{2}\times \mathbb{Q}_{p}^{2})$ and $\widetilde{E}:=\mathcal{D}%
_{k_{3}}(\mathbb{Q}_{p}^{2})$, so that $\mathbf{k}_{\widetilde{E}}=%
\widetilde{\mathbf{k}}_{E}=k_{3}$. The pairing associated to $t_{3,%
\underline{k}}$ is given by $\left\langle \mu _{12},\mu _{3}\right\rangle
:=(\mu _{12}\widehat{\otimes }\mu _{3})(\Delta _{\underline{k}})$ and we
define $\left\langle \mu _{12},\mu _{3}\right\rangle ^{\sim }:=(\mu _{12}%
\widehat{\otimes }\mu _{3})(\widetilde{\Delta }_{\underline{k}})$. Since $%
\mathbb{Q}_{p}^{2}\times \mathbb{Q}_{p}^{2}\times \mathbb{Q}%
_{p}^{2}=(W^{2}\times \widehat{W})\sqcup Z$ with $Z$ an open subset, 
\begin{align*}
\mathcal{A}_{k_{1}\boxplus k_{2}\boxplus k_{3}}(\mathbb{Q}_{p}^{2}\times 
\mathbb{Q}_{p}^{2}\times \mathbb{Q}_{p}^{2})& =\mathcal{A}_{k_{1}\boxplus
k_{2}\boxplus k_{3}}(W^{2}\times \widehat{W})\oplus \mathcal{A}%
_{k_{1}\boxplus k_{2}\boxplus k_{3}}(Z) \\
\text{and}\quad \mathcal{D}_{k_{1}\boxplus k_{2}\boxplus k_{3}}(\mathbb{Q}%
_{p}^{2}\times \mathbb{Q}_{p}^{2}\times \mathbb{Q}_{p}^{2})& =\mathcal{D}%
_{k_{1}\boxplus k_{2}\boxplus k_{3}}(W^{2}\times \widehat{W})\oplus \mathcal{%
D}_{k_{1}\boxplus k_{2}\boxplus k_{3}}(Z).
\end{align*}%
For elements $\mu _{12}\in \mathcal{D}_{k_{1}\boxplus k_{2}}(W^{2})$ and $%
\mu _{3}\in \mathcal{D}_{k_{3}}(\widehat{W})$, the distribution $\mu _{12}%
\widehat{\otimes }\mu _{3}$ is supported on $\mathcal{D}_{k_{1}\boxplus
k_{2}\boxplus k_{3}}(W^{2}\times \widehat{W})$ and we see that $\left\langle
\mu _{12},\mu _{3}\right\rangle ^{\sim }=\left\langle \mu _{12},\mu
_{3}\right\rangle $. We note that the relation $\sigma \widetilde{\Delta }_{%
\underline{k}}=\det (\sigma )^{\underline{k}}\widetilde{\Delta }_{\underline{%
k}}$ for every $\sigma \in \Sigma _{p}$\ implies%
\begin{equation*}
\left\langle \mu _{12}\sigma ,\mu _{3}\sigma \right\rangle ^{\sim }:=(\mu
_{12}\sigma \widehat{\otimes }\mu _{3}\sigma )(\widetilde{\Delta }_{%
\underline{k}})=\det (\sigma )^{\underline{k}}(\mu _{12}\widehat{\otimes }%
\mu _{3})(\widetilde{\Delta }_{\underline{k}})=\det (\sigma )^{\underline{k}%
}\left\langle \mu _{12},\mu _{3}\right\rangle ^{\sim }\text{.}
\end{equation*}%
Now apply Remark \ref{R1} in order to get the first statement. Finally, the
second statement follows by a simple computation.
\end{proof}

\bigskip

Write $\mathbf{p}^{\prime }\in \mathbb{A}_{f}^{\times }$ for the finite
idele $\left( \mathbf{p}^{\prime }\right) _{v}=p$ for every $v\neq p$ and $%
\left( \mathbf{p}^{\prime }\right) _{p}=1$.

\begin{corollary}
\label{C:afafsdsf}For $i=1,2$, let $\varphi _{i}\in M_{p}(\mathcal{D}%
_{k_{i}}(W),\omega _{0,p}^{k_{i}})^{\Gamma _{0}\left( p\mathbb{Z}_{p}\right)
}$ be a $U_{p}$-eigenvector with $\varphi _{i}|U_{p}=\alpha _{i}\varphi _{i}$
and let $\varphi _{3}\in M_{p}(\mathcal{D}_{k_{3}}(\widehat{W}),\omega
_{0,p}^{k_{3}})^{\Gamma _{0}\left( p\mathbb{Z}_{p}\right) }$ be a $%
U_{p}^{\iota }$-eigenvector with $\varphi _{3}|U_{p}^{\iota }=\alpha
_{3}\varphi _{3}$. Then%
\begin{equation*}
t_{3,\underline{k}}^{\circ }(\varphi _{1}\otimes \varphi _{2}\otimes \varphi
_{3})=(1-\omega _{0,3}\left( \mathbf{p}^{\prime }\right) \frac{\alpha _{3}}{%
\alpha _{1}\alpha _{2}}p^{\underline{k}_{3}^{\ast }})t_{3,\underline{k}%
}(\varphi _{1}\otimes \varphi _{2}\otimes \varphi _{3}).
\end{equation*}
\end{corollary}

\begin{proof}
Recall the morphism $\cdot ^{0}:\mathcal{D}_{\mathbf{k}_{1}\boxplus \mathbf{k%
}_{2}}\left( W\times W\right) \rightarrow \mathcal{D}_{k_{1}\boxplus
k_{2}}\left( \left( W\times W\right) _{0}\right) $. Given $\mu _{i}\in 
\mathcal{D}_{\mathbf{k}_{i}}\left( W\right) $ for $i=1,2$ and $\mu _{3}\in 
\mathcal{D}_{\mathbf{k}_{3}}\left( \widehat{W}\right) $, we may therefore
consider%
\begin{equation*}
\left\langle \mu _{12},\mu _{3}\right\rangle _{t}^{\circ }:=\left( \mu
_{12}^{0}\boxtimes \mu _{3}\right) \left( \Delta _{\underline{\mathbf{k}}%
}^{\circ }\right) \text{.}
\end{equation*}%
This is granted by Lemma \ref{L 3-forms}, which also implies $\left\langle
-,-\right\rangle _{t}^{\circ }\in Hom_{\mathcal{O}\left[ \Gamma _{0}\left( p%
\mathbb{Z}_{p}\right) \right] }\left( \mathcal{D}_{\mathbf{k}_{1}\boxplus 
\mathbf{k}_{2}}\left( \left( W^{2}\right) _{0}\right) \otimes \mathcal{D}_{%
\mathbf{k}_{3}}\left( \widehat{W}\right) ,\mathcal{O}\left( \underline{%
\mathbf{k}}^{\ast }\right) \right) $. Taking $\Lambda =\left\langle
-,-\right\rangle _{t}^{\circ }$ in $\left( \text{\ref{p-adic F Prof n-Form2}}%
\right) $ gives the bilinear form%
\begin{equation*}
\left\langle -,-\right\rangle _{t}^{\circ }:M_{p}\left( \mathcal{D}_{\mathbf{%
k}_{1}\boxplus \mathbf{k}_{2}}\left( \left( W^{2}\right) _{0}\right) ,\omega
_{0,p,D}\right) \otimes M_{p}\left( \mathcal{D}_{\mathbf{k}_{3}}\left( 
\widehat{W}\right) ,\omega _{0,p,E}\right) \rightarrow \mathcal{O}\text{,}
\end{equation*}%
It is clear that $\left\langle \mathbf{P}^{\mathbf{k}_{1},\mathbf{k}%
_{2}}\left( \mu _{1}\widehat{\otimes }_{\iota }\mu _{2}\right) ^{0},\mu
_{3}\right\rangle _{t}^{\circ }=t_{3,\underline{\mathbf{k}}}^{\circ }\left(
\mu _{1},\mu _{2},\mu _{3}\right) $ (we just need to check the equality on
Dirac distributions), from which we see that%
\begin{equation*}
t_{3,\underline{\mathbf{k}}}^{\circ }\left( \varphi _{1},\varphi
_{2},\varphi _{3}\right) =\left\langle \left( \varphi _{1}\otimes \varphi
_{2}\right) ^{0},\varphi _{3}\right\rangle _{t}^{\circ }\text{,}
\end{equation*}%
if $\varphi _{1}\otimes \varphi _{2}$ is viewed as an element of $%
M_{p}\left( \mathcal{D}_{\mathbf{k}_{1}\boxplus \mathbf{k}_{2}}\left(
W^{2}\right) ,\omega _{0,p,D}\right) $. A similar result holds true for $%
t_{3,\underline{k}}$, namely we may define as above%
\begin{equation*}
\left\langle -,-\right\rangle _{t}:M_{p}\left( \mathcal{D}_{\mathbf{k}%
_{1}\boxplus \mathbf{k}_{2}}\left( W^{2}\right) ,\omega _{0,p,D}\right)
\otimes M_{p}\left( \mathcal{D}_{\mathbf{k}_{3}}\left( W\right) ,\omega
_{0,p,E}\right) \rightarrow \mathcal{O}
\end{equation*}%
for which%
\begin{equation*}
t_{3,\underline{k}}\left( \varphi _{1},\varphi _{2},\varphi _{3}\right)
=\left\langle \varphi _{1}\otimes \varphi _{2},\varphi _{3}\right\rangle _{t}%
\text{.}
\end{equation*}

By Proposition~\ref{P:afafsdsf}, we have 
\begin{equation}
\langle \varphi _{1}\otimes \varphi _{2},\varphi _{3}\rangle _{t}=\frac{1}{%
\alpha _{1}\alpha _{3}}\langle (\varphi _{1}\otimes \varphi
_{2})|U_{p},\varphi _{3}\rangle _{t}+\langle i_{\ast }(\varphi _{1}\otimes
\varphi _{2})^{0},\varphi _{3}\rangle _{t}  \label{E:aaa}
\end{equation}%
Proposition \ref{P:heckeadjointprop}, which applies thanks to Lemma \ref%
{L:afafsdsf} implies that 
\begin{equation}
\langle (\varphi _{1}\otimes \varphi _{2})|U_{p},\varphi _{3}\rangle
_{t}=\omega _{0,3}\left( \mathbf{p}^{\prime }\right) p^{\underline{k}%
_{3}^{\ast }}\langle \varphi _{1}\otimes \varphi _{2},\varphi
_{3}|U_{p}^{\iota }\rangle _{t}=\omega _{0,3}\left( \mathbf{p}^{\prime
}\right) p^{\underline{k}_{3}^{\ast }}\alpha _{3}\langle \varphi _{1}\otimes
\varphi _{2},\varphi _{3}\rangle _{t}.  \label{E:bbb}
\end{equation}%
Finally, it is easy to see that $\left\langle i_{\ast }\left( \mu
_{1}\otimes \mu _{2}\right) ^{0},\mu _{3}\right\rangle _{t}=\left\langle
\left( \mu _{1}\otimes \mu _{2}\right) ^{0},\mu _{3}\right\rangle
_{t}^{\circ }$ (once again checking the equality on Dirac distributions),
from which we see that%
\begin{equation}
\left\langle i_{\ast }\left( \varphi _{1}\otimes \varphi _{2}\right)
^{0},\varphi _{3}\right\rangle _{t}=\left\langle \left( \varphi _{1}\otimes
\varphi _{2}\right) ^{0},\varphi _{3}\right\rangle _{t}^{\circ }\text{.}
\label{E:ccc}
\end{equation}%
The result follows by combining~\eqref{E:aaa}, \eqref{E:bbb}, and~%
\eqref{E:ccc}.
\end{proof}

\section{Proof of the interpolation property}

\subsection{\label{SS p-Stabilization}Degeneracy maps and $p$-stabilizations}

Write $\widehat{\pi }_{p}$ (resp. $\widehat{\omega }_{p}$)\ for the idele
concentrated at $p$, where we have%
\begin{equation*}
\left( \widehat{\pi }_{p}\right) _{p}=\pi _{p}:=\left( 
\begin{array}{cc}
1 & 0 \\ 
0 & p%
\end{array}%
\right) \text{, }\left( \widehat{\omega }_{p}\right) _{p}=\omega
_{p}:=\left( 
\begin{array}{cc}
0 & -1 \\ 
p & 0%
\end{array}%
\right) \text{.}
\end{equation*}%
We fix levels $K\subset K^{\#}$ where $K_{v}=K_{v}^{\#}$ for every (finite)$%
\ $place $v\neq p$ and $K_{p}=\Gamma _{0}\left( p\mathbb{Z}_{p}\right)
\subset \mathbf{GL}_{2}\left( \mathbb{Z}_{p}\right) =K_{p}^{\#}$. Justified
by Example \ref{Example Nebetype1}, we suppose in this \S \ref{SS
p-Stabilization} that we may write $\omega _{0,p}\left( z\right) =\omega
_{0}\left( \frac{z}{\mathrm{N}_{f}\left( z\right) }\right) \mathrm{N}%
_{f}^{k}\left( z\right) $. We have have two degeneracy maps%
\begin{equation*}
K^{\#}1K=K^{\#}1,K^{\#}\widehat{\pi }_{p}^{\iota }K:M\left( \mathbf{V}%
_{k,F},\omega _{0,p}\right) ^{K^{\#}}\rightarrow M\left( \mathbf{V}%
_{k,F},\omega _{0,p}\right) ^{K}\text{.}
\end{equation*}%
Define%
\begin{equation*}
\varphi ^{\left( p\right) }:=\varphi \mid K^{\#}\widehat{\pi }_{p}^{\iota
}K\in M\left( \mathbf{V}_{k,F},\omega _{0,p}\right) ^{K}\text{.}
\end{equation*}

Let us now fix $0\neq \varphi \in M\left( \mathbf{V}_{k,F},\omega
_{0,p}\right) ^{K^{\#}}$ such that $\varphi \mid T_{p}=a_{p}\left( \varphi
\right) \varphi $ and define $M\left( \mathbf{V}_{k,F},\omega _{0,p}\right)
^{K,\varphi \text{\textrm{-old}}}\subset M\left( \mathbf{V}_{k,F},\omega
_{0,p}\right) ^{K}$ to be the span of $\left\{ \varphi ,\varphi ^{\left(
p\right) }\right\} $. Let%
\begin{equation*}
X^{2}-a_{p}\left( \varphi \right) X+\omega _{0}\left( \mathbf{p}^{\prime
}\right) p^{k+1}=\left( X-\alpha _{p}\left( \varphi \right) \right) \left(
X-\beta _{p}\left( \varphi \right) \right)
\end{equation*}%
be the Hecke polynomial at $p$ attached to $\varphi $ and define%
\begin{eqnarray*}
\varphi ^{\alpha } &=&\varphi ^{\alpha _{p}\left( \varphi \right) }:=\varphi
-\alpha _{p}\left( \varphi \right) ^{-1}\varphi ^{\left( p\right) }\in
M\left( \mathbf{V}_{k,F},\omega _{0,p}\right) ^{K,\varphi \text{\textrm{-old}%
}}\text{,} \\
\varphi ^{\beta } &=&\varphi ^{\beta _{p}\left( \varphi \right) }:=\varphi
-\beta _{p}\left( \varphi \right) ^{-1}\varphi ^{\left( p\right) }\in
M\left( \mathbf{V}_{k,F},\omega _{0,p}\right) ^{K,\varphi \text{\textrm{-old}%
}}\text{.}
\end{eqnarray*}%
We say that $\varphi $ is supersingular if $\alpha _{p}\left( \varphi
\right) =\beta _{p}\left( \varphi \right) $.

The following result is standard and we leave the proof to the reader.

\begin{proposition}
\label{p-stabilization P1}The following facts are true, assuming that $F$ is
a field such that $\alpha _{p}\left( \varphi \right) ,\beta _{p}\left(
\varphi \right) \in F$ for the statements $2.-5.$.

\begin{itemize}
\item[$\left( 1\right) $] The space $M\left( \mathbf{V}_{k,F},\omega
_{0,p}\right) ^{K,\varphi \text{\textrm{-old}}}$ is two dimensional with
basis $\left\{ \varphi ,\varphi ^{\left( p\right) }\right\} $, stable under
the action of the $U_{p}$ and $W_{p}$ operators.

\item[$\left( 2\right) $] We have $\varphi ^{\alpha }\mid U_{p}=\alpha
_{p}\left( \varphi \right) \varphi $, $\varphi ^{\beta }\mid U_{p}=\beta
_{p}\left( \varphi \right) \varphi $ and, if $\psi \in M\left( \mathbf{V}%
_{k,F},\omega _{0,p}\right) ^{K,\varphi \text{\textrm{-old}}}$ is such that $%
\psi \mid U_{p}=\rho \psi $, then $\psi =\varphi ^{\alpha }$ or $\psi
=\varphi ^{\beta }$ up to a scalar factor.

\item[$\left( 3\right) $] We have%
\begin{equation*}
F\left( \varphi ^{\alpha }\mid W_{p}\right) \cap F\varphi ^{\alpha }=F\left(
\varphi ^{\alpha }\mid W_{p}\right) \cap F\varphi ^{\beta }=0\text{ (resp. }%
F\left( \varphi ^{\beta }\mid W_{p}\right) \cap F\varphi ^{\alpha }=F\left(
\varphi ^{\beta }\mid W_{p}\right) \cap F\varphi ^{\beta }=0\text{)}
\end{equation*}%
unless $\alpha _{p}\left( \varphi \right) ^{2}=p^{k}$ (resp. $\beta
_{p}\left( \varphi \right) ^{2}=p^{k}$) and, in this case, $\varphi ^{\alpha
}\mid W_{p}=-\alpha _{p}\left( \varphi \right) \varphi ^{\alpha }$ (resp. $%
\varphi ^{\beta }\mid W_{p}=-\beta _{p}\left( \varphi \right) \varphi
^{\beta }$). In general,%
\begin{equation*}
\varphi ^{\alpha }\mid W_{p}=\varphi ^{\left( p\right) }-\alpha _{p}\left(
\varphi \right) ^{-1}p^{k}\varphi \text{ (resp. }\varphi ^{\beta }\mid
W_{p}=\varphi ^{\left( p\right) }-\beta _{p}\left( \varphi \right)
^{-1}p^{k}\varphi \text{).}
\end{equation*}

\item[$\left( 4\right) $] If $\varphi $ is not supersingular, then $U_{p}$
is diagonalizable on $M\left( \mathbf{V}_{k,F},\omega _{0,p}\right)
^{K,\varphi \text{\textrm{-old}}}$ and we have%
\begin{equation*}
M\left( \mathbf{V}_{k,F},\omega _{0,p}\right) ^{K,\varphi \text{\textrm{-old}%
}}=F\varphi ^{\alpha }\oplus F\varphi ^{\beta }\text{.}
\end{equation*}

\item[$\left( 5\right) $] If $\varphi $ is supersingular, then%
\begin{equation*}
0\neq F\varphi ^{\alpha }=F\varphi ^{\beta }\subset M\left( \mathbf{V}%
_{k,F},\omega _{0,p}\right) ^{K,\varphi \text{\textrm{-old}}}
\end{equation*}%
is a one dimensional subspace of $M\left( \mathbf{V}_{k,F},\omega
_{0,p}\right) ^{K,\varphi \text{\textrm{-old}}}$ and $U_{p}$ is not
diagonalizable on $M\left( \mathbf{V}_{k,F},\omega _{0,p}\right) ^{K,\varphi 
\text{\textrm{-old}}}$.
\end{itemize}
\end{proposition}

\bigskip

\begin{definition}
If $\underline{\varphi }=\left( \varphi _{1},\varphi _{2},\varphi
_{3}\right) $ is a vector with $\varphi _{i}\in M\left( \mathbf{V}%
_{k_{i},F},\omega _{0,p,i}\right) ^{K^{\#}}$ such that $\varphi _{i}\mid
T_{p}=a_{p}\left( \varphi _{i}\right) \varphi _{i}$, we set $\underline{k}%
^{\ast }:=\frac{k_{1}+k_{2}+k_{3}}{2}$\ and then we define%
\begin{eqnarray*}
\mathcal{E}_{p}\left( \underline{\varphi },X_{1},X_{2},X_{3}\right)
&:&=1-\left( X_{1}^{-1}a_{p}\left( \varphi _{1}\right)
+X_{2}^{-1}a_{p}\left( \varphi _{2}\right) +X_{3}^{-1}a_{p}\left( \varphi
_{3}\right) \right) \\
&&+X_{1}^{-1}X_{2}^{-1}a_{p}\left( \varphi _{3}\right)
+X_{1}^{-1}a_{p}\left( \varphi _{2}\right) X_{3}^{-1}+a_{p}\left( \varphi
_{1}\right) X_{2}^{-1}X_{3}^{-1} \\
&&-X_{1}^{-1}X_{2}^{-1}X_{3}^{-1}p^{2\underline{k}^{\ast }}
\end{eqnarray*}
\end{definition}

Recall the trilinear form $\left( \text{\ref{F trilinear alg}}\right) $.

\begin{proposition}
\label{p-stabilization P2}Suppose that $\varphi _{i}\in M\left( \mathbf{V}%
_{k_{i},F},\omega _{0,p,i}\right) ^{K^{\#}}$ are such that $\varphi _{i}\mid
T_{p}=a_{p}\left( \varphi _{i}\right) \varphi _{i}$ and let $\underline{%
\varphi }=\left( \varphi _{1},\varphi _{2},\varphi _{3}\right) $. Then,
setting $\alpha _{i}:=\alpha _{p}\left( \varphi _{i}\right) $, $\beta
_{i}^{\ast }:=\omega _{0,i}\left( \mathbf{p}^{\prime }\right) ^{-1}\beta
_{p}\left( \varphi _{i}\right) p^{-1}$ and $\underline{k}^{\ast }:=\frac{%
k_{1}+k_{2}+k_{3}}{2}$,%
\begin{eqnarray*}
&&t_{\underline{k}}\left( \varphi _{1}^{\left( \alpha _{1}\right) }\mid
W_{p},\varphi _{2}^{\left( \alpha _{2}\right) },\varphi _{3}^{\left( \alpha
_{3}\right) }\right) =-\alpha _{1}^{-1}p^{k_{1}}\mathcal{E}_{p}\left( 
\underline{\varphi },\beta _{1}^{\ast },\alpha _{2},\alpha _{3}\right) t_{%
\underline{k}}\left( \varphi _{1},\varphi _{2},\varphi _{3}\right) \text{,}
\\
&&t_{\underline{k}}\left( \varphi _{1}^{\left( \alpha _{1}\right) },\varphi
_{2}^{\left( \alpha _{2}\right) }\mid W_{p},\varphi _{3}^{\left( \alpha
_{3}\right) }\right) =-\alpha _{2}^{-1}p^{k_{2}}\mathcal{E}_{p}\left( 
\underline{\varphi },\alpha _{1},\beta _{2}^{\ast },\alpha _{3}\right) t_{%
\underline{k}}\left( \varphi _{1},\varphi _{2},\varphi _{3}\right) \text{,}
\\
&&t_{\underline{k}}\left( \varphi _{1}^{\left( \alpha _{1}\right) },\varphi
_{2}^{\left( \alpha _{2}\right) },\varphi _{3}^{\left( \alpha _{3}\right)
}\mid W_{p}\right) =-\alpha _{3}^{-1}p^{k_{3}}\mathcal{E}_{p}\left( 
\underline{\varphi },\alpha _{1},\alpha _{2},\beta _{3}^{\ast }\right) t_{%
\underline{k}}\left( \varphi _{1},\varphi _{2},\varphi _{3}\right) \text{.}
\end{eqnarray*}
\end{proposition}

\begin{proof}
We have, by definition and Proposition \ref{p-stabilization P1} $\left(
3\right) $,%
\begin{eqnarray*}
t_{\underline{k}}\left( \varphi _{1}^{\left( \alpha _{1}\right) },\varphi
_{2}^{\left( \alpha _{2}\right) },\varphi _{3}^{\left( \alpha _{3}\right)
}\mid W_{p}\right) &=&-t_{\underline{k}}\left( \varphi _{1}-\alpha
_{1}^{-1}\varphi _{1}^{\left( p\right) },\varphi _{2}-\alpha
_{2}^{-1}\varphi _{2}^{\left( p\right) },\alpha _{3}^{-1}p^{k_{3}}\varphi
_{3}-\varphi _{3}^{\left( p\right) }\right) \\
&=&A^{\left( 3\right) }-B^{\left( 3\right) }+C^{\left( 3\right) }-D^{\left(
3\right) }
\end{eqnarray*}%
where%
\begin{eqnarray*}
A^{\left( 3\right) } &=&\alpha _{3}^{-1}p^{k_{3}}t_{\underline{k}}\left(
\varphi _{1},\varphi _{2},\varphi _{3}\right) \text{,} \\
B^{\left( 3\right) } &=&\alpha _{1}^{-1}\alpha _{3}^{-1}p^{k_{3}}t_{%
\underline{k}}\left( \varphi _{1}^{\left( p\right) },\varphi _{2},\varphi
_{3}\right) +\alpha _{2}^{-1}\alpha _{3}^{-1}p^{k_{3}}t_{\underline{k}%
}\left( \varphi _{1},\varphi _{2}^{\left( p\right) },\varphi _{3}\right) +t_{%
\underline{k}}\left( \varphi _{1},\varphi _{2},\varphi _{3}^{\left( p\right)
}\right) \\
C^{\left( 3\right) } &=&\alpha _{1}^{-1}\alpha _{2}^{-1}\alpha
_{3}^{-1}p^{k_{3}}t_{\underline{k}}\left( \varphi _{1}^{\left( p\right)
},\varphi _{2}^{\left( p\right) },\varphi _{3}\right) +\alpha _{1}^{-1}t_{%
\underline{k}}\left( \varphi _{1}^{\left( p\right) },\varphi _{2},\varphi
_{3}^{\left( p\right) }\right) +\alpha _{2}^{-1}t_{\underline{k}}\left(
\varphi _{1},\varphi _{2}^{\left( p\right) },\varphi _{3}^{\left( p\right)
}\right) \\
D^{\left( 3\right) } &=&\alpha _{1}^{-1}\alpha _{2}^{-1}t_{\underline{k}%
}\left( \varphi _{1}^{\left( p\right) },\varphi _{2}^{\left( p\right)
},\varphi _{3}^{\left( p\right) }\right) \text{.}
\end{eqnarray*}

Regarding $t_{\underline{k}}$ as a pairing as we did in the proof of
Corollary \ref{C:afafsdsf} (for $t_{3,\underline{k}}$), we compute%
\begin{eqnarray*}
t_{\underline{k}}\left( \varphi _{1},\varphi _{2},\varphi _{3}^{\left(
p\right) }\right) &=&\left\langle \varphi _{1}\otimes \varphi _{2},\varphi
_{3}\mid K^{\#}\widehat{\pi }_{p}^{\iota }K\right\rangle _{t}=\left\langle
\varphi _{1}\otimes \varphi _{2},\varphi _{3}\mid T_{p}\right\rangle _{t} \\
&=&a_{p}\left( \varphi _{3}\right) \left\langle \varphi _{1}\otimes \varphi
_{2},\varphi _{3}\right\rangle _{t}=a_{p}\left( \varphi _{3}\right) t_{%
\underline{k}}\left( \varphi _{1},\varphi _{2},\varphi _{3}\right) \text{.}
\end{eqnarray*}%
Working in a similar way for the other first two terms of $B^{\left(
3\right) }$ we deduce (recall $\alpha _{3}p^{-k_{3}}=\omega _{0,3}\left( 
\mathbf{p}^{\prime }\right) \beta _{3}^{-1}p$)%
\begin{eqnarray*}
B^{\left( 3\right) } &=&\left\{ \alpha _{1}^{-1}a_{p}\left( \varphi
_{1}\right) \alpha _{3}^{-1}p^{k_{3}}+\alpha _{2}^{-1}a_{p}\left( \varphi
_{2}\right) \alpha _{3}^{-1}p^{k_{3}}+a_{p}\left( \varphi _{3}\right)
\right\} t_{\underline{k}}\left( \varphi _{1},\varphi _{2},\varphi
_{3}\right) \\
&=&\alpha _{3}^{-1}p^{k_{3}}\left\{ \alpha _{1}^{-1}a_{p}\left( \varphi
_{1}\right) +\alpha _{2}^{-1}a_{p}\left( \varphi _{2}\right) +a_{p}\left(
\varphi _{3}\right) \alpha _{3}p^{-k_{3}}\right\} t_{\underline{k}}\left(
\varphi _{1},\varphi _{2},\varphi _{3}\right) \\
&=&\alpha _{3}^{-1}p^{k_{3}}\left\{ \alpha _{1}^{-1}a_{p}\left( \varphi
_{1}\right) +\alpha _{2}^{-1}a_{p}\left( \varphi _{2}\right) +\omega
_{0,3}\left( \mathbf{p}^{\prime }\right) \beta _{3}^{-1}pa_{p}\left( \varphi
_{3}\right) \right\} t_{\underline{k}}\left( \varphi _{1},\varphi
_{2},\varphi _{3}\right) \text{.}
\end{eqnarray*}

Noticing that we have $K^{\#}\widehat{\pi }_{p}^{\iota }K=K^{\#}\widehat{\pi 
}_{p}^{\iota }$, we find%
\begin{equation*}
\varphi _{2}^{\left( p\right) }\otimes \varphi _{3}^{\left( p\right)
}=\left( \varphi _{2}\otimes \varphi _{3}\right) \mid K^{\#}\widehat{\pi }%
_{p}^{\iota }K\text{.}
\end{equation*}%
Hence we find, using the adjointness property of Proposition \ref%
{P:heckeadjointprop},%
\begin{eqnarray*}
t_{\underline{k}}\left( \varphi _{1},\varphi _{2}^{\left( p\right) },\varphi
_{3}^{\left( p\right) }\right) &=&\left\langle \varphi _{1},\varphi
_{2}\otimes \varphi _{3}\mid K^{\#}\widehat{\pi }_{p}^{\iota }K\right\rangle
_{t}=\left\langle \varphi _{1}\mid T_{p},\varphi _{2}\otimes \varphi
_{3}\right\rangle _{t} \\
&=&a_{p}\left( \varphi _{3}\right) \left\langle \varphi _{1},\varphi
_{2}\otimes \varphi _{3}\right\rangle _{t}=a_{p}\left( \varphi _{3}\right)
t_{\underline{k}}\left( \varphi _{1},\varphi _{2},\varphi _{3}\right) \text{.%
}
\end{eqnarray*}%
Working in a similar way for the other first two terms of $C^{\left(
3\right) }$ we deduce (recall $\alpha _{3}p^{-k_{3}}=\omega _{0,3}\left( 
\mathbf{p}^{\prime }\right) \beta _{3}^{-1}p$)%
\begin{eqnarray*}
C^{\left( 3\right) } &=&\left\{ \alpha _{1}^{-1}\alpha _{2}^{-1}a_{p}\left(
\varphi _{3}\right) \alpha _{3}^{-1}p^{k_{3}}+\alpha _{1}^{-1}a_{p}\left(
\varphi _{2}\right) +\alpha _{2}^{-1}a_{p}\left( \varphi _{1}\right)
\right\} t_{\underline{k}}\left( \varphi _{1},\varphi _{2},\varphi
_{3}\right) \\
&=&\alpha _{3}^{-1}p^{k_{3}}\left\{ \alpha _{1}^{-1}\alpha
_{2}^{-1}a_{p}\left( \varphi _{3}\right) +\alpha _{1}^{-1}a_{p}\left(
\varphi _{2}\right) \alpha _{3}p^{-k_{3}}+\alpha _{2}^{-1}a_{p}\left(
\varphi _{1}\right) \alpha _{3}p^{-k_{3}}\right\} t_{\underline{k}}\left(
\varphi _{1},\varphi _{2},\varphi _{3}\right) \\
&=&\alpha _{3}^{-1}p^{k_{3}}\left\{ \alpha _{1}^{-1}\alpha
_{2}^{-1}a_{p}\left( \varphi _{3}\right) +\alpha _{1}^{-1}\beta
_{3}^{-1}pa_{p}\left( \varphi _{2}\right) +\alpha _{2}^{-1}\omega
_{0,3}\left( \mathbf{p}^{\prime }\right) \beta _{3}^{-1}pa_{p}\left( \varphi
_{1}\right) \right\} t_{\underline{k}}\left( \varphi _{1},\varphi
_{2},\varphi _{3}\right) \text{.}
\end{eqnarray*}

Finally, once again using $K^{\#}\widehat{\pi }_{p}^{\iota }K=K^{\#}\widehat{%
\pi }_{p}^{\iota }$, we find%
\begin{equation*}
\varphi _{1}^{\left( p\right) }\otimes \varphi _{2}^{\left( p\right)
}\otimes \varphi _{3}^{\left( p\right) }=\varphi _{1}\otimes \varphi
_{2}\otimes \varphi _{3}\mid K^{\#}\widehat{\pi }_{p}^{\iota }
\end{equation*}%
and then $t_{\underline{k}}\left( \varphi _{1}^{\left( p\right) },\varphi
_{2}^{\left( p\right) },\varphi _{3}^{\left( p\right) }\right) =p^{2%
\underline{k}^{\ast }}t_{\underline{k}}\left( \varphi _{1},\varphi
_{2},\varphi _{3}\right) $. Hence we find%
\begin{eqnarray*}
D^{\left( 3\right) } &=&\alpha _{1}^{-1}\alpha _{2}^{-1}p^{2\underline{k}%
^{\ast }}t_{\underline{k}}\left( \varphi _{1},\varphi _{2},\varphi
_{3}\right) =\alpha _{3}^{-1}p^{k_{3}}\cdot \alpha _{1}^{-1}\alpha
_{2}^{-1}p^{2\underline{k}^{\ast }}\alpha _{3}p^{-k_{3}}t_{\underline{k}%
}\left( \varphi _{1},\varphi _{2},\varphi _{3}\right) \\
&=&\alpha _{3}^{-1}p^{k_{3}}\cdot \alpha _{1}^{-1}\alpha _{2}^{-1}\omega
_{0,k_{i}}\left( \mathbf{p}^{\prime }\right) \beta _{3}^{-1}p\cdot p^{2%
\underline{k}^{\ast }}t_{\underline{k}}\left( \varphi _{1},\varphi
_{2},\varphi _{3}\right) \text{.}
\end{eqnarray*}

Inserting these computations of $B^{\left( 3\right) }$, $C^{\left( 3\right)
} $ and $D^{\left( 3\right) }$ gives the third equation. The first two
equations are proved in a similar way.
\end{proof}

\subsection{Interpolation property of the $p$-adic trilinear form}

Recall our given $\underline{\mathbf{k}}=\left( \mathbf{k}_{1},\mathbf{k}%
_{2},\mathbf{k}_{3}\right) $ and consider the spaces $M_{p}(\mathcal{D}_{%
\mathbf{k}_{i}}(W),\omega _{0,p}^{\mathbf{k}_{i}})^{\Gamma _{0}\left( p%
\mathbb{Z}_{p}\right) }$, where $\omega _{0,p}^{\mathbf{k}_{i}}\left(
z\right) =\omega _{0,i}\left( \frac{z}{\mathrm{N}_{f}\left( z\right) }%
\right) \left( \frac{z}{\mathrm{N}_{f}\left( z\right) }\right) _{p}^{-%
\mathbf{k}_{i}}$ with $\omega _{0,i}$ taking values in $F$ and $\omega
_{0,1}\omega _{0,2}\omega _{0,3}=1$. We have specialization map attached to $%
\phi _{i}:\mathbf{k}_{i}\rightarrow k_{i}\in \mathbb{N}$:%
\begin{equation*}
\phi _{i,\ast }^{\mathrm{al}\text{\textrm{g}}}:M_{p}(\mathcal{D}_{\mathbf{k}%
_{i}}(W),\omega _{0,p}^{\mathbf{k}_{i}})^{\Gamma _{0}\left( p\mathbb{Z}%
_{p}\right) }\rightarrow M_{p}(\mathbf{V}_{k_{i},F},\omega
_{0,p}^{k_{i}})^{\Gamma _{0}\left( p\mathbb{Z}_{p}\right) }\text{,}
\end{equation*}%
where $\omega _{0,p}^{k_{i}}\left( z\right) =\omega _{0,i}\left( \frac{z}{%
\mathrm{N}_{f}\left( z\right) }\right) \left( \frac{z}{\mathrm{N}_{f}\left(
z\right) }\right) _{p}^{-k_{i}}$. The $U_{p}$-operator acts on these spaces
and the Ash-Stevens theory of \cite{ASdef} applies to show that there exists
a slope $\leq h\in \mathbb{R}$ decompositions and that $U_{p}$-eigenvectors
of slope $\leq h<k+1$ on $M_{p}(\mathbf{V}_{k_{i},F},\omega
_{0,p}^{k_{i}})^{\Gamma _{0}\left( p\mathbb{Z}_{p}\right) }$ lifts to
families belonging to $M_{p}(\mathcal{D}_{\mathbf{k}_{i}}(W),\omega _{0,p}^{%
\mathbf{k}_{i}})^{\Gamma _{0}\left( p\mathbb{Z}_{p}\right) }$ in an
essential unique way when $\phi _{i}:\mathbf{k}_{i}\rightarrow k_{i}$ is
obtained from $k_{i}\in U_{i}\subset \mathcal{X}_{\mathbb{Z}_{p}^{\times }}$
(see \cite{Ch} and \cite[Theorem 3.7]{Se}, for example). This motivates our
interest in this kind of spaces. If $\varphi _{i}\in M_{p}(\mathcal{D}_{%
\mathbf{k}_{i}}(W),\omega _{0,p}^{\mathbf{k}_{i}})$, we define%
\begin{equation*}
\mathcal{L}_{p,3}\left( \underline{\varphi }\right) :=t_{3,\underline{%
\mathbf{k}}}^{\circ }\left( \varphi _{1}\otimes \varphi _{2}\otimes \varphi
_{3}\mid W_{3}\right) \in \mathcal{O}_{\underline{\mathbf{k}}}\text{ where }%
\underline{\varphi }=\left( \varphi _{1},\varphi _{2},\varphi _{3}\right) 
\text{;}
\end{equation*}%
in a similar way we can define $\mathcal{L}_{p,1}$ and $\mathcal{L}_{p,2}$.
Set $\phi :=\phi _{1}\otimes \phi _{2}\otimes \phi _{3}$ and $\phi _{\ast }^{%
\mathrm{al}\text{\textrm{g}}}:=\phi _{1,\ast }^{\mathrm{al}\text{\textrm{g}}%
}\otimes \phi _{2,\ast }^{\mathrm{al}\text{\textrm{g}}}\otimes \phi _{3,\ast
}^{\mathrm{al}\text{\textrm{g}}}$. When $\varphi _{i}\in M_{p}(\mathcal{D}_{%
\mathbf{k}_{i}}(W),\omega _{0,p}^{\mathbf{k}_{i}})^{\Gamma _{0}\left( p%
\mathbb{Z}_{p}\right) }$ such that $\varphi _{i}\mid U_{p}=\alpha _{i}\left(
\varphi _{i}\right) \varphi _{i}$, setting $\varphi _{i,k_{i}}:=\phi
_{i,\ast }^{\mathrm{al}\text{\textrm{g}}}\left( \varphi _{i}\right) \in
M_{p}(\mathbf{V}_{k_{i},F},\omega _{0,p}^{k_{i}})^{\Gamma _{0}\left( p%
\mathbb{Z}_{p}\right) }$ we have $\varphi _{i,k_{i}}\mid U_{p}=\alpha
_{i}\left( \varphi _{i,k_{i}}\right) \varphi _{i,k_{i}}$ where $\alpha
_{i}\left( \varphi _{i,k_{i}}\right) =\phi _{i}\left( \alpha _{i}\left(
\varphi _{i}\right) \right) $ and, thanks to Proposition \ref%
{p-stabilization P1} $\left( 2\right) $, we have $\varphi _{i,k_{i}}=\varphi
_{i,k_{i}}^{\#,\alpha _{i}\left( \varphi _{i,k_{i}}\right) }$ for a uniquely
determined $\varphi _{i,k_{i}}^{\#}\in M_{p}(\mathbf{V}_{k_{i},F},\omega
_{0,p}^{k_{i}})^{\mathbf{GL}_{2}\left( \mathbb{Z}_{p}\right) }$. Let us
write $M_{p}(\mathcal{D}_{\mathbf{k}_{i}}(W),\omega _{0,p}^{\mathbf{k}%
_{i}})^{\Gamma _{0}\left( p\mathbb{Z}_{p}\right) ,\alpha _{i}}\subset M_{p}(%
\mathcal{D}_{\mathbf{k}_{i}}(W),\omega _{0,p}^{\mathbf{k}_{i}})^{\Gamma
_{0}\left( p\mathbb{Z}_{p}\right) }$ for the submodule of $U_{p}$%
-eigenvectors with eigenvalue $\alpha _{i}$ and set $\alpha _{i,k_{i}}:=\phi
_{i}\left( \alpha _{i}\right) $ and $\beta _{i,k_{i}}^{\ast }:=\alpha
_{i,k_{i}}^{-1}p^{k_{i}}$. If $\underline{\varphi }=\left( \varphi
_{1},\varphi _{2},\varphi _{3}\right) $ is such that $\varphi _{i}\in M_{p}(%
\mathcal{D}_{\mathbf{k}_{i}}(W),\omega _{0,p}^{\mathbf{k}_{i}})^{\Gamma
_{0}\left( p\mathbb{Z}_{p}\right) ,\alpha _{i}}$ and $\underline{\alpha }%
:=\left( \alpha _{1},\alpha _{2},\alpha _{3}\right) $, we set $\underline{%
\varphi }_{\underline{k}}^{\#}:=\left( \varphi _{1,k_{1}}^{\#},\varphi
_{2,k_{2}}^{\#},\varphi _{3,k_{3}}^{\#}\right) $ and then define%
\begin{eqnarray}
\mathcal{E}_{p,3}^{\left( 1\right) }\left( \underline{\varphi },\underline{k}%
\right)  &=&\mathcal{E}_{p,3}^{\left( 1\right) }\left( \underline{\alpha },%
\underline{k}\right) :=(1-\omega _{0,3}\left( \mathbf{p}^{\prime }\right) 
\frac{\alpha _{3,k_{3}}}{\alpha _{1,k_{1}}\alpha _{2,k_{2}}}p^{\underline{k}%
_{3}^{\ast }})\text{,}  \notag \\
\mathcal{E}_{p,3}^{\left( 2\right) }\left( \underline{\varphi },\underline{k}%
\right)  &=&\mathcal{E}_{p,3}^{\left( 2\right) }\left( \underline{\alpha },%
\underline{k}\right) :=-\alpha _{3,k_{3}}^{-1}p^{k_{3}}\mathcal{E}_{p}\left( 
\underline{\varphi }_{\underline{k}}^{\#},\alpha _{1,k_{1}},\alpha
_{2,k_{2}},\beta _{3,k_{3}}^{\ast }\right) \text{;}  \label{F Euler}
\end{eqnarray}%
we may similarly define the quantities $\mathcal{E}_{p,1}^{\left( 1\right) }$%
, $\mathcal{E}_{p,1}^{\left( 1\right) }$, $\mathcal{E}_{p,2}^{\left(
2\right) }$ and $\mathcal{E}_{p,2}^{\left( 2\right) }$. Let $Z_{i,\underline{%
\alpha }}$ be the set of those $\underline{k}$s such that $\mathcal{E}%
_{p,i}^{\left( j\right) }\left( \underline{\varphi },\underline{k}\right) =0$
for $j=1$ or $2$.

As a combination of Theorem \ref{Special value T Ichino}, $\left( \text{\ref%
{F Interpolation}}\right) $, Corollary \ref{C:afafsdsf} and Proposition \ref%
{p-stabilization P2} we get the following result.

\begin{theorem}
\label{T Main}For $i=1,2,3$, the trilinear form%
\begin{equation*}
\mathcal{L}_{p,i}:M_{p}(\mathcal{D}_{\mathbf{k}_{1}}(W),\omega _{0,p}^{%
\mathbf{k}_{1}})^{\Gamma _{0}\left( p\mathbb{Z}_{p}\right) }\otimes M_{p}(%
\mathcal{D}_{\mathbf{k}_{2}}(W),\omega _{0,p}^{\mathbf{k}_{2}})^{\Gamma
_{0}\left( p\mathbb{Z}_{p}\right) }\otimes M_{p}(\mathcal{D}_{\mathbf{k}%
_{3}}(W),\omega _{0,p}^{\mathbf{k}_{3}})^{\Gamma _{0}\left( p\mathbb{Z}%
_{p}\right) }\rightarrow \mathcal{O}_{\underline{\mathbf{k}}}
\end{equation*}%
has the property that, for every $\underline{\alpha }=\left( \alpha
_{1},\alpha _{2},\alpha _{3}\right) \neq \underline{0}$, the restriction $%
\mathcal{L}_{p,i}^{\underline{\alpha }}$\ of it to%
\begin{equation*}
M_{p}(\mathcal{D}_{\mathbf{k}_{1}}(W),\omega _{0,p}^{\mathbf{k}%
_{1}})^{\Gamma _{0}\left( p\mathbb{Z}_{p}\right) ,\alpha _{1}}\otimes M_{p}(%
\mathcal{D}_{\mathbf{k}_{2}}(W),\omega _{0,p}^{\mathbf{k}_{2}})^{\Gamma
_{0}\left( p\mathbb{Z}_{p}\right) ,\alpha _{2}}\otimes M_{p}(\mathcal{D}_{%
\mathbf{k}_{3}}(W),\omega _{0,p}^{\mathbf{k}_{3}})^{\Gamma _{0}\left( p%
\mathbb{Z}_{p}\right) ,\alpha _{3}}
\end{equation*}%
satisfies the following interpolation property.

For every $\phi :\underline{\mathbf{k}}\rightarrow \underline{k}$ with $%
\underline{k}\in U_{\underline{\mathbf{k}}}\cap \Sigma _{123}$ arithmetic
and balanced,%
\begin{equation}
\phi \circ \mathcal{L}_{i}=\mathcal{E}_{p,i}^{\left( 1\right) }\left( -,%
\underline{k}\right) \mathcal{E}_{p,i}^{\left( 2\right) }\left( -,\underline{%
k}\right) t_{\underline{k}}\left( -\right)  \label{T Main Claim 1}
\end{equation}%
and this interpolation property uniquely characterize $\mathcal{L}_{p,i}$
on\ when $\mathcal{O}_{\underline{\mathbf{k}}}$ is reduced.

For every $\underline{\varphi }$ in this space, let $\pi _{\underline{k}}$
be the automorphic representation attached to $\underline{\varphi }$ at an
arithmetic $\underline{k}$ and suppose that either $L\left( \pi _{\underline{%
k}},1/2\right) =0$ for every arithmetic $\underline{k}$ or there is some
arithmetic $\underline{k}^{0}\notin Z_{i,\underline{\alpha }}$ such that $%
L\left( \pi _{\underline{k}^{0}},1/2\right) \neq 0$ and that $B$ is the
quaternion algebra for $\pi _{\underline{k}^{0}}$\ predicted by \cite{Pr} at 
$\underline{k}^{0}$. Then there is a Zariski open subset $\phi \neq U_{%
\underline{\mathbf{k}}}^{\underline{\varphi }}\subset \mathrm{Sp}\left( 
\mathcal{O}_{\underline{\mathbf{k}}}\right) $ only depending on the $\mathbf{%
B}^{\times }\left( \mathbb{A}_{f}^{p}\right) $-representation $M_{\underline{%
\varphi }}$ generated by $\underline{\varphi }$ over $\mathcal{O}_{%
\underline{\mathbf{k}}}$ such that, for every $\phi :\underline{\mathbf{k}}%
\rightarrow \underline{k}$ with $\underline{k}\in U_{\underline{\mathbf{k}}%
}^{\underline{\varphi }}\cap \Sigma _{123}$ arithmetic and balanced,%
\begin{equation}
\mathcal{L}_{i}\left( \underline{\varphi }\right) \left( \underline{k}%
\right) ^{2}=\mathcal{E}_{p,i}^{\left( 1\right) }\left( \underline{\varphi },%
\underline{k}\right) ^{2}\mathcal{E}_{p,i}^{\left( 2\right) }\left( 
\underline{\varphi },\underline{k}\right) ^{2}\frac{1}{2^{3}m_{\mathbf{Z}_{%
\mathbf{B}}\mathbf{\backslash B},\infty }^{2}}\frac{\zeta _{\mathbb{Q}%
}^{2}\left( 2\right) L\left( 1/2,\pi _{\underline{k}}\right) }{L\left( 1,-,%
\mathrm{Ad}\right) }\prod\nolimits_{v}\alpha _{v}\left( \phi \left( 
\underline{\varphi }\right) \right) \text{ on }M_{\underline{\varphi }}\text{%
.}  \label{T Main Claim 2}
\end{equation}
\end{theorem}

\begin{proof}
Formula $\left( \text{\ref{T Main Claim 1}}\right) $ follows from $\left( 
\text{\ref{F Interpolation}}\right) $, Corollary \ref{C:afafsdsf} and
Proposition \ref{p-stabilization P2}. Suppose first $\underline{\mathbf{k}}$
is associated to the open affinoid $U_{\underline{\mathbf{k}}}:=\mathrm{Sp}%
\left( \mathcal{O}_{\underline{\mathbf{k}}}\right) \subset \mathcal{X}^{3}$.

Let $\pi _{\underline{k}}$ be the automorphic representation attached to $%
\underline{\varphi }$ and first assume that $L\left( \pi _{\underline{k}%
},1/2\right) =0$ for every arithmetic $\underline{k}$. Then, by the Jacquet
conjecture proved by Harris and Kudla (see \cite{HK}), the period integral
at all these $\underline{k}$s is zero for every $\underline{\varphi }%
^{\prime }\in M_{\underline{\varphi }}$; it follows from $\left( \text{\ref%
{T Main Claim 1}}\right) $ that $\mathcal{L}_{3}\left( \underline{\varphi }%
^{\prime }\right) \left( \underline{k}\right) =0$ for every arithmetic $%
\underline{k}$. By Zariski density of arithmetic points we deduce that $%
\mathcal{L}_{3}\left( \underline{\varphi }^{\prime }\right) \in \mathcal{O}_{%
\underline{\mathbf{k}}}$ is zero. Then $U_{\underline{\mathbf{k}}}^{%
\underline{\varphi }}:=U_{\underline{\mathbf{k}}}$ is such that $\left( 
\text{\ref{T Main Claim 2}}\right) $ is true at $\left( \underline{\varphi }%
^{\prime },\underline{k}\right) $ for every $\underline{\varphi }^{\prime
}\in M_{\underline{\varphi }}$\ and every arithmetic $\underline{k}\in U_{%
\underline{\mathbf{k}}}^{\underline{\varphi }}$.

Suppose now that there is some $\underline{k}^{0}$ as above. Then the right
hand side of $\left( \text{\ref{T Main Claim 1}}\right) $ is non-zero and,
hence, $\mathcal{L}_{3}\left( \underline{\varphi }^{0}\right) \left( 
\underline{k}^{0}\right) \neq 0$ for some $\underline{\varphi }^{0}\in M_{%
\underline{\varphi }}$. Since $U_{\underline{\mathbf{k}}}\subset \mathcal{X}%
^{3}$ is associated to a domain, there is a Zariski open subset $\phi \neq
U_{\underline{\mathbf{k}}}^{\underline{\varphi }^{0}}\subset U_{\underline{%
\mathbf{k}}}$ such that $\mathcal{L}_{3}\left( \underline{\varphi }%
^{0}\right) \left( \underline{\kappa }\right) \neq 0$ for every $\underline{%
\kappa }\in U_{\underline{\mathbf{k}}}^{\underline{\varphi }^{0}}$ and, in
particular, for every arithmetic $\underline{k}\in U_{\underline{\mathbf{k}}%
}^{\underline{\varphi }^{0}}$. But then we deduce from $\left( \text{\ref{T
Main Claim 1}}\right) $ that $Z_{i,\underline{\alpha }}\subset U_{\underline{%
\mathbf{k}}}-U_{\underline{\mathbf{k}}}^{\underline{\varphi }^{0}}$ and that
the period integral at all these $\underline{k}$s is non-zero. It follows
that Theorem \ref{Special value T Ichino} $\left( 1\right) $ is in force at
these $k$s: indeed the non-vanishing of the period integral implies the
non-vanishing of the product of the local forms at $\underline{k}$
(appearing in the definition of $C\left( \underline{\varphi }^{0},\underline{%
k}\right) $)\ and, hence, our $B$ is the quaternion algebra predicted by 
\cite{Pr} for all these $k$s. We deduce (once again applying $\left( \text{%
\ref{T Main Claim 1}}\right) $)\ that $\left( \text{\ref{T Main Claim 2}}%
\right) $ is true at $\left( \underline{\varphi }^{\prime },\underline{k}%
\right) $ for every $\underline{\varphi }^{\prime }\in M_{\underline{\varphi 
}}$\ and every arithmetic $\underline{k}\in U_{\underline{\mathbf{k}}}^{%
\underline{\varphi }}:=U_{\underline{\mathbf{k}}}^{\underline{\varphi }^{0}}$%
. Finally, an arbitrary $\underline{\mathbf{k}}$ is associated to a morphism 
$\mathrm{Sp}\left( \mathcal{O}_{\underline{\mathbf{k}}}\right) \rightarrow 
\mathcal{X}^{3}$ and we can pull-back.
\end{proof}

\bigskip

As explained in the introduction (see $\left( A\right) $ and $\left(
B\right) $), under the assumption that there is some arithmetic $\underline{k%
}^{0}\notin Z_{i,\underline{\alpha }}$ such that $L\left( \pi _{\underline{k}%
^{0}},1/2\right) \neq 0$ and $B$ is the quaternion algebra for $\pi _{%
\underline{k}^{0}}$\ predicted by \cite{Pr} at $\underline{k}^{0}$, there is
a test family $\underline{\varphi }=\left( \varphi _{1},\varphi _{2},\varphi
_{3}\right) $ on some open affinoid $U^{\prime }\subset U_{\underline{%
\mathbf{k}}}^{\underline{\varphi }}$ which corresponds, by the
Jacquet-Langlands correspondence (applied componentwisely), to a Coleman
family $\mathbf{f}=\left( \mathbf{f}_{1},\mathbf{f}_{2},\mathbf{f}%
_{3}\right) $ on $U^{\prime }$. Let $V_{\mathbf{f}_{i}}$ be the $p$-adic
representation attached to the Coleman family $\mathbf{f}_{i}$ and set $V_{%
\mathbf{f}}:=V_{\mathbf{f}_{1}}\otimes V_{\mathbf{f}_{2}}\otimes V_{\mathbf{f%
}_{3}}$. These three $p$-adic $L$-functions $\mathcal{L}_{i}\left( 
\underline{\varphi }\right) $ encodes information about the extended Selmer
group $\widetilde{H}_{f,\Sigma _{123}}^{1}\left( \mathbb{Q},V_{\mathbf{f}%
}\right) $ which interpolates the Selmer groups $\widetilde{H}_{f}^{1}\left( 
\mathbb{Q},V_{\mathbf{f}_{\underline{k}}^{\#}}\right) $ defined by the
Bloch-Kato conditions at every $\underline{k}\in \Sigma _{123}$. For
example, the vanishing of the Euler factors $\mathcal{E}_{p,i}^{\left(
1\right) }\left( \underline{\varphi },\underline{k}\right) $ is related to
the dimension of the Nekovar period space at $\underline{k}$ (see \cite{BSV}%
).


\begin{thebibliography}{99}
\bibitem{ASdef} A.\ Ash, G.\ Stevens, \emph{$p$-adic deformations of
arithmetic cohomology}. Submitted.

\bibitem{BSV} M. Bertolini, M. A. Seveso and R. Venerucci, \emph{On
exceptional zeros of triple product }$p$\emph{-adic }$L$\emph{-functions}.
In progress.

\bibitem{BoeSch} S.\ B\"{o}cherer and R.\ Schulze-Pillot, \emph{On central
critical values of triple product L-functions}. Number theory (Paris,
1994-1995) (D. Sinnou, ed.) Cambridge Univ. Press, Lond. Math. Soc. Lect.
Note Ser. \textbf{235} (1996), 1-46.

\bibitem{Ch} G. Chenevier, \emph{Familles }$p$\emph{-adiques de formes
automorphes pour }$\mathbf{GL}_{n}$, J. Reine Angew. Math. \textbf{570}
(2004), 143-217.

\bibitem{DR1} H.\ Darmon and V.\ Rotger, \emph{Diagonal cycles and Euler
systems I: A $p$-adic Gross-Zagier formula}, Ann. Scient. Ec. Norm. Sup., 4e
ser. \textbf{47} no. 4 (2014), 779-832.

\bibitem{DR2} H.\ Darmon and V.\ Rotger, \emph{Diagonal cycles and Euler
systems II: the Birch and Swinnerton-Dyer conjecture for Hasse-Weil-Artin
L-series}. Submitted.

\bibitem{GS} M. Greenberg and M. A. Seveso, $p$\emph{-families of modular
forms and }$p$\emph{-adic Abel-Jacobi maps}. Preprint.

\bibitem{GS2} M. Greenberg and M. A. Seveso, \emph{On the rationality of
period integrals and special value formulas in the compact case}. Preprint.

\bibitem{GrSt} R. Greenberg and G. Stevens, $p$\emph{-adic }$L$\emph{%
-functions and }$p$\emph{-adic periods of modular forms}, Invent.\ math.\ 
\textbf{111} (1993), 401-447.

\bibitem{HK} M.\ Harris and S.\ S.\ Kudla, \emph{The Central Critical Value
of a Triple Product $L$-Function}, Ann.\ of Math. (2)\ \textbf{133} no.\ 3
(1991), 605-672.

\bibitem{HT} M.\ Harris and J.\ Tilouine, $p$\emph{-adic measures and square
roots of special values of triple product }$L$\emph{-functions}, Math.
Annalen \textbf{320} (2001), 127-147.

\bibitem{Ich} A.\ Ichino, \emph{Trilinear forms and the central values of
triple product $L$-functions}, Duke Math. J. \textbf{145} no. 2 (2008),
281-307.

\bibitem{MTT} B. Mazur, J. Tate and J. Teitelbaum, \emph{On }$p$\emph{-adic
analogs of the conjectures of Birch and Swinnerton-Dyer}, Invent.\ math.\ 
\textbf{84} (1986), 1-48.

\bibitem{Nek} J. Nekov\'{a}r, \emph{Selmer complexes}, Ast\'{e}risque 
\textbf{310} (2006), viii+559.

\bibitem{Pot} J. Pottharst, \emph{Analytic families of finite slope Selmer
groups}, Algebr. Number Theory \textbf{7} no. 7 (2013), 1571-1612.

\bibitem{Pr} D.\ Prasad, \emph{Trilinear forms for representations of $\GL%
(2) $ and local $\epsilon $-factors}, Compos. Math. \textbf{75} Issue 1
(1990), 1-46.

\bibitem{Se} M. A. Seveso, \emph{Heegner cycles and derivatives of p-adic
L-functions}, J. Reine Angew. Math. \textbf{686} (2014), 111-148.

\bibitem{V1} R. Venerucci, $p$\emph{-adic regulators and }$p$\emph{-adic
families of modular forms}, Ph. D. thesis.

\bibitem{V2} R. Venerucci, \emph{Exceptional zero formulae and a conjecture
of Perrin-Riou}, to appear in Inventiones Mathematicae
\end{thebibliography}
\end{document}